\documentclass[onefignum,onetabnum]{siamart220329}

\usepackage{subfigure}
\usepackage{booktabs}
\usepackage{amsmath,amsxtra,amsfonts,amscd,amssymb,bm}
\usepackage{exscale}
\usepackage{geometry}
\usepackage{dcolumn}% Align table columns on decimal point
\usepackage[shortlabels]{enumitem}
\usepackage{multirow}

\newtheorem{remark}{Remark}
\newcommand\bs[1]{\boldsymbol{#1}}
\DeclareMathOperator\diag{diag}

\begin{document}
\title{Iterative methods of linearized moment equations for rarefied gases
\thanks{Submitted to arXiv July 10, 2024.
\funding{Zhenning Cai' s work was supported by the Academic Research Fund of the Ministry of Education of Singapore under grant
A-0008592-00-00.}
}}
\headers{Iterative methods of linearized moment equations}{Xiaoyu Dong and Zhenning Cai}

\author{
Xiaoyu Dong\thanks{Department of Mathematics, National University of Singapore, Singapore 119076 (\email{dongxiaoyu@lsec.cc.ac.cn}).}
\and
Zhenning Cai\thanks{Department of Mathematics, National University of Singapore, Singapore 119076 (\email{matcz@nus.edu.sg}).}
}

\maketitle

\begin{abstract}
We study the iterative methods for large moment systems derived from the linearized Boltzmann equation. By Fourier analysis, it is shown that the direct application of the block symmetric Gauss-Seidel (BSGS) method has slower convergence for smaller Knudsen numbers. Better convergence rates for dense flows are then achieved by coupling the BSGS method with the micro-macro decomposition, which treats the moment equations as a coupled system with a microscopic part and a macroscopic part. Since the macroscopic part contains only a small number of equations, it can be solved accurately during the iteration with a relatively small computational cost, which accelerates the overall iteration. The method is further generalized to the multiscale decomposition which splits the moment system into many subsystems with different orders of magnitude. Both one- and two-dimensional numerical tests are carried out to examine the performances of these methods. Possible issues regarding the efficiency and convergence are discussed in the conclusion.

\end{abstract}


%%================Introduction========================

\section{Introduction} \label{sec:intro}

Rarefied gas flows occur in various natural and engineering systems, including high-altitude atmosphere, space environments, vacuum systems and micro/nano-scale devices. When the Knudsen number, defined as the ratio of the molecular mean free path of a gas to a characteristic length scale of the flow, is small, classical continuum fluid models such as Euler and Navier-Stokes equations become inaccurate \cite{Chapman1990, Cercignani1988}. One needs to solve the Boltzmann equation, a basic kinetic model for rarefied gases, to obtain quantitatively correct flow structures.

One major difficulty in solving the Boltzmann equation comes from its high dimensionality. Compared with continuum models, the Boltzmann equation has an additional velocity variable that is usually three-dimensional. It doubles the number of dimensions and adds a huge amount of computational cost. Many simulations require supercomputers or GPUs to get accurate results \cite{Dimarco2018, Jaiswal2019}.
The moment method, proposed by Grad \cite{Grad1949}, is one major approach to the dimensionality reduction. Grad's moment method adds balance laws of non-equilibrium quantities to the conservation laws of mass, momentum and energy. With a few more variables added to the system, the results can be significantly improved \cite{Struchtrup2002, Baliti2020}. Grad's moment method suffers from difficulties including the loss of hyperbolicity \cite{Muller1993} and the convergence issue \cite{Cai2020}, but in the linear regime where the flow is close to a global equilibrium, Grad's method performs well and can be applied with many moments \cite{Sarna2018, Li2023}, and the convergence towards the Boltzmann equation has been proven in \cite{Sarna2020}. With moment methods, we expect that simulations of two-dimensional problems can be completed on personal computers within hours or even minutes \cite{Torrilhon2017, Liu2023}.

Numerical methods for moment methods including 13 or 26 moments have been studied extensively in the recent two decades. Since the first few moment equations are usually conservation laws, the finite volume method has been a popular approach adopted in many works including earlier papers such as \cite{Torrilhon2006, Gu2009}. Other strategies such as finite difference schemes and discontinuous Galerkin methods have also been studied in other works such as \cite{Rana2013} and \cite{Singh2024}. These works, as well as the references therein, have well established the general framework for most moment equations, and with slight changes, these methods can also be applied to equations with more variables. For example, the finite volume method is applied to large moment systems in \cite{Cai2010, Hu2019}. Another research direction related to numerical methods of moment equations is to deal with the nonlinearity, especially for moment equations derived based on the maximum entropy method \cite{Levermore1996}. For example, in \cite{Abdelmalik2016, Woude2024}, a nonlinear discontinuous Galerkin method is applied to solve moment equations that minimize an approximate entropy function defined by polynomials; in \cite{Boccelli2024}, the authors solve an interpolative approximation of the fourth-order maximum-entropy model using the finite volume method, focusing on the estimation of wave speeds; in \cite{Schneider2022}, the authors developed a numerical scheme that preserves the realizability of moments.

To solve steady-state moment equations, one can simply apply the time-marching scheme until the change of the solution is sufficiently small. Other methods that solve steady-state equations directly have also been studied, especially for linearized moment equations. In \cite{Torrilhon2017}, the method is extended to discontinuous Galerkin and the resulting linear system is solved by the PARDISO solver. Recently, finite element methods are also applied in the numerical solvers of moment equations, including the mixed finite element method for regularized 13-moment equations in \cite{Theisen2021}, and the method is generalized to larger moment equations in \cite{Christhuraj2024}. Another efficient approach to solving moment equations with a small number of moments is based on fundamental solutions \cite{Lockerby2016, Claydon2017, Himanshi2023}, which are often applied to small moment systems due to the requirement on explicit forms of fundamental solutions.

In this work, we will also study the numerical solver of linear steady-state moment equations. Instead of the spatial discretization, we will mainly focus on the iterative solver to get fast convergence rates. We will design our method such that it has the potential to be generalized to nonlinear equations. Since the moment method can be considered as the semidiscrete Boltzmann equation with the velocity variable discretized, ideas from the other numerical methods for the steady-state Boltzmann equation can be borrowed to design our numerical scheme. One classical iterative method to solve steady-state equations is the source iteration \cite{Adams2002}, also known as the conventional iterative scheme (CIS) \cite{Su2020Fast}, which has a slow convergence when the collision operator gets stiff in the case of small Knudsen numbers. By von Neumann analysis, it can be shown that the error grows in the form of $[1-O(\epsilon^2)]^n$ with $\epsilon$ being the Knudsen number and $n$ being the number of iterations. A recent improvement of the CIS is the general synthetic iterative scheme (GSIS) \cite{Su2020Can, Su2020Fast}, which is shown to have a uniform lower bound of the amplification factor for all Knudsen numbers $\epsilon$. This method is based on the micro-macro decomposition of distribution functions, and has been applied to many other cases \cite{Zhu2021, Zeng2023}. 

In our work, we will put forward a family of new iterative solvers based on the block symmetric Gauss-Seidel (BSGS) method. We will first show that the straightforward application of the first-order BSGS method has the convergence rate $[1-O(\epsilon)]^n$, which is better than the CIS. To improve its performance in dense regimes, we adopt an idea similar to GSIS by coupling the BSGS method with a micro-macro decomposition of the moments. Our scheme is further generalized to second-order methods and multiscale decomposition of the moments, which results in a family of iterative schemes for the steady-state moment equations. In the rest part of this paper, we will first introduce the moment equations and its spatial discretization in Section  \ref{sec:upwind_FVM_method}, and then study the BSGS method and its convergence in Section \ref{sec:BSGS_method}. Section \ref{sec:BSGS_MM_method} introduces the micro-macro decomposition to the iterative method to improve the convergence rate at small Knudsen numbers, and it is further generalized by replacing the micro-macro decomposition with the multiscale decomposition. Numerical tests are carried out in Section \ref{sec:example} to show the performances of these methods under different parameter settings. Finally, some concluding remarks are given in Section \ref{sec:conclu}.

%%================UpwindScheme========================

\section{Upwind finite volume method for the linearized Boltzmann equation} \label{sec:upwind_FVM_method}

In this paper, we consider the general linearized steady-state Boltzmann equation
\begin{equation} \label{eq:Boltz_eqn}
\bs{v} \cdot \nabla_{\bs{x}} f (\bs{x}, \bs{v}) = \frac{1}{\epsilon} \mathcal{L}[f](\bs{x}, \bs{v}), \quad \bs{x}, \bs{v} \in \mathbb{R}^d
\end{equation}
where $f(\bs{x}, \bs{v})$ is the distribution function of position $\bs{x} = (x_1,\ldots,x_d)^T$ and molecular velocity $\bs{v} = (v_1, \ldots, v_d)^T$. $\epsilon$ denotes Knudsen number and the operator $\mathcal{L}$ represents the linearized Boltzmann collision operator.

\subsection{Linearized Boltzmann equation and moment equations}

The distribution function $f$ in \eqref{eq:Boltz_eqn} is discretized by the spectral method, which is based on the expansion of $f$ into an infinite series: 
\begin{equation} \label{eq:inf_series}
f(\bs{x}, \bs{v}) = \sum_{n = 0}^{\infty} u^n (\bs{x}) \varphi_n(\bs{v}) \omega(\bs{v})
\end{equation}
where $\omega(\cdot)$ is the weight function, and $\varphi_n$ stands for the basis functions satisfying
\begin{equation*} 
\int \varphi^*_n(\bs{v}) \varphi_{n'}(\bs{v}) \omega(\bs{v}) \mathrm{d} \bs{v}
\left \{
\begin{array}{ll}
=0 & \text{if } n \neq n', \\
> 0 & \text{if } n = n'.
\end{array}
\right.
\end{equation*}
When the basis functions $\varphi_n(\cdot)$ are chosen as polynomials, the coefficients $u^n$ denote the moments of the distributions function $f$. The moment equations of $u^n$ can be derived as \cite{Grad1949}
\begin{equation} \label{eq:moment_eqn}
\sum_{i=1}^d \bs{A}_i \frac{\partial \bs{u}}{\partial x_i} = \frac{1}{\epsilon} \bs{L} \bs{u}
\end{equation}
by truncating the infinite series \eqref{eq:inf_series} to $N+1$ terms and choosing the test functions to be $\varphi_n$ with $n = 0,1,\dots,N$. Here the unknown vector $\bs{u}$ is
\begin{equation*}
\bs{u} = \left( u^0, u^1, \ldots, u^N \right)^T.
\end{equation*}
The coefficient matrices $\bs{A}_i$, whose elements are expressed as
\begin{equation*}
    \bs{A}_{i,mn} = \frac{\int v_i \varphi^*_m(\bs{v}) \varphi_{n}(\bs{v}) \omega(\bs{v}) \mathrm{d} \bs{v}}{\int |\varphi_{m}(\bs{v})|^2 \omega(\bs{v}) \mathrm{d} \bs{v}}, \quad m,n=0,1,\ldots,N
\end{equation*}
have real eigenvalues, and $\bs{L}$ is Hermitian negative semidefinite  due to the self-adjointness of the linear collision operator $\mathcal{L}$. When $N$ tends to infinity, the equations \eqref{eq:moment_eqn} are expected to converge to the linearized Boltzmann equation \eqref{eq:Boltz_eqn}.

\begin{remark}
In most moment methods, the weight function $\omega(\bs{v})$ is chosen as $C \exp(-|\bs{v}|^2/2)$ which is the Maxwellian, where $C$ is the normalizing constant. Then, by selecting different basis functions $\varphi_n$, the system \eqref{eq:moment_eqn} will correspond to different moment equations. Examples include Grad's moment equations \cite{Grad1949} and regularized moment equations \cite{Struchtrup2003}.
\end{remark}

\subsection{Spatial discretization} \label{sec:upwind_FVM}

As mentioned in the introduction, the finite volume method is a popular and classical approach to solving moment equations. Here we also adopt the finite volume method in our spatial discretization. In particular, the upwind method with linear reconstruction is applied to discretize the advection term. Such a method is commonly used in computational fluid dynamics, and we will detail the one-dimensional cases in the following part to facilitate our future discussion.

We assume that $\bs{u}$ is homogeneous in $x_2,\ldots, x_d$, and thus \eqref{eq:moment_eqn} can be simplified to
\begin{equation} \label{eq:1D_prob}
\bs{A} \frac{\mathrm{d} \bs{u}}{\mathrm{d} x} = \frac{1}{\epsilon} \bs{L} \bs{u},
\end{equation}
where $x$ and $\bs{A}$ refer to $x_1$ and $\bs{A}_1$, respectively.
For simplicity, we assume that the spatial grid is uniform, and the cell size is $\Delta x$. Thus, using $\bar{\bs{u}}_j$ to represent average $\bs{u}(x)$ on the $j$th grid cell, the upwind method can be formulated as
\begin{equation} \label{eq:upwind_FVM}
\bs{A}^+ (\bs{u}_{j+1/2}^- - \bs{u}_{j-1/2}^-) + \bs{A}^-(\bs{u}_{j+1/2}^+ - \bs{u}_{j-1/2}^+) 
= \frac{1}{\epsilon} \Delta x \bs{L} \bar{\bs{u}}_j,
\end{equation}
where $\bs{A}^{\pm}$ can be obtained by diagonalization of $\bs{A}$:
\begin{equation*}
\bs{A} = \bs{R} \bs{D} \bs{R}^{-1}, \qquad 
|\bs{A}| = \bs{R} |\bs{D}| \bs{R}^{-1}, \qquad \bs{A}^{\pm} = \frac{1}{2} (\bs{A} \pm |\bs{A}|)
\end{equation*}
with $\bs{D} = \diag(\lambda_0, \lambda_1, \ldots, \lambda_N)$ and $|\bs{D}| = \diag(|\lambda_0|, |\lambda_1|, \ldots, |\lambda_N|)$, and $\bs{u}_{j+1/2}^{\pm}$ are values of the numerical solution on the cell boundaries obtained by linear reconstruction. For first-order schemes,  $\bs{u}_{j+1/2}^- = \bar{\bs{u}}_j$ and $\bs{u}_{j+1/2}^+ = \bar{\bs{u}}_{j+1}$, so that the first-order upwind scheme turns out to be
\begin{equation} \label{eq:1D_first}
- \bs{A}^+ \bar{\bs{u}}_{j-1}
+ \left( |\bs{A}| - \frac{1}{\epsilon} \Delta x \bs{L} \right) \bar{\bs{u}}_j 
+ \bs{A}^- \bar{\bs{u}}_{j+1}
=\bs{0}. 
\end{equation}
To obtain second-order schemes, we apply linear reconstructions without limiters:
\begin{equation*} 
\bs{u}_{j+1/2}^- = \bar{\bs{u}}_j + \frac{\bar{\bs{u}}_{j+1} - \bar{\bs{u}}_{j-1}}{4}, \qquad
\bs{u}_{j+1/2}^+ = \bar{\bs{u}}_{j+1} + \frac{\bar{\bs{u}}_{j+2} - \bar{\bs{u}}_j}{4}, 
\end{equation*}
so that the numerical scheme reads
\begin{equation} \label{eq:1D_second}
\frac{1}{4} \bs{A}^+ \bar{\bs{u}}_{j-2}
- \left( \frac{1}{4} \bs{A} + \bs{A}^+ \right) \bar{\bs{u}}_{j-1}
+ \left( \frac{3}{4} |\bs{A}| - \frac{1}{\epsilon} \Delta x \bs{L} \right) \bar{\bs{u}}_j 
+ \left( \frac{1}{4} \bs{A} + \bs{A}^- \right) \bar{\bs{u}}_{j+1}
- \frac{1}{4} \bs{A}^- \bar{\bs{u}}_{j+2}
=\bs{0}. 
\end{equation}
For two-dimensional problems, there are similar expressions which are presented in Section $1$ of Supplementary Materials.

The spatial discretization gives rise to a large linear system to be solved numerically. As mentioned in \Cref{sec:intro}, the conventional iterative scheme has a slow convergence rate when $\epsilon$ is small. In particular, the number of iterations is expected to be proportional to $\epsilon^{-2}$ when $\epsilon$ is small. Our approach to breaking this constraint will be introduced in the next section.

%%================BSGS========================

\section{Block symmetric Gauss-Seidel method to solve the linearized Boltzmann equation} \label{sec:BSGS_method}

We take the one-dimensional problem \eqref{eq:1D_prob} as an example to introduce our solver for the discrete equations \eqref{eq:1D_first} and \eqref{eq:1D_second}. Briefly speaking, the block symmetric Gauss-Seidel (BSGS) method will be applied to solve $\bs{\bar{u}}_j$ through the upwind schemes \eqref{eq:1D_first} and \eqref{eq:1D_second}, and both first- and second-order schemes will be presented in this section. 

\subsection{Block symmetric Gauss-Seidel method for the first-order scheme} \label{sec:SGS_first} 

According to the first-order upwind scheme \eqref{eq:1D_first}, the moments $\bs{\bar{u}}_j$, $j=1,\ldots,M$ satisfy the linear equations 
\begin{equation} \label{eq:first_eqn}
\begin{pmatrix}
\bs{H}+\bs{H}_1 & \bs{A}^- & & & \\
-\bs{A}^+ & \bs{H} & \bs{A}^- & & \\
& \ddots & \ddots & \ddots & \\
& & -\bs{A}^+ & \bs{H} & \bs{A}^- \\
& & & -\bs{A}^+ & \bs{H}+\bs{H}_M
\end{pmatrix}
\begin{pmatrix}
\bs{\bar{u}}_1 \\
\bs{\bar{u}}_2 \\
\vdots \\
\bs{\bar{u}}_{M-1} \\
\bs{\bar{u}}_M 
\end{pmatrix}
= 
\begin{pmatrix}
\bs{g}_1 \\
\bs{0} \\
\vdots \\
\bs{0} \\
\bs{g}_M 
\end{pmatrix}
.
\end{equation}
where $\bs{H} = |\bs{A}| - \Delta x \bs{L} / \epsilon$, and $\bs{H}_1, \bs{H}_M, \bs{g}_1$ and  $\bs{g}_M$ are related to the boundary conditions. For bounded problems with wall boundary conditions, an additional condition specifying the total mass in the computational domain is needed to uniquely determine the solution. Let $\bar{u}_j^0$ be the first component of $\bar{\bs{u}}_j$. Then the condition can be written as
\begin{equation} \label{eq:normalization}
{\Delta x} \sum_{j=1}^M \bar{u}_j^0 = C,
\end{equation} 
where the total mass $C$ is given.
Since the operator $|v| - \Delta x \mathcal{L}/\epsilon$ is positive definite, the matrix $\bs{H}$ is invertible in most cases. This allows us to apply the block symmetric Gauss-Seidel method to \eqref{eq:first_eqn}. 
In general, the block symmetric Gauss-Seidel method is unable to maintain the equality \eqref{eq:normalization}. A normalization is then applied after each iteration to recover the property \eqref{eq:normalization}. In detail, the arbitrary initial values $\bs{\bar{u}}_j^{(0)}$, $j=1,\ldots,M$ satisfying \eqref{eq:normalization} are given firstly and for every iterative step $n$, we calculate $\bs{\bar{u}}_j^{(n+1/2)}$, $j=1,2,\ldots,M$ sequentially by
\begin{equation} \label{eq:BSGS_forward}
- \bs{A}^+ \bar{\bs{u}}_{j-1}^{(n+1/2)}
+ \left( |\bs{A}| - \frac{1}{\epsilon} \Delta x \bs{L} \right) \bar{\bs{u}}_j^{(n+1/2)} 
+ \bs{A}^- \bar{\bs{u}}_{j+1}^{(n)}
=\bs{0}
\end{equation}
where $\bar{\bs{u}}_0^{(n+1/2)}$ and $\bar{\bs{u}}_{M+1}^{(n)}$ are associated with boundary conditions, then calculate $\bs{\bar{u}}_j^{(n+1)}$, $j=1,2,\ldots,M$ in reverse order, that is to say, $j=M,M-1,\ldots,2,1$, by
\begin{equation} \label{eq:BSGS_backward}
- \bs{A}^+ \bar{\bs{u}}_{j-1}^{(n+1/2)}
+ \left( |\bs{A}| - \frac{1}{\epsilon} \Delta x \bs{L} \right) \bar{\bs{u}}_j^{(n+1)} 
+ \bs{A}^- \bar{\bs{u}}_{j+1}^{(n+1)}
=\bs{0}.
\end{equation}
Note that one cannot omit the step calculating $\bs{\bar{u}}_j^{(n+1)}$, $j=1,2,\ldots,M$ in reverse order which performs the backward Gauss-Seidel iteration since information propagates in both directions due to the hyperbolic nature of the equation. Below, this algorithm will be called the BSGS method for short. It can be shown by local Fourier analysis that when $\epsilon$ approaches zero, the convergence rate of this iterative method has a lower bound depending only on $\Delta x$, which means the method does not suffer significant slowdown in this asymptotic limit. In fact, the spectral radius of the iteration matrix can be estimated as $\lambda_0 + \epsilon \lambda_1$ with $\lambda_0 < 1$ and $\lambda_1 < 0$, meaning that the convergence does worsen when $\epsilon$ gets smaller, but the number of iterations won't blow up in this limit. Note that the dependence on $\Delta x$ is a property of the symmetric Gauss-Seidel method for general boundary value problems. The details of the analysis can be found in Appendix \ref{appen:BSGS}. 

\subsection{Block symmetric successive relaxation method for the second-order scheme} \label{sec:SGS_second}

For the second-order scheme \eqref{eq:1D_second}, a direct application of the BSGS method may cause divergence (see Figure \ref{fig:diverge} in our numerical tests). One possible reason is that the reconstruction destroys the structure similar to diagonally dominant matrices, which guarantees the convergence of the Gauss-Seidel method. We therefore tweak the original method by adding a diagonal term $\alpha|\bs{A}|(\bar{\bs{u}}_j^{(n+1/2)} - \bar{\bs{u}}_j^{(n)})$ (or $\alpha|\bs{A}|(\bar{\bs{u}}_j^{(n+1)} - \bar{\bs{u}}_j^{(n+1/2)})$ during the backward iteration), so that the iteration becomes
\begin{gather*}
\frac{1}{4} \bs{A}^+ \bar{\bs{u}}_{j-2}^{(n+1/2)}
- \left( \frac{1}{4} \bs{A} + \bs{A}^+ \right) \bar{\bs{u}}_{j-1}^{(n+1/2)}
+ \bs{S}_{\alpha} \bar{\bs{u}}_j^{(n+1/2)} 
+ \left( \frac{1}{4} \bs{A} + \bs{A}^- \right) \bar{\bs{u}}_{j+1}^{(n)}
- \frac{1}{4} \bs{A}^- \bar{\bs{u}}_{j+2}^{(n)}
=\alpha|\bs{A}| \bar{\bs{u}}_j^{(n)}, \\
\label{eq:SGS_second_2}
\frac{1}{4} \bs{A}^+ \bar{\bs{u}}_{j-2}^{(n+1/2)}
- \left( \frac{1}{4} \bs{A} + \bs{A}^+ \right) \bar{\bs{u}}_{j-1}^{(n+1/2)}
+ \bs{S}_{\alpha} \bar{\bs{u}}_j^{(n+1)} 
+ \left( \frac{1}{4} \bs{A} + \bs{A}^- \right) \bar{\bs{u}}_{j+1}^{(n+1)}
- \frac{1}{4} \bs{A}^- \bar{\bs{u}}_{j+2}^{(n+1)}
=\alpha|\bs{A}|\bar{\bs{u}}_j^{(n+1/2)}, 
\end{gather*}
where
\begin{equation*}
\bs{S}_{\alpha} =\frac{3}{4} |\bs{A}| - \frac{1}{\epsilon} \Delta x \bs{L} + \alpha|\bs{A}|.
\end{equation*}
These equations are to replace steps (2) and (3) in the algorithm introduced in Section \ref{sec:SGS_first}. In most of our test cases, we choose $\alpha = 1/4$, which is sufficient to recover the convergence of the iteration. For conciseness, we call this method the block symmetric successive relaxation (BSSR) method.

The convergence analysis of the BSSR method is carried out in Appendix \ref{appen:BSSR}, where it shows by local Fourier analysis that the method converges for all nonnegative $\alpha$. Note that our analysis based on the Fourier method works only for problems with periodic boundary conditions. In practice, when other types of boundary conditions, such as wall boundary conditions, are applied, the convergence is not guaranteed. Such cases will surface in our numerical tests in Section \ref{sec:example}. This explains why we need to choose a positive $\alpha$ to enhance the stability of our iterations.

%%================BSGS-MM========================

\section{Coupling with the micro-macro decomposition} \label{sec:BSGS_MM_method}

Recall that the analysis of the BSGS method shows us that the number of iterations does increase when the Knudsen number gets smaller. To improve the convergence rate in the case of low Knudsen numbers, we will revamp our iterative method using the micro-macro decomposition \cite{Bennoune2008}. Below we will mainly study the coupling of the first-order BSGS method and the micro-macro decomposition for the one-dimensional Boltzmann equation 
\begin{equation} \label{eq:1D_Boltz}
v \frac{\partial f}{\partial x}(x,v)  = \frac{1}{\epsilon} \mathcal{L}[f](x,v).
\end{equation}
where the kernel of the linearized collision operator is
$
\operatorname{ker} \mathcal{L} = \operatorname{span} \{ \bs{\Phi}(v) \omega(v) \}
$
where
\begin{equation} \label{eq:macro_base_phi}
\bs{\Phi}(v) = (1,v,v^2-1), \quad \omega(v) = \frac{1}{\sqrt{2\pi}} \exp \left( -\frac{v^2}{2}\right).
\end{equation}
The idea can be naturally generalized to the multi-dimensional case and the BSSR method.

\subsection{BSGS method with micro-macro decomposition} \label{sec:BSGS_MM}
To implement the micro-macro decomposition, we set the weight function $\omega(v)$ to be the Maxwellian \eqref{eq:macro_base_phi} and choose basis functions such that $\big(\varphi_0(v), \varphi_1(v), \varphi_2(v) \big) = \bs{\Phi}(v)$. Thus, the coefficients $\bs{u}$ can be split into equilibrium and non-equilibrium variables:
\begin{displaymath}
\bs{u}_*^1 = (u^0, u^1, u^2)^T, \qquad \bs{u}_*^2 = (u^3, \cdots, u^N)^T.
\end{displaymath}
Thus, the linear system \eqref{eq:1D_prob} can be written as
\begin{equation} \label{eq:micro_macro}
\bs{A}_{11}^* \frac{\mathrm{d} \bs{u}_*^1}{\mathrm{d} x} + \bs{A}_{12}^* \frac{\mathrm{d} \bs{u}_*^2}{\mathrm{d} x} = \bs{0}, \qquad
\bs{A}_{21}^* \frac{\mathrm{d} \bs{u}_*^1}{\mathrm{d} x} + \bs{A}_{22}^* \frac{\mathrm{d} \bs{u}_*^2}{\mathrm{d} x} = \frac{1}{\epsilon} \bs{L}_{22}^* \bs{u}_*^2
\end{equation}
where $\bs{L}_{22}^*$ is an $(N-2)\times(N-2)$ matrix representing the collision term, and 
\begin{displaymath}
\bs{A}_{11}^* = \begin{pmatrix}
  0 & 1 & 0 \\ 1 & 0 & 2 \\ 0 & 1 & 0
\end{pmatrix}.
\end{displaymath}
The idea of the micro-macro decomposition is to solve the two equations in \eqref{eq:micro_macro} alternately. We first assume $\bs{u}_*^2$ is given and solve $\bs{u}_*^1$ from the first equation, and then plug the result into the second equation to solve $\bs{u}_*^2$, which completes one iteration. However, for fixed $\bs{u}_*^2$, the first equation of \eqref{eq:micro_macro} may not admit a solution since for $\bs{v} = (1,0,-1)^T$,
\begin{displaymath}
  \bs{v}^T \left( \bs{A}^*_{11} \frac{\mathrm{d} \bs{u}_*^1}{\mathrm{d} x} + \bs{A}^*_{12} \frac{\mathrm{d} \bs{u}_*^2}{\mathrm{d} x}  \right) = \bs{v}^T \bs{A}^*_{12} \frac{\mathrm{d} \bs{u}_*^2}{\mathrm{d} x},
\end{displaymath}
which may not be zero for the given $\bs{u}_*^2$. Our solution is to include a few more coefficients in the first part to guarantee the solvability of both equations. In general, we assume
\begin{equation} \label{eq:u_twopart}
\bs{u}^1 = \left( u^0, u^1, u^2, \ldots, u^{N_0} \right)^T, \qquad 
\bs{u}^2 = \left( u^{N_0+1}, \ldots, u^N \right)^T, 
\end{equation}
which satisfy the linear system
\begin{equation} \label{eq:general_MM}
\bs{A}_{11} \frac{\mathrm{d} \bs{u}^1}{\mathrm{d} x} + \bs{A}_{12} \frac{\mathrm{d} \bs{u}^2}{\mathrm{d} x}= \frac{1}{\epsilon} (\bs{L}_{11} \bs{u}^1 + \bs{L}_{12} \bs{u}^2), \quad
\bs{A}_{21} \frac{\mathrm{d} \bs{u}^1}{\mathrm{d} x} + \bs{A}_{22} \frac{\mathrm{d} \bs{u}^2}{\mathrm{d} x}= \frac{1}{\epsilon} (\bs{L}_{21} \bs{u}^1 + \bs{L}_{22} \bs{u}^2).
\end{equation}
In practice, choosing $N_0 = 3$ usually suffices to allow iterating between $\bs{u}^1$ and $\bs{u}^2$ smoothly.

In general, it is unnecessary to solve $\bs{u}^1$ and $\bs{u}^2$ exactly during the inner iteration. Since the equation of $\bs{u}^1$ contains only a small number of unknowns, below we will assume that the first equation of \eqref{eq:general_MM} is accurately solved but only one BSGS iteration is applied to the second equation. We call this approach the BSGS-MM (BSGS with micro-macro decomposition) method. To apply it to the upwind scheme \eqref{eq:upwind_FVM}, we define the following block structures for the matrices $\bs{A}^{\pm}$ and $|\bs{A}|$:
\begin{equation*}
\bs{A}^{\pm} = 
\begin{pmatrix}
\bs{A}_{11}^{\pm} & \bs{A}_{12}^{\pm } \\
\bs{A}_{21}^{\pm} & \bs{A}_{22}^{\pm }
\end{pmatrix}, \qquad 
|\bs{A}| = 
\begin{pmatrix}
|\bs{A}_{11}| & |\bs{A}_{12}| \\
|\bs{A}_{21}| & |\bs{A}_{22}|
\end{pmatrix}.
\end{equation*}
Note that here $\bs{A}_{ij}^{\pm}$ is defined as a block of $\bs{A}^{\pm}$, which may not be the positive/negative part of $\bs{A}_{ij}$ based on the eigendecomposition. The detail of the BSGS-MM method is given below: 
\begin{enumerate}[(1)]
\item Set the error tolerance $tol$ and the maximum number $N_s$ iterations. Set $n=0$ and prepare the initial values $\bs{\bar{u}}_j^{1,(0)}$ and $\bs{\bar{u}}_j^{2,(0)}$ for $j=1,\ldots,M$ satisfying \eqref{eq:normalization}.

\item Solve $\bs{\bar{u}}_j^{1,(n+1)}$ for all $j=1,2,\ldots,M$ by 
\begin{equation} \label{eq:Euler}
\begin{split}
& - \bs{A}_{11}^+ \bar{\bs{u}}_{j-1}^{1,(n+1)}
+ \left( |\bs{A}_{11}| - \frac{1}{\epsilon} \Delta x \bs{L}_{11} \right) \bar{\bs{u}}_j^{1,(n+1)}
+ \bs{A}_{11}^- \bar{\bs{u}}_{j+1}^{1,(n+1)} \\
& \hspace{8em} =
\bs{A}_{12}^+ \bar{\bs{u}}_{j-1}^{2,(n)}
- \left( |\bs{A}_{12}| - \frac{1}{\epsilon} \Delta x \bs{L}_{12} \right) \bar{\bs{u}}_j^{2,(n)} 
- \bs{A}_{12}^- \bar{\bs{u}}_{j+1}^{2,(n)}, 
\end{split}
\end{equation}
where $\bar{\bs{u}}_0^{\cdot,(\cdot)}$ and $\bar{\bs{u}}_{M+1}^{\cdot,(\cdot)}$ are based on the boundary conditions. Use the normalizing condition \eqref{eq:normalization} if necessary.
 
\item Calculate $\bs{\bar{u}}_{j}^{2,(n+1/2)}$, $j=1,2,\ldots,M$ sequentially by
\begin{equation} \label{eq:BSGS_MM_forward}
\begin{split}
& - \bs{A}_{21}^+ \bar{\bs{u}}_{j-1}^{1,(n+1)}
+ \left( |\bs{A}_{21}| - \frac{1}{\epsilon} \Delta x \bs{L}_{21} \right) \bar{\bs{u}}_j^{1,(n+1)}
+ \bs{A}_{21}^- \bar{\bs{u}}_{j+1}^{1,(n+1)} \\
& \hspace{8em} =
\bs{A}_{22}^+ \bar{\bs{u}}_{j-1}^{2,(n+1/2)}
- \left( |\bs{A}_{22}| - \frac{1}{\epsilon} \Delta x \bs{L}_{22} \right) \bar{\bs{u}}_j^{2,(n+1/2)} 
- \bs{A}_{22}^- \bar{\bs{u}}_{j+1}^{2,(n)}
\end{split}
\end{equation}
where $\bar{\bs{u}}_0^{\cdot,(\cdot)}$ and $\bar{\bs{u}}_{M+1}^{\cdot,(\cdot)}$ are based on the boundary conditions.

\item Calculate $\bs{\bar{u}}_j^{2,(n+1)}$, $j=1,2,\ldots,M$ in reverse order, that is to say, $j=M,M-1,\ldots,2,1$ by
\begin{equation} \label{eq:BSGS_MM_backward}
\begin{split}
& - \bs{A}_{21}^+ \bar{\bs{u}}_{j-1}^{1,(n+1)}
+ \left( |\bs{A}_{21}| - \frac{1}{\epsilon} \Delta x \bs{L}_{21} \right) \bar{\bs{u}}_j^{1,(n+1)}
+ \bs{A}_{21}^- \bar{\bs{u}}_{j+1}^{1,(n+1)} \\
& \hspace{8em} =
\bs{A}_{22}^+ \bar{\bs{u}}_{j-1}^{2,(n+1/2)}
- \left( |\bs{A}_{22}| - \frac{1}{\epsilon} \Delta x \bs{L}_{22} \right) \bar{\bs{u}}_j^{2,(n+1)} 
- \bs{A}_{22}^- \bar{\bs{u}}_{j+1}^{2,(n+1)}
\end{split}
\end{equation}
where $\bar{\bs{u}}_0^{\cdot,(\cdot)}$ and $\bar{\bs{u}}_{M+1}^{\cdot,(\cdot)}$ are based on the boundary conditions.

\item Compute the residual of the equations \eqref{eq:first_eqn} for $\bs{\bar{u}}_j = \bs{\bar{u}}_j^{(n+1)}$, $j=1,\ldots,M$. Stop if the norm of the residual is less than $tol$. Otherwise, check if $n>N_s$. If not, increase $n$ by $1$ and return to step (2). Otherwise, print warning messages about the failure of convergence and stop.

\end{enumerate}
In step (2), the linear system \eqref{eq:Euler} can be solved either directly or iteratively, depending on the size of the problem. When \eqref{eq:Euler} is solved iteratively, the tolerance of the residual should be slightly lower than $tol$ to ensure that the final stopping criterion in step (5) can be achieved. Extension to the second-order BSSR method is straightforward. One just needs to add the correction terms from the linear reconstruction to all the equations in steps (2)(3) and (4), and add the relaxation term to the equations in steps (3) and (4). This will be referred to as BSSR-MM method hereafter. We can also get the three-dimensional BSGS-MM or BSSR-MM method by including all the five conservative quantities as well as a few non-equilibrium variables in $\bs{u}_1$ to guarantee the solvability of the equations.

Intuitively, when $\epsilon$ is small, the conservative variables included in $\bs{u}^1$ contains major information of the distribution function, so that solving \eqref{eq:Euler} exactly already provides a good approximation of the exact solution. This can be confirmed by the Fourier analysis in Appendix \ref{appen:BSGS_MM}. However, such analysis will also show that the convergence of the BSGS-MM method is slow when $\epsilon$ is large. This issue will be tackled in the following subsection.

\subsection{Variations of the BSGS-MM method}

In this section, we consider approaches that can achieve fast convergence for both small and large values of $\epsilon$. Again, we will focus only on the improvements of the BSGS-MM method, and the extension to the BSSR-MM method is again straightforward.

\subsubsection{The Hybrid BSGS-MM method} \label{sec:coupling}

It has been shown that the original BSGS method and the BSGS-MM method converge fast in exactly opposite regimes. Therefore, it is natural to consider hybridizing the two methods to cover the entire range of $\epsilon$. Here we propose to add a $N_b$ BSGS iterations after each BSGS-MM iteration. For simplicity, we use $\mathcal{T}_{\text{BSGS}}$ and $\mathcal{T}_{\text{BSGS-MM}}$ to denote the iteration operators for the two schemes, or more specifically, the equation $\bar{\bs{u}}^{(n+1)} = \mathcal{T}_{\text{BSGS}} \bar{\bs{u}}^{(n)}$ refers to \eqref{eq:BSGS_forward} and the reverse sweeping \eqref{eq:BSGS_backward}, and $\bar{\bs{u}}^{(n+1)} = \mathcal{T}_{\text{BSGS-MM}} \bar{\bs{u}}^{(n)}$ refers to \eqref{eq:Euler}\eqref{eq:BSGS_MM_forward}\eqref{eq:BSGS_MM_backward}. Then the Hybrid BSGS-MM method can be written as
\begin{equation} \label{eq:HybridBSGSMM}
\bar{\bs{u}}^{(n+1)} = \mathcal{T}_{\text{BSGS-MM}} (\mathcal{T}_{\text{BSGS}})^{N_b} \bar{\bs{u}}^{(n)}.
\end{equation}
For large $\epsilon$, we tend to choose a larger $N_b$ to gain better convergence. Note that when $\epsilon$ is not spatially homogeneous, this scheme also allows to have different $N_b$ for different spatial grids. We will study this method in our future works.

\subsubsection{The multiscale BSGS (BSGS-MS) method} \label{sec:advanced}

In equations \eqref{eq:general_MM}, the moment equations are divided into two sets since $\bs{u}^1$  and $\bs{u}^2$ have different orders of magnitude with respect to $\epsilon$. This idea can be generalized if $\bs{u}^2$ can be further separated to different magnitudes. Analysis on the magnitude of moments has been studied in several works including \cite{Struchtrup2004, Cai2012, Struchtrup2013, Cai2023}. These analyses may allow us to split $\bs{u}$ into
\begin{equation} \label{eq:u_decomp}
\bs{u} = \begin{pmatrix} \bs{u}^1 \\ \bs{u}^2 \\ \vdots  \\ \bs{u}^K \end{pmatrix},
\end{equation}
so that $\bs{u}^1, \ldots, \bs{u}^K$ have increasing orders of magnitude with respect to $\epsilon$.
Thus, the moment equations \eqref{eq:1D_prob} can be rewritten in the following form:
\begin{equation*} 
\left\{ \begin{aligned}
&\bs{A}_{11} \frac{\mathrm{d} \bs{u}^1}{\mathrm{d} x} + \sum_{k=2}^{K}\bs{A}_{1 k} \frac{\mathrm{d} \bs{u}^k}{\mathrm{d} x}= \frac{1}{\epsilon} \left( \bs{L}_{11} \bs{u}^1 + \sum_{k=2}^{K} \bs{L}_{1 k} \bs{u}^k \right), \\
&\bs{A}_{j1} \frac{\mathrm{d} \bs{u}^1}{\mathrm{d} x} + \sum_{k=2}^{K}\bs{A}_{j k} \frac{\mathrm{d} \bs{u}^k}{\mathrm{d} x}= \frac{1}{\epsilon} \left( \bs{L}_{j1} \bs{u}^1 + \sum_{k=2}^{K} \bs{L}_{j k} \bs{u}^k \right), \quad j=2,\cdots,K.
\end{aligned} \right.
\end{equation*}
where $\bs{A}_{jk}$, $j,k = 1,\cdots,K$ are submatrices of $\bs{A}$ determined by
\begin{equation*}
\bs{A} = 
\begin{pmatrix}
\bs{A}_{11} & \bs{A}_{21} & \cdots & \bs{A}_{K 1} \\
\bs{A}_{12} & \bs{A}_{22} & \cdots & \bs{A}_{K 2} \\
\vdots & \vdots & \ddots & \vdots \\
\bs{A}_{1 K} & \bs{A}_{2 K} & \cdots & \bs{A}_{K K}
\end{pmatrix}.
\end{equation*}
Here $\bs{u}^1$ can be chosen as the same vector as in \eqref{eq:u_twopart} to guarantee the solvability of the first equation.  The BSGS-MS method solves these equations one after another, and the procedures of the algorithm are listed in Section $2$ of Supplementary Materials.

Again, when $N$ is large, the number of components in $\bs{u}^1$ is significantly less than $N$, leading to lower computational cost compared with the BSGS method. Another benefit of the BSGS-MS method is that the linear systems to be solved are smaller than both BSGS and BSGS-MM method.

\subsubsection{The Hybrid BSGS-MS method} \label{sec:hybrid_BSGS_MS}
To obtain better convergence for large $\epsilon$, we can also combine the BSGS-MS method and the direct BSGS scans, leading to the ``Hybrid BSGS-MS method''. Such a method can be denoted by
\begin{displaymath}
\bar{\bs{u}}^{(n+1)} = \mathcal{T}_{\text{BSGS-MS}} (\mathcal{T}_{\text{BSGS}})^{N_b} \bar{\bs{u}}^{(n)},
\end{displaymath}
where $\mathcal{T}_{\text{BSGS-MS}}$ is the iteration operator for the BSGS-MS method.

%%================Examples========================

\section{Numerical examples} \label{sec:example}

In this section, we are going to carry out some numerical tests to demonstrate the performances of the BSGS and BSGS-MM methods for various Knudsen numbers $\epsilon$. For the first order finite volume method \eqref{eq:1D_first}, the following five methods will be tested:
\begin{itemize}
\item The BSGS method (section \ref{sec:SGS_first}),
\item The BSGS-MM method (section \ref{sec:BSGS_MM}),
\item The Hybrid BSGS-MM method with $N_b$ BSGS steps per iteration, denoted as ``Hybrid BSGS-MM-$N_b$'' below (section \ref{sec:coupling}),
\item The BSGS-MS method (section \ref{sec:advanced})
\item The Hybrid BSGS-MS method with $N_b$ BSGS steps per iteration, denoted as ``Hybrid BSGS-MS-$N_b$'' below (section \ref{sec:hybrid_BSGS_MS}).
\end{itemize}
For the second-order scheme \eqref{eq:1D_second}, we will also test the five methods mentioned above, with ``BSGS'' replaced with ``BSSR'', referring to a relaxation to improve the convergence (see \ref{sec:SGS_second}). In all tests below, the relaxation parameter $\alpha$ is chosen as $1/4$.

\subsection{One-dimensional tests: heat transfer} \label{sec:1D_exm}

We first consider a toy model in which both the spatial and velocity variables are one-dimensional. Below we will introduce the problem settings before showing our numerical results.

\subsubsection{Problem setting}

The Boltzmann equation can thus be written as \eqref{eq:1D_Boltz}, and we study the BGK collision operator defined by
\begin{equation*}
\mathcal{L}[f](x,v) = \mathcal{M}[f](x,v) - f(x,v), \quad
\mathcal{M}[f](v) = \left[ \int_{\mathbb{R}} \bs{\Phi}(v) f(v) \,\mathrm{d}v \right] \begin{pmatrix} 1 \\ v \\ (v^2-1) / 2\end{pmatrix} \omega(v).
\end{equation*}
The basis functions $\varphi_n$ in \eqref{eq:inf_series} are chosen as Hermite polynomials
\begin{equation*}
\varphi_n(v) = H_n(v) := (-1)^{n} \exp \left( \frac{v^2}{2} \right) \frac{\mathrm{d}^n}{\mathrm{d} v^n} \exp \left( -\frac{v^2}{2} \right), \qquad \forall n = 0,1,\ldots 
\end{equation*}
As a result, the moment equations become a system of ordinary differential equations \eqref{eq:1D_prob} with
\begin{equation} \label{eq:1D_coeff}
\bs{u} =
\begin{pmatrix}
u^0 \\
u^1 \\
\vdots \\
u^N
\end{pmatrix}, 
\quad
\bs{A} =
\begin{pmatrix}
0 & 1 & & & & \\
1 & 0 & 2 & & & \\
& 1 & 0 & 3 & & \\
& & \ddots & \ddots & \ddots & \\
& & & 1 & 0 & N \\
& & & & 1 & 0 \\
\end{pmatrix},
\quad
\bs{L} =
\begin{pmatrix}
0 & & & & & \\
& 0 & & & & \\
& & 0 & & & \\
& & & -1 & & \\
& & & & \ddots & \\
& & & & & -1 \\
\end{pmatrix}.
\end{equation}
The components of $\bs{u}$ are related to the macroscopic flow quantities by
\begin{equation*}
\rho(x) = u^0(x) , \quad
U(x) = u^1(x), \quad
T(x) = 2u^2(x),
\end{equation*}
where $\rho$ is the fluid density, $U$ is the flow velocity, and $T$ is the temperature of the fluid.

For the one-dimensional heat transfer problem, we assume that the spatial domain is $(0,1)$, and the boundary conditions are given by the fully diffusive model based on wall temperatures:
\begin{equation} \label{eq:1D_bound}
f(0,v) = \left[ \rho_w^0 + \frac{T_w^0}{2} (v^2-1)\right] \omega(v), \quad v>0; \qquad
f(1,v) = \left[ \rho_w^1 + \frac{T_w^1}{2} (v^2-1)\right] \omega(v), \quad v<0.
\end{equation}
Here $T_w^0$ and $T_w^1$ are the wall temperatures at $x = 0$ and $x = 1$, respectively, and $\rho_w^0$ and $\rho_w^1$ are determined by the zero-velocity condition on the boundary:
\begin{displaymath}
\int_\mathbb{R} v f(0,v) \,\mathrm{d}v =\int_\mathbb{R} v f(1,v) \,\mathrm{d}v = 0.
\end{displaymath}
According to \cite{Grad1949}, the boundary conditions of moment equations \eqref{eq:1D_prob} are derived by taking odd moments of \eqref{eq:1D_bound}. In our implementation, the formulation of boundary conditions follows the approach in \cite{Sarna2017, Cai2023} to satisfy the $L^2$ stability, and the result has the form
\begin{equation*}
\bs{u}^{\mathrm{odd}} (0) = \bs{B} \bs{u}^{\mathrm{even}}(0) + \bs{g} T_w^{0}, \qquad
\bs{u}^{\mathrm{odd}} (1) = - \bs{B} \bs{u}^{\mathrm{even}}(1) + \bs{g} T_w^{1},
\end{equation*}
where $\bs{u}^{\mathrm{odd}} = (u^1, u^3, \ldots, u^{N_o})^T$ and $\bs{u}^{\mathrm{even}} = (u^0, u^2, \ldots, u^{N_e})^T$ with $N_o = N_e+1 = N$ if $N$ is odd and $N_o+1 = N_e = N$ if $N$ is even. The details of the matrix $\bs{B}$ and the vector $\bs{g}$ can be found in Section $3$ of Supplementary Material. Numerically, it is implemented by setting $\bs{u}_{1/2}^-$ and $\bs{u}_{M+1/2}^+$ in \eqref{eq:upwind_FVM} according to
\begin{equation} 
\label{eq:bc}
\left \{
\begin{aligned}
& \bs{u}_{1/2}^{-,\mathrm{even}} = \bs{u}_{1/2}^{+,\mathrm{even}}, \\
& \bs{u}_{1/2}^{-,\mathrm{odd}} = 2 \left( \bs{B} \bs{u}_{1/2}^{+,\mathrm{even}} + \bs{g} T_w^0 \right) - \bs{u}_{1/2}^{+,\mathrm{odd}};
\end{aligned}
\right. 
\,
\left \{
\begin{aligned}
& \bs{u}_{M+1/2}^{+,\mathrm{odd}} = \bs{u}_{M+1/2}^{-,\mathrm{odd}}, \\
& \bs{u}_{M+1/2}^{+,\mathrm{odd}} = -2 \left( \bs{B} \bs{u}_{M+1/2}^{-,\mathrm{even}} + \bs{g} T_w^1 \right) - \bs{u}_{M+1/2}^{-,\mathrm{odd}}.
\end{aligned}
\right.
\end{equation}
In the first-order scheme, the right-hand sides are given by $\bs{u}_{1/2}^+ = \bs{u}_1$ and $\bs{u}_{M+1/2}^- = \bs{u}_M$. In the second-order scheme, the linear reconstruction on boundary cells is set to be
\begin{equation*} 
\bs{u}_{1/2}^{+} = \frac{3\bar{\bs{u}}_{1} - \bar{\bs{u}}_{2}}{2}, \quad
\bs{u}_{3/2}^{-} = \frac{\bar{\bs{u}}_{1} + \bar{\bs{u}}_{2}}{2};
\qquad
\bs{u}_{M-1/2}^{+} = 
\frac{\bar{\bs{u}}_{M-1} + \bar{\bs{u}}_{M}}{2}, \quad
\bs{u}_{M+1/2}^{-} =
\frac{3\bar{\bs{u}}_{M} - \bar{\bs{u}}_{M-1}}{2},
\end{equation*}
which are also applied to the boundary conditions \eqref{eq:bc}. To uniquely determine the solution, the total mass needs to be specified:
\begin{equation*} 
\int_{0}^{1} \rho(x) \,\mathrm{d}x = 1, \quad \text{or} \quad
\int_{0}^{1} u^0(x) \,\mathrm{d}x = 1.
\end{equation*}
In all our numerical tests, we choose $T_w^0 = 0$ and $T_w^1 = 1$.

\subsubsection{Numerical results}
To verify the convergence of our method, we select $N=5$ for which the exact solution $\bs{u}^{\mathrm{exact}}(x)$ of \eqref{eq:1D_prob} can be obtained analytically. The iteration is terminated if the residual has a norm less than $10^{-10}$. The density and temperature results for $M = 40$ are given in Figure \ref{fig:N5_solution}, from which one sees that both the first-order scheme \eqref{eq:1D_first} and the second-order scheme \eqref{eq:1D_second} provide results that are indistinguishable from the exact solution. We therefore test the order of convergence by checking the following $L^2$ error of the solution:
\begin{equation*}
    err = \sqrt{\frac{1}{M} \sum_{j=1}^M \left\| \bar{\bs{u}}_j - \frac{1}{\Delta x} \int_{x_j-\Delta x/2}^{x_j+\Delta x/2} \bs{u}^{\mathrm{exact}}(x) \,\mathrm{d}x \right\|_2^2 }.
\end{equation*}
The tests are carried out for $\epsilon = 10^{-k}$  with $k = 0,1,2,3$ and $M = 2^j \cdot 80$ with $j = 0,1,\cdots,5$, and the errors of both schemes are displayed in Figure \ref{fig:1D_order}. It is clear that the desired order of convergence can be achieved in all cases, and the second-order scheme has a significantly lower numerical error. When $\epsilon$ is smaller, the $L^2$ error is larger, which is possibly due to the larger total variance in the exact solution as indicated in Figure \ref{fig:N5_solution}.

\begin{figure}[!ht] \label{fig:N5_solution}
    \centering
    \subfigure[Density]{
    \includegraphics[width=0.4\textwidth, trim=35 20 55 35, clip]{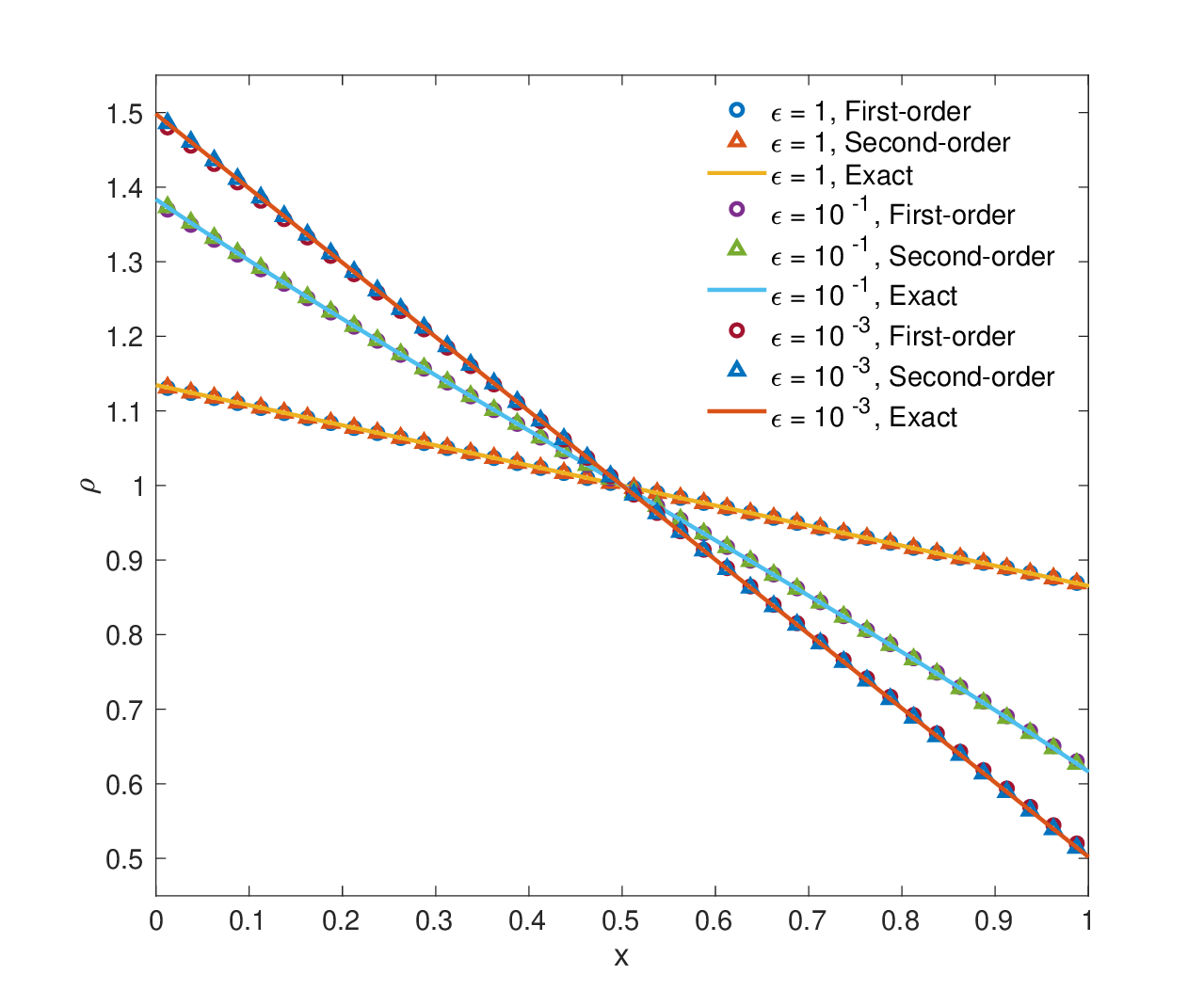}
    } \quad
    \subfigure[Temperature]{
    \includegraphics[width=0.4\textwidth, trim=35 20 55 35, clip]{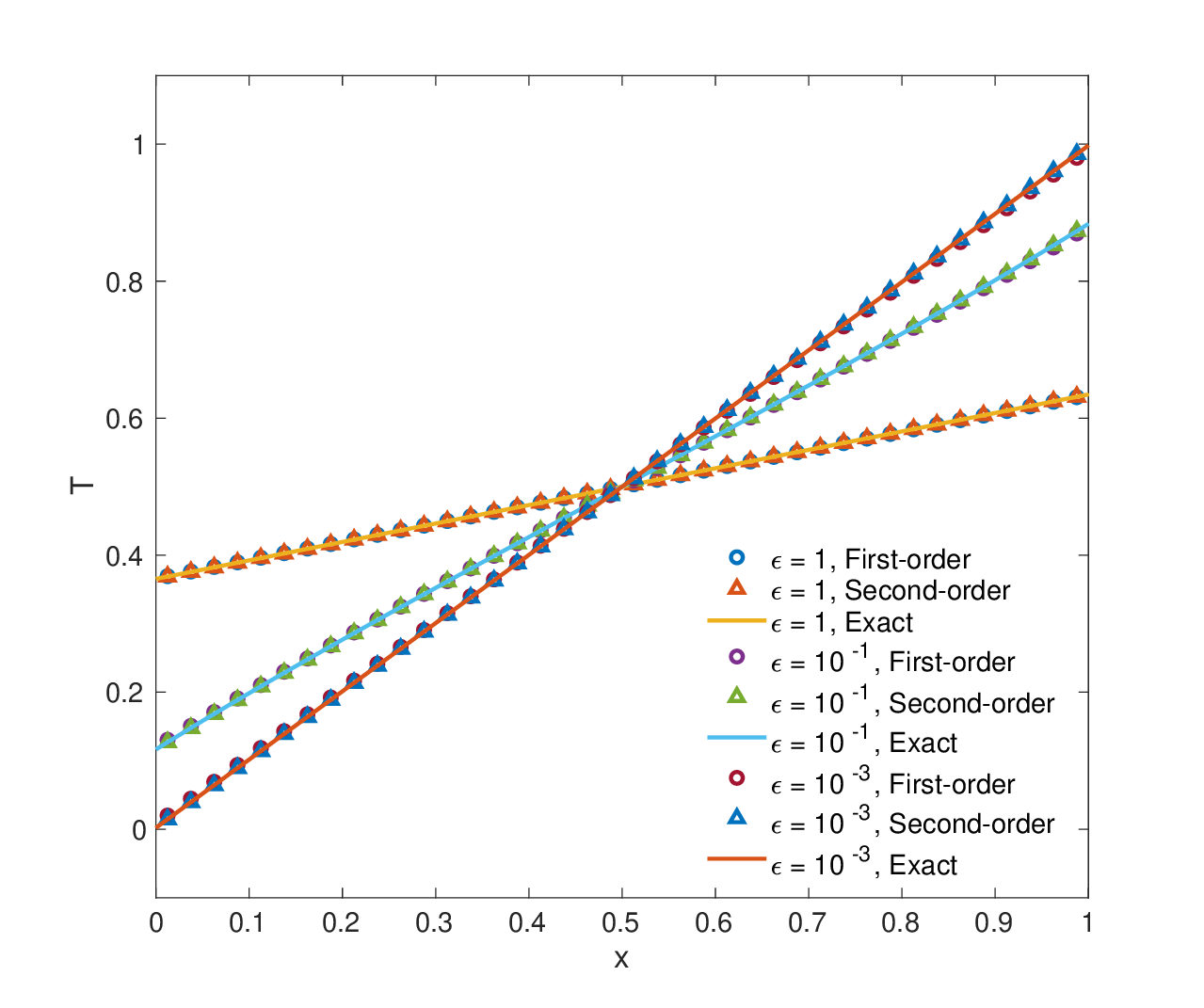}
    }
    \caption{The distributions of density and temperature with $N=5$.}
\end{figure}

\begin{figure}[!ht] \label{fig:1D_order}
    \centering
    \subfigure[First-order scheme]{
    \includegraphics[width=0.4\textwidth, trim=5 0 38 15, clip]{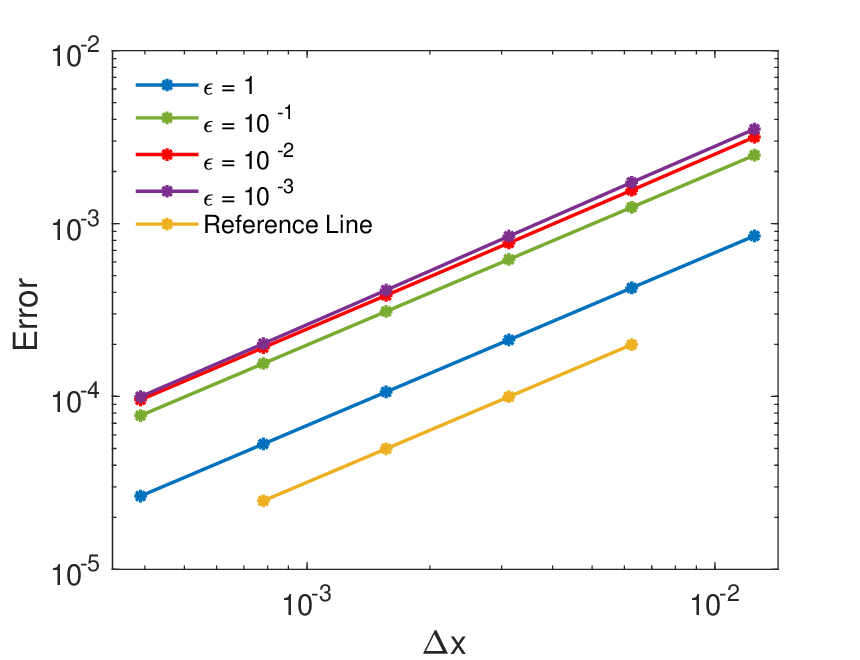}
    } \quad
    \subfigure[Second-order scheme]{
    \includegraphics[width=0.42\textwidth, trim=2 0 38 20, clip]{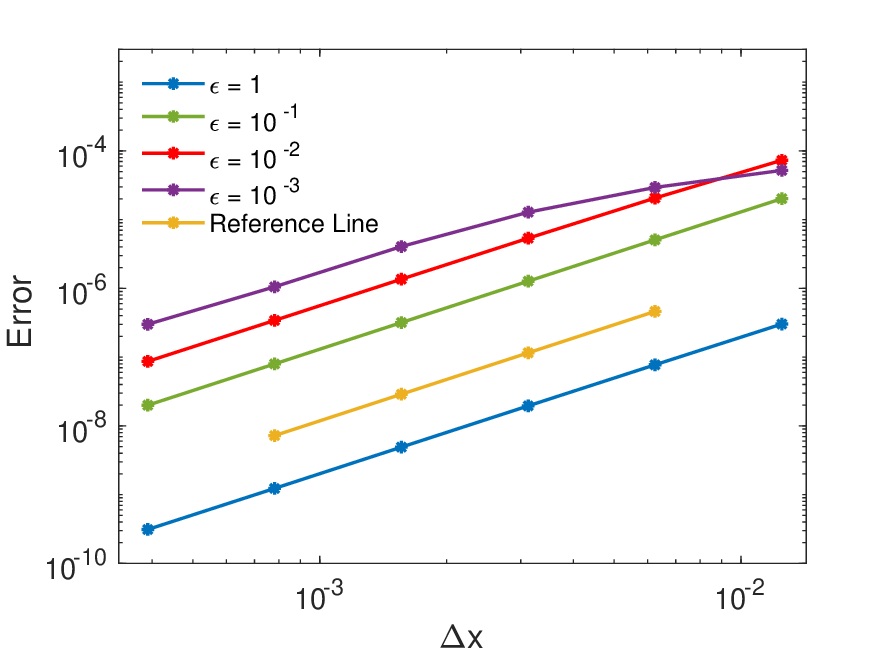}
    }
    \caption{The errors of upwind finite volume schemes for various $\epsilon$.}
\end{figure}

We now focus on the efficiency of different numerical methods. A larger moment system with $N = 16$, discretized on a uniform grid with $\Delta x = 5 \times 10^{-3}$, is considered in all the simulations below.  The distributions of density $\rho(x)$ and temperature $T(x)$ calculated by second-order scheme \eqref{eq:1D_first} are performed in figure \ref{fig:1D_macro_quan}.
\begin{figure}[!ht] \label{fig:1D_macro_quan}
    \centering
    \subfigure[Density]{
    \includegraphics[width=0.4\textwidth, trim=15 0 40 22, clip]{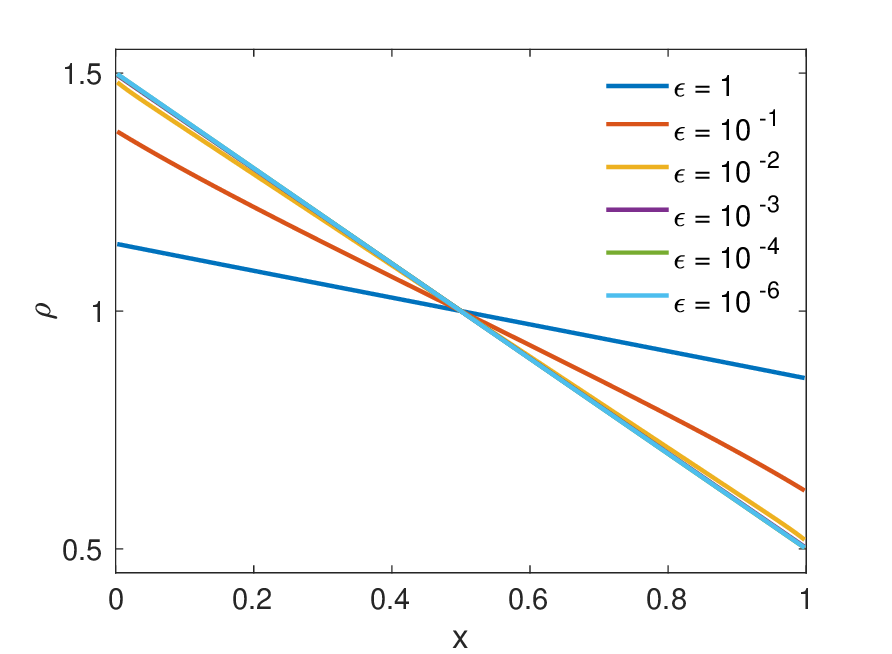}
    } \quad
    \subfigure[Temperature]{
    \includegraphics[width=0.41\textwidth, trim=10 0 40 22, clip]{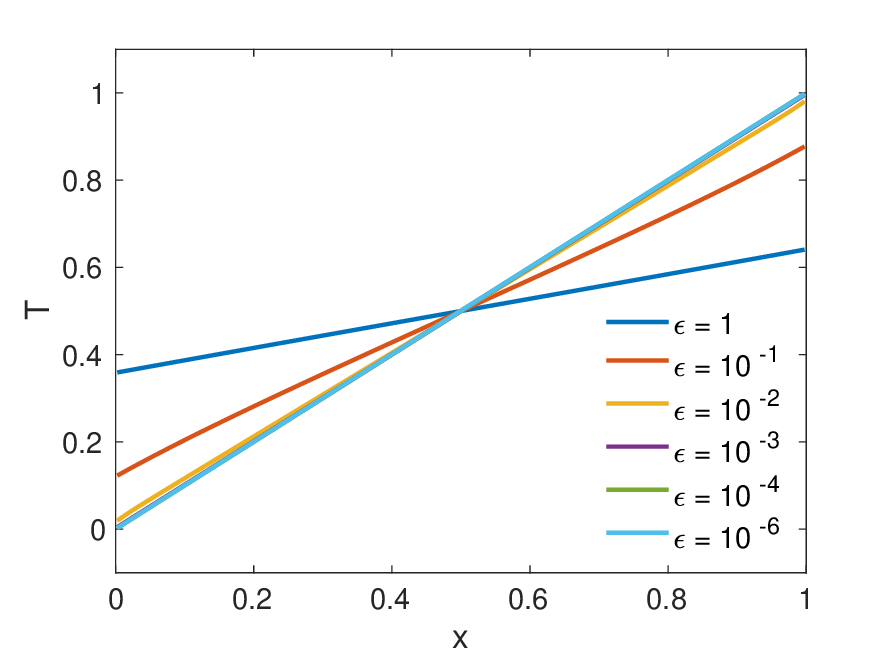}
    }
    \caption{The distributions of density and temperature with $N=16$.}
\end{figure}
We can see that the distributions of density and temperature almost  coincide for $\epsilon = 10^{-4}$ and $10^{-6}$. Totally, when $\epsilon$ is less, they are closer and the values near the boundary are also closer to the boundary values, which means when $\epsilon$ is small, there exists the boundary layer in the system. Also, the numerical results verifies the correctness of our algorithm.

Beginning with first-order method, we show in Figure \ref{fig:BSGS_BSGSMM_rate} how the residual decreases in the BSGS method and the BSGS-MM methods where $\bs{u}^1 = (u^0, u^1, u^2, u^3)$ in \eqref{eq:u_twopart}. For the BSGS method (Figures \ref{fig:BSGS_rate} and \ref{fig:BSGS_rate_magnified}), the trend of the curves agrees with the general behavior of the symmetric Gauss-Seidel method: the residual decreases quickly in the first few steps due to the fast elimination of high-frequency errors, and then the convergence slows down and the rate is governed by the lowest-frequency part of the error. More iterations are needed for smaller $\epsilon$, but it can be expected that the convergence rate has a lower bound,  as predicted in Appendix \ref{appen:BSGS}. On the contrary, the BSGS-MM method (Figure \ref{fig:BSGSMM_rate}) converges faster as $\epsilon$ decreases, which is in line with our theoretical results in Appendix \ref{appen:BSGS_MM}. When $\epsilon \leqslant 10^{-4}$, only one iteration is needed to meet our stopping criterion, since $\bs{u}^2$ in \eqref{eq:u_twopart} actually has magnitude $O(\epsilon^2)$ and has little impact on the error of the solution.

\begin{figure}[!ht] \label{fig:BSGS_BSGSMM_rate}
    \centering
    \subfigure[BSGS method]{
    \label{fig:BSGS_rate}
    \includegraphics[width=0.415\textwidth, trim=5 0 25 22, clip]{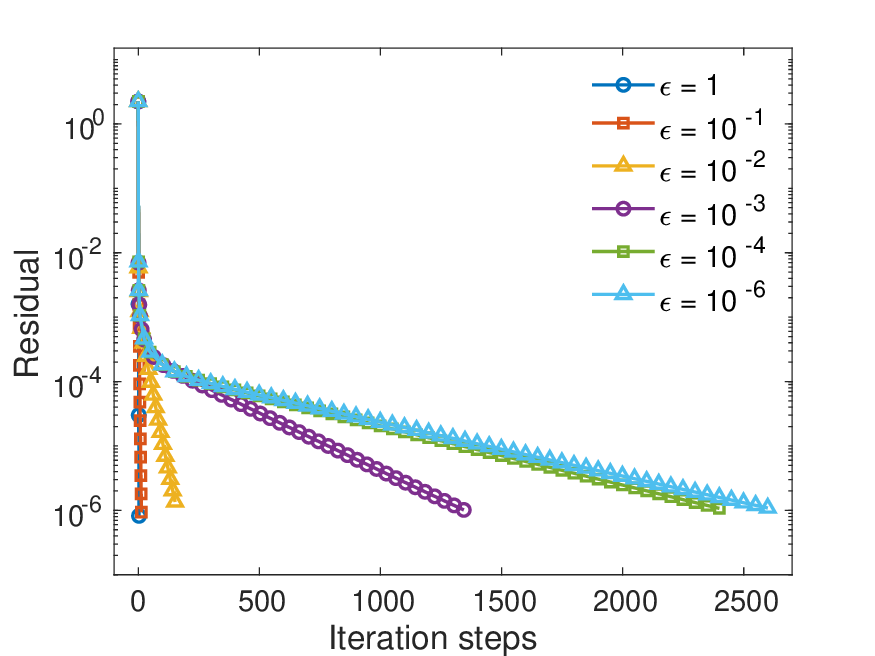}
    } \quad
    \subfigure[First 50 steps of BSGS]{
    \label{fig:BSGS_rate_magnified}
    \includegraphics[width=0.4\textwidth, trim=5 0 40 22, clip]{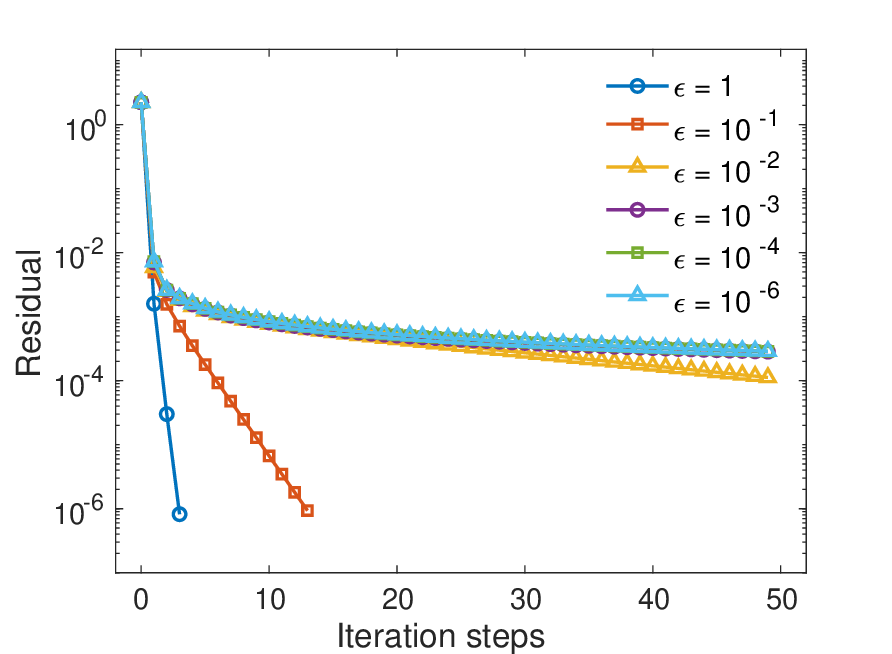}
    } \\
    \subfigure[BSGS-MM method]{
    \label{fig:BSGSMM_rate}
    \includegraphics[width=0.4\textwidth, trim=5 0 25 22, clip]{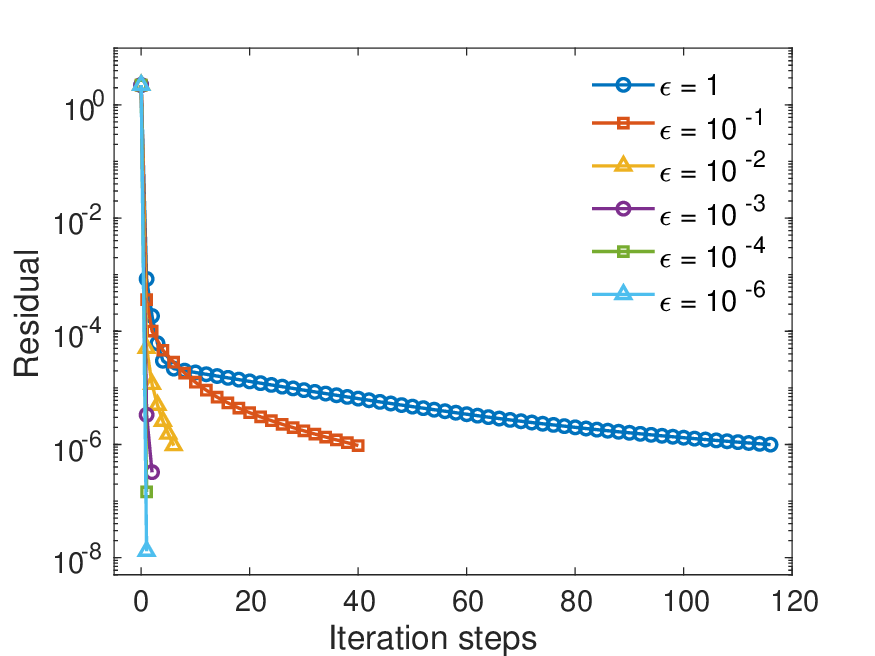}
    } \quad
    \subfigure[First 20 steps of BSGS-MM]{
    \label{fig:BSGSMM_rate_magnified}
    \includegraphics[width=0.4\textwidth, trim=5 0 40 22, clip]{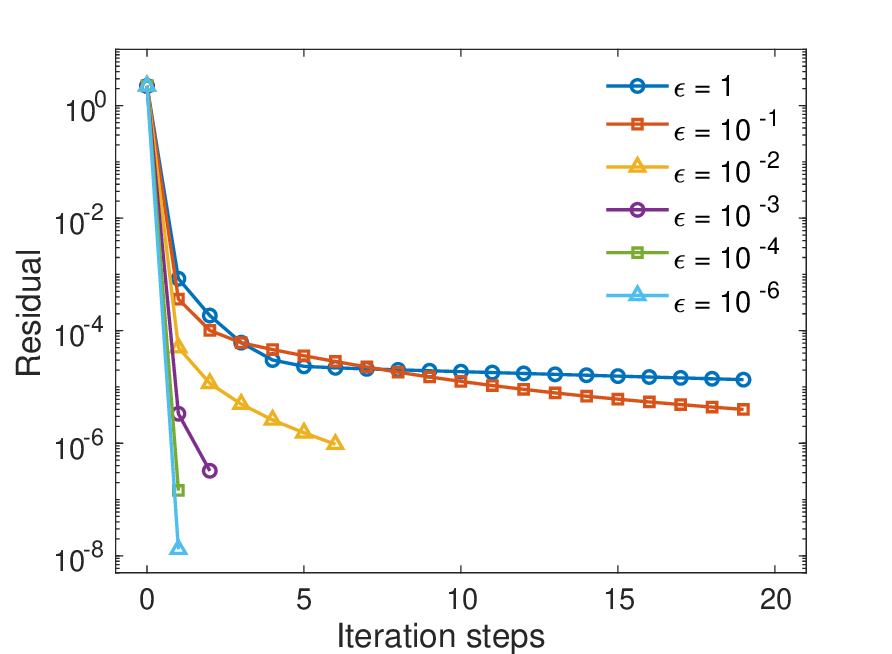}
    }
    \caption{The convergence rate of BSGS and BSGS-MM methods for one-dimensional heat transfer.}
\end{figure}

Here we will take a closer look at the BSGS-MM method. Note that this approach requires solving the equation of $\bs{u}^1$ \eqref{eq:Euler} in every iteration. Despite the small number of components in $\bs{u}^1$, solving such a system, which is essentially the Navier-Stokes equations, could be nontrivial, especially in the multi-dimensional case. Although in the one-dimensional case, the equation can be solved efficiently by the Thomas algorithm, to better understand the computational cost in more general cases, we solve $\bs{u}^1$ also by the BSGS iteration and the numbers of such inner iterations are given in Figure \ref{fig:BSGS_MM_inner}, where the horizontal axis denotes the outer iterations and the vertical axis denotes the number of inner iterations. In \eqref{eq:u_decomp}, we use $\bs{u}^1 = (u^0, u^1, u^2, u^3)$ and $\bs{u}^k = u^{k-2}$, $k=2,\ldots,N-2$. Note that the BSGS-MS method does not converge for $\epsilon = 1$, $10^{-1}$ and $10^{-2}$, and the BSGS-MM method also requires many outer iterations compared with other methods, which shows the importance of hybridizing with the BSGS method to bring down the high-frequency errors, and the figures show that the number of iterations is significantly reduced with $N_b = 1$ only. In general, the number of inner iterations increases as $\epsilon$ decreases. In the one-dimensional case, this indeed causes a rise of the computational cost, since $\bs{u}^1$ occupies about a quarter of the total number of variables, which is non-negligible. Improvements can be made by adopting a better numerical solver for equations of $\bs{u}^1$, which is left to our future works.

\begin{figure}[!ht] \label{fig:BSGS_MM_inner}
    \centering
    \subfigure[$\epsilon = 1$]{
    \includegraphics[width=0.4\textwidth, trim=10 0 28 18, clip]{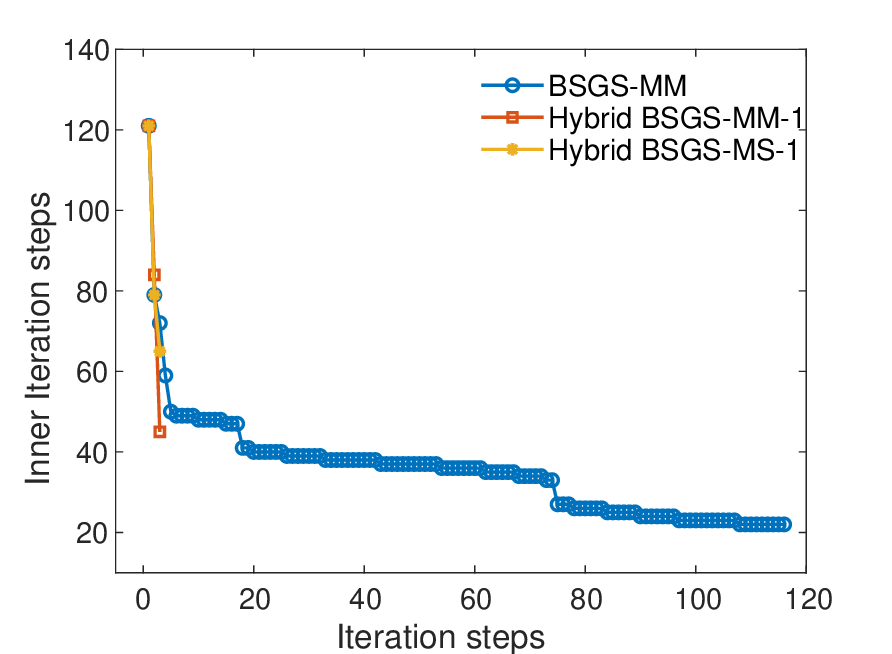}
    } \quad
    \subfigure[$\epsilon = 10^{-1}$]{
    \includegraphics[width=0.4\textwidth, trim=10 0 28 18, clip]{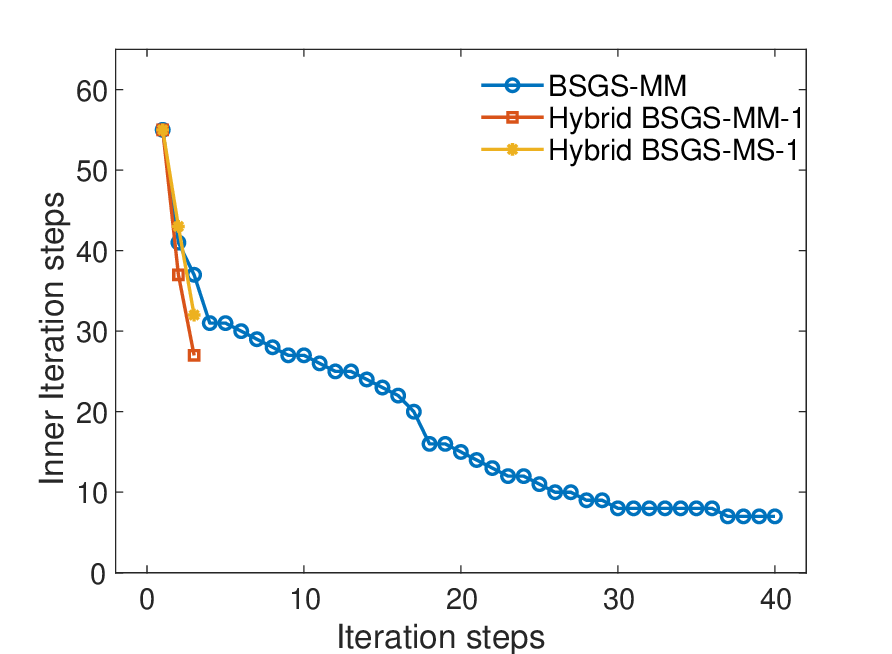}
    } \\
    \subfigure[$\epsilon = 10^{-2}$]{
    \includegraphics[width=0.4\textwidth, trim=10 0 28 18, clip]{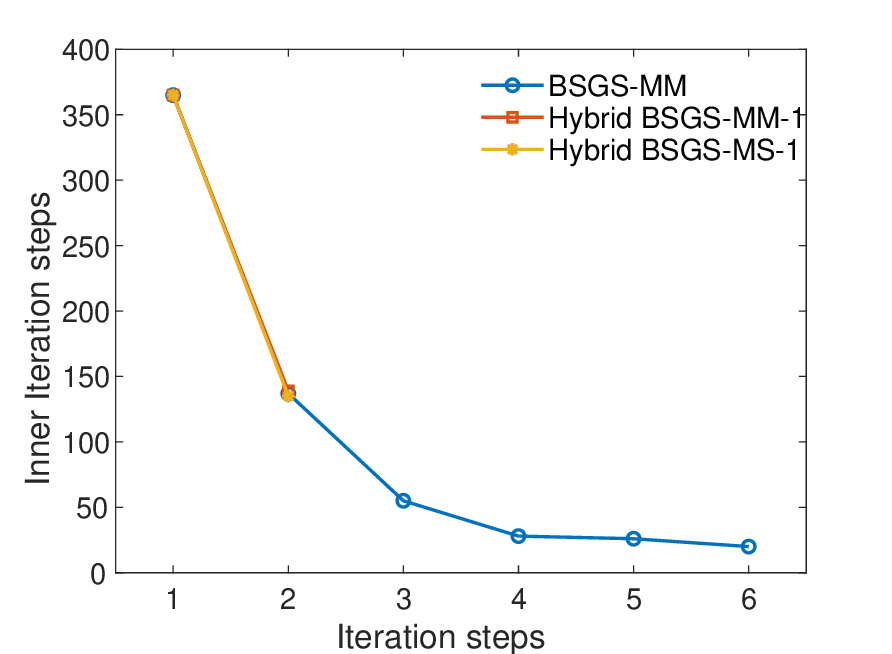}
    } \quad
    \subfigure[$\epsilon = 10^{-3}$]{
    \includegraphics[width=0.405\textwidth, trim=5 0 28 18, clip]{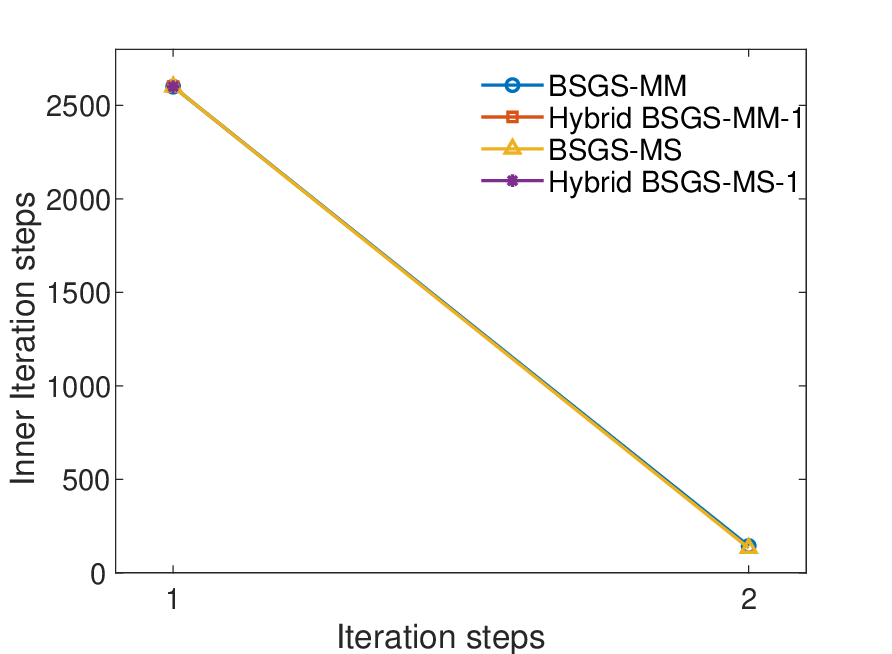} 
    } 
    \caption{The inner iterations of (Hybrid) BSGS-MM and BSGS-MS methods for one-dimensional heat transfer.}
\end{figure}

Now we study second-order methods. We will first show that a direct application of the symmetric Gauss-Seidel method without relaxation may lead to divergent results. To this end, we apply the BSGS method and the BSSR method to the second-order scheme \eqref{eq:1D_second}, and the results are given in Figure \ref{fig:diverge} for $\epsilon = 1$ and $10^{-4}$. The BSGS method without relaxation quickly loses stability, and the BSSR method shows linear convergence for $\epsilon = 1$, but converges slowly for $\epsilon = 10^{-4}$, which is consistent with the first-order results. More results on the BSSR and BSSR-MM methods are plotted in Figure \ref{fig:BSSR_rate_1} and \ref{fig:BSSR_rate_2}. Compared with Figure \ref{fig:BSGS_BSGSMM_rate}, the general behaviors of both methods are similar, but the iterative methods for second-order methods converge more slowly for the same grid size due to the more complicated structures of the matrices. The numbers of inner iterations to solve $\bs{u}^1$ for various methods are displayed in Figure \ref{fig:BSSR_MM_inner}. One can again see that inside each outer iteration, more inner iterations are needed for smaller $\epsilon$. The BSSR-MS method again diverges for $\epsilon = 1$, $10^{-1}$ and $10^{-2}$, and even the Hybrid BSSR-MS-1 method does not converge for in these three cases. We therefore increase the value of $N_b$ in our experiments. For larger $\epsilon$, a larger value of $N_b$ is needed, and the number of inner iterations can also be found in Figure \ref{fig:BSSR_MM_inner}.

\begin{figure}[!ht] 
    \centering
    \subfigure[Divergence of BSGS method]{
    \includegraphics[width=0.295\textwidth, trim=5 0 32 15, clip]{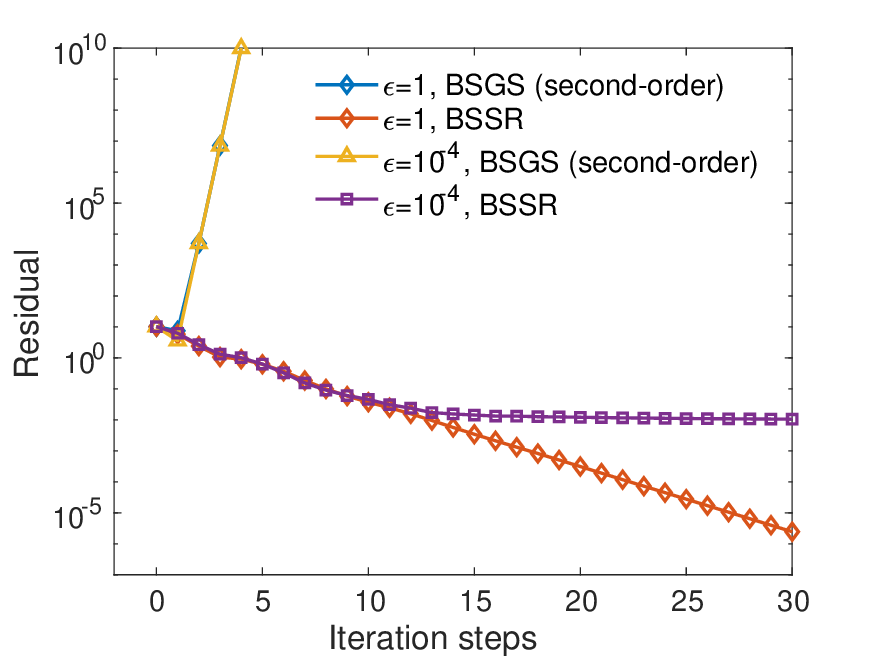} \label{fig:diverge}
    } 
    \subfigure[BSSR method]{ \label{fig:BSSR_rate_1}
    \includegraphics[width=0.30\textwidth, trim=5 0 38 22, clip]{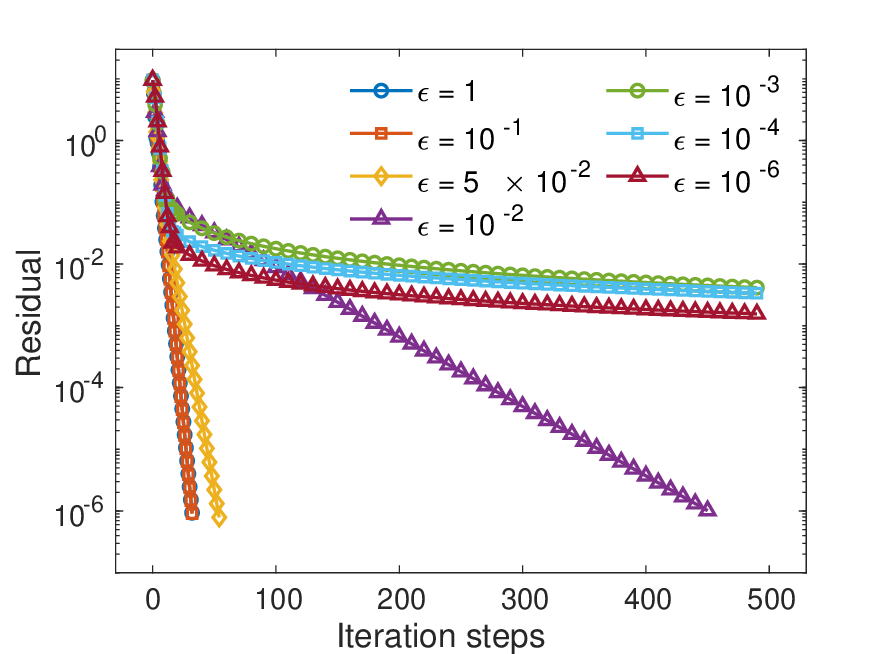} 
    } 
    \subfigure[BSSR-MM method]{ \label{fig:BSSR_rate_2}
    \includegraphics[width=0.30\textwidth, trim=5 0 38 22, clip]{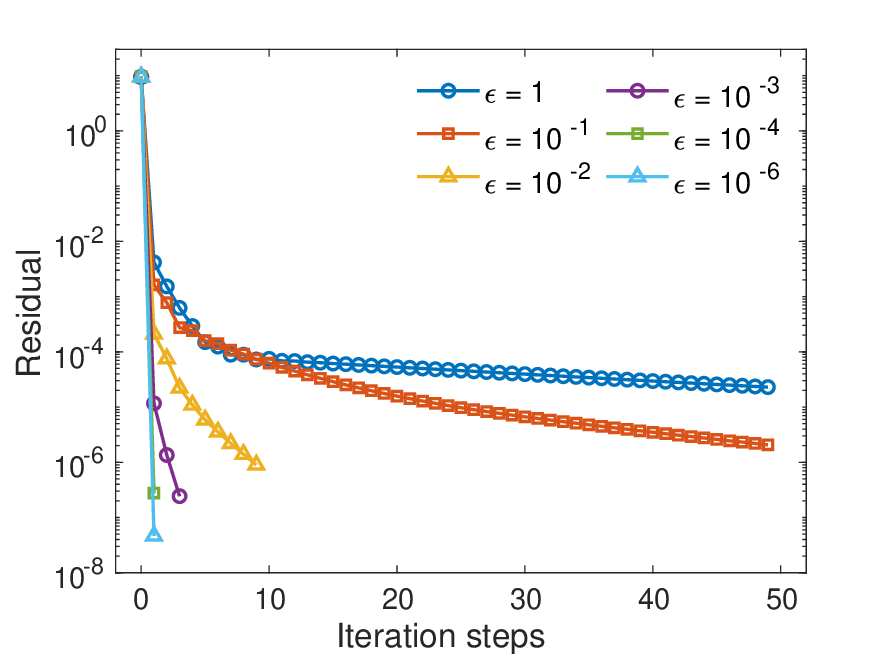}
    }
    \caption{The convergence of iterative methods for the second-order scheme for one-dimensional heat transfer.}
\end{figure}

\begin{figure}[!ht] \label{fig:BSSR_MM_inner}
    \centering
    \subfigure[$\epsilon = 1$]{
    \includegraphics[width=0.4\textwidth, trim=10 0 28 18, clip]{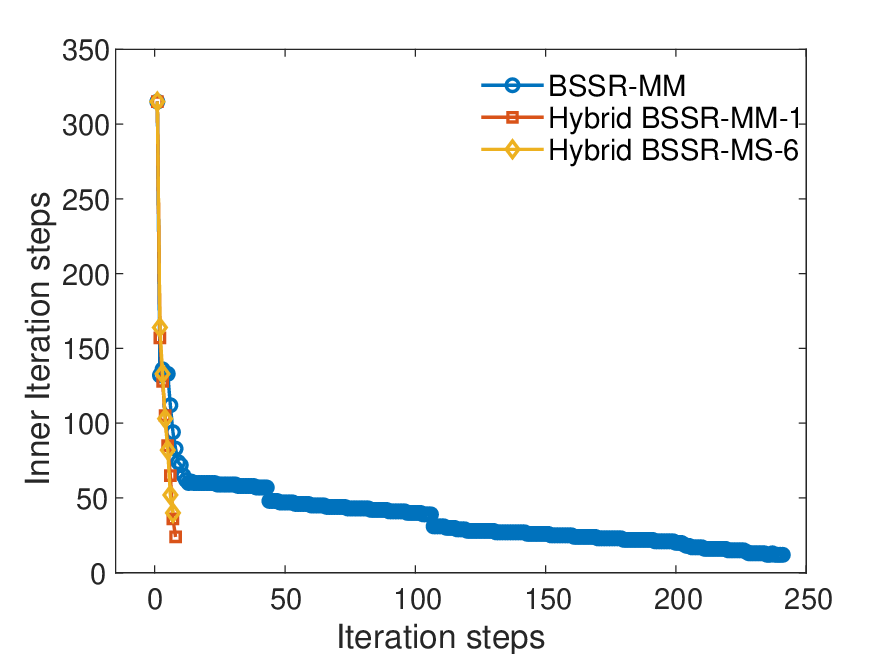}
    } \quad
    \subfigure[$\epsilon = 10^{-1}$]{
    \includegraphics[width=0.4\textwidth, trim=10 0 28 18, clip]{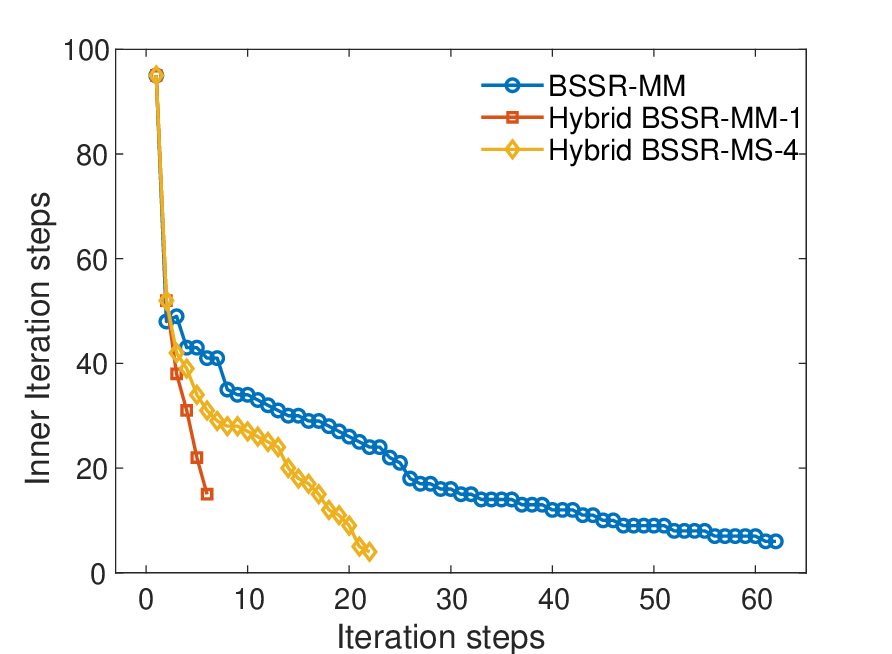}
    } \\
    \subfigure[$\epsilon = 10^{-2}$]{
    \includegraphics[width=0.4\textwidth, trim=10 0 28 18, clip]{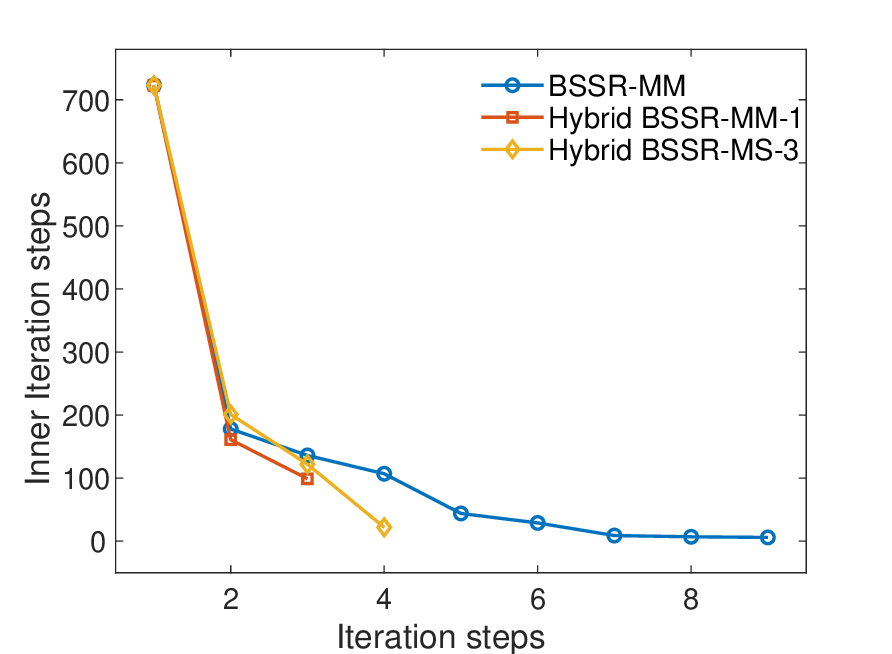}
    } \quad
    \subfigure[$\epsilon = 10^{-3}$]{
    \includegraphics[width=0.405\textwidth, trim=5 0 28 18, clip]{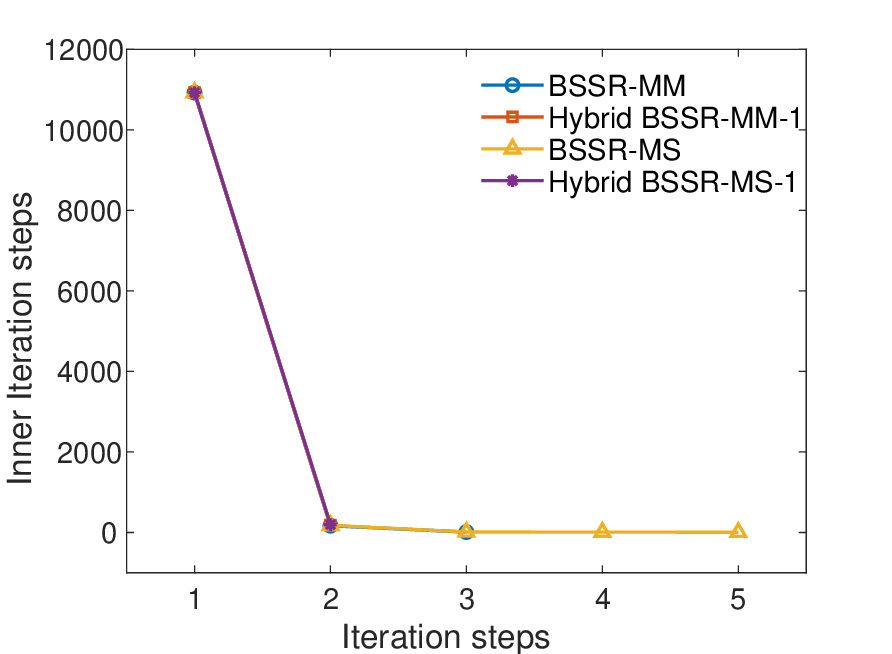} 
    } 
    \caption{The inner iterations of (Hybrid) BSSR-MM and BSSR-MS methods for one-dimensional heat transfer.}
\end{figure}

Furthermore, we compare the actual computational time for all the above methods. The data are listed in Table \ref{tab:1D_time}. For large $\epsilon$, the BSGS (or BSSR) method always has the best performance. When $\epsilon$ gets smaller, the performance of other methods is improved. In particular, the Hybrid BSSR-MM method and the Hybrid BSSR-MS method have similar computational times. However, in this one-dimensional toy model, the number of components in $\bs{u}^1$ is still about $1/4$ of the total number of variables, and solving the equation of $\bs{u}^1$ takes a large amount of time when $\epsilon$ is small. As a result, the BSGS and the BSSR methods still hold the best performance for $\epsilon = 10^{-6}$, even if other methods require only one iteration. In the next section, we will consider the three-dimensional velocity space, so that the proportion of $\bs{u}^1$ in $\bs{u}$ will become smaller. 

\begin{table}[!ht] \label{tab:1D_time}
    \caption{Computing time for one-dimensional heat transfer (measure time by the second).}
    \footnotesize
    \centering
    \begin{tabular}{ccccccccc}
        \hline 
        \hline 
        & & & $\epsilon = 1$ & $\epsilon = 10^{-1}$ & $\epsilon = 10^{-2}$ & $\epsilon = 10^{-3}$ & $\epsilon = 10^{-4}$ & $\epsilon = 10^{-6}$ \\
        \hline
        \multirow{5}{4em}{First order} & BSGS & & $0.0042$ & $0.0383$ & $0.1182$ & $0.7626$ & $1.2799$ & $1.4125$ \\
& BSGS-MM & & $2.2980$ & $0.6766$ & $0.3534$ & $1.1341$ & $1.7307$ & $1.8482$ \\
& Hybrid BSGS-MM-$1$ & & $0.1275$ & $0.1746$ & $0.4342$ & $1.0431$ & $1.7575$ & $1.9325$ \\
& BSGS-MS & & -- & -- & -- & $1.0499$ & $1.7966$ & $1.9426$ \\
& Hybrid BSGS-MS-$1$ & & $0.1447$ & $0.1796$ & $0.4480$ & $1.0000$ & $1.8690$ & $1.8713$ \\
        \hline
        \multirow{6}{4em}{Second order} & BSSR & & $0.0585$ & $0.1302$ & $0.6447$ & $10.220$ & $88.469$ & $2350.6$ \\
& BSSR-MM & & $9.2520$ & $1.7299$ & $1.2873$ & $9.5040$ & $66.147$ & $4103.9$ \\
& Hybrid BSSR-MM-$1$ & & $1.2217$ & $0.5678$ & $1.4010$ & $9.2619$ & $68.572$ & $4277.2$ \\
& BSSR-MS & & -- & -- & -- & $9.3642$ & $69.734$ & $4276.2$ \\
& Hybrid BSSR-MS-$N_b$ & & $1.3608$ & $0.6585$ & $1.6115$ & $9.8017$ & $66.349$ & $4265.6$ \\
& & & ($N_b$ = 6)  & ($N_b$ = 4) & ($N_b$ = 3) & ($N_b$ = 1) & ($N_b$ = 1) & ($N_b$ = 1) \\
        \hline
        \hline 
    \end{tabular}
\end{table}

Lastly, we present a comparison between the general synthetic iterative scheme (GSIS \cite{Su2020Can}) and the hybrid BSGS-MM method. It has been shown in \cite{Su2020Can} that the GSIS can also achieve fast convergence for both small and large Knudsen numbers.  Briefly speaking, the iteration in the GSIS is to solve the steady-state Navier-Stokes-Fourier equations with higher-order corrections to find conservative variables, and then plug them into the Maxwellian in the BGK collision term and solve the distribution function. When applied to one-dimensional moment equations, it turns out to be
\begin{equation} \label{eq:GSIS}
- \bs{A}^+ \bar{\bs{u}}_{j-1}^{(n+1)}
+ \left( |\bs{A}| - \frac{1}{\epsilon} \Delta x \bs{I} \right) \bar{\bs{u}}_j^{(n+1)} 
+ \bs{A}^- \bar{\bs{u}}_{j+1}^{(n+1)}
=\frac{1}{\epsilon} \Delta x (\bs{I} + \bs{L}) \bar{\bs{u}}_j^{(*)},
\end{equation}
where
\begin{equation*}
\bs{\bar{u}_j}^{(*)} = \begin{pmatrix} \bs{\bar{u}}_j^{1,(n+1)} \\ \bs{\bar{u}}_j^{2,(n)}, \end{pmatrix}
\end{equation*} 
and $\bs{\bar{u}}_j^{1,(n+1)}$ is solved from \eqref{eq:Euler} with $N_0 = 3$ chosen in \eqref{eq:u_twopart}. The comparison in the number of iterations are given in Fig. \ref{fig:comparison}. It is shown that both methods have fast convergence, and GSIS has a slightly larger number of iterations. Within each iteration, both GSIS and Hybrid BSGS-MM-1 need to solve the macroscopic equations \eqref{eq:Euler} once. The Hybrid BSGS-MM-1 method needs to apply the forward-backward scan of the SGS method twice, once in $\mathcal{T}_{\text{BSGS}}$ and once in $\mathcal{T}_{\text{BSGS-MM}}$ (see \eqref{eq:HybridBSGSMM}). In GSIS, each iteration requires solving \eqref{eq:GSIS} exactly, which also involves a few SGS scans. In our experiments, the convergence of these scans can always be achieved within 3 iterations, and therefore, the computational costs of each iteration in both methods are similar. Overall speaking, the Hybrid BSGS-MM-1 method has better performance in this test case.
\begin{figure}[!ht] 
    \centering
    \subfigure[GSIS]{
    \label{fig:GSIS}
    \includegraphics[width=0.4\textwidth, trim=5 0 40 22, clip]{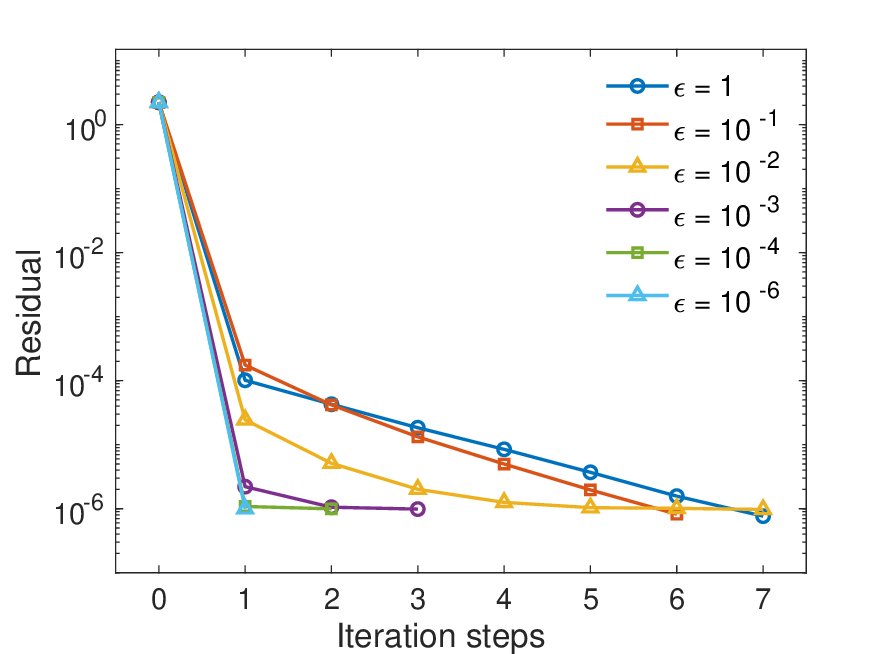}
    } \quad
    \subfigure[Hybrid BSGS-MM-1 method]{
    \includegraphics[width=0.4\textwidth, trim=5 0 40 22, clip]{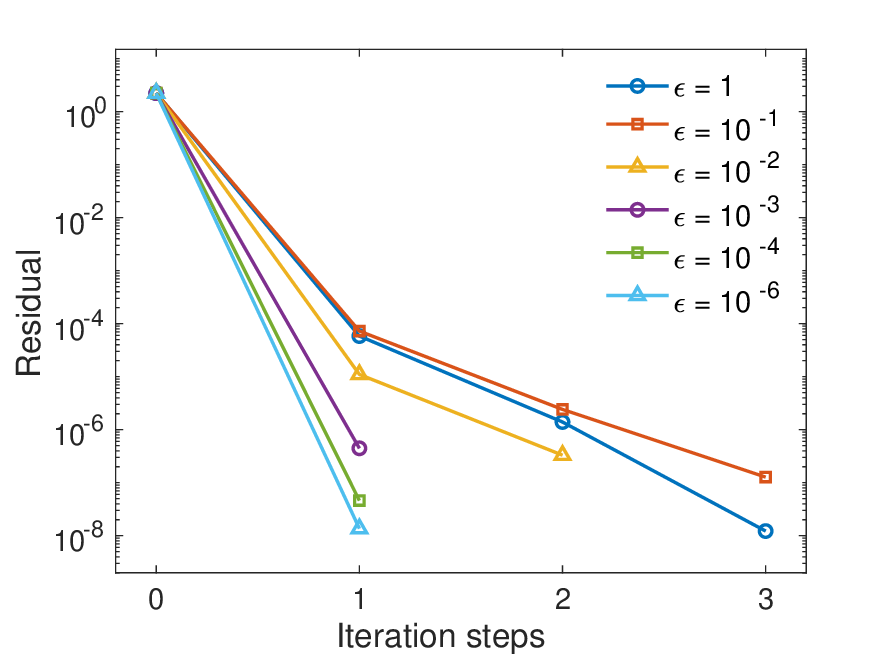}
    }
    \caption{The convergence rate of GSIS and Hybrid BSGS-MM method for one-dimensional heat transfer.}\label{fig:comparison}
\end{figure}

\subsection{Examples with two-dimensional space and three-dimensional velocity}
To reflect a more realistic case, we study the following linearized Boltzmann equation with $\bs{x} \in \mathbb{R}^2$ and $\bs{v} \in \mathbb{R}^3$, and we adopt the expansion of the distribution function in \cite{Kumar1966,Hu2020burnett} using Burnett polynomials:
\begin{equation} \label{eq:3d_expansion}
f(\bs{x}, \bs{v}) = \sum_{l=0}^{\infty} \sum_{m=-l}^{l} \sum_{n=0}^{\infty} u^{l m n} (\bs{x}) \varphi_{l m n} (\bs{v}) \omega(\bs{v})
\end{equation}
where
\begin{equation*}
\varphi_{l m n} (\bs{v}) = |\bs{v}|^l Y_{lm} \left( \frac{\bs{v}}{|\bs{v}|} \right) L_n^{(l+1/2)} \left( \frac{|\bs{v}|^2}{2} \right),
\quad
\omega(\bs{v}) = \frac{1}{(\sqrt{2 \pi})^3} \exp \left( - \frac{|\bs{v}|^2}{2} \right)
\end{equation*}
with $Y_{lm}(\cdot)$ are real spherical harmonics \cite{Blanco1997} and $L_n^{(l+1/2)}$ represent Laguerre polynomials. The collision operator $\mathcal{L}$ is chosen based on Maxwell molecules, so that the matrix $\bs{L}$ in \eqref{eq:moment_eqn} is diagonal. Some of the diagonal values are given by the following facts:
 \begin{gather*}
 \mathcal{L} [\varphi_{000} \omega] = 0, \quad \mathcal{L} [\varphi_{001} \omega] = 0, \quad \mathcal{L} [\varphi_{1m0} \omega] = 0, \quad \mathcal{L} [\varphi_{1m1} \omega] = -\frac{2}{3}\varphi_{1m1} \omega, \quad \mathcal{L} [\varphi_{2m0} \omega] = -\varphi_{2m0} \omega.
 \end{gather*}
More coefficients can be found in \cite{Alterman}. The macroscopic quantities including the density and the temperature are related to the coefficients by $(\rho, T) = \left( u_{000}, -u_{001}\right) / \left( 2 \sqrt{\pi} \right)$. The truncation of the infinite series \eqref{eq:3d_expansion} is done by selecting a positive integer $L$, and then reserve terms with $l=0,1,\cdots,L$ and $m=-l,\cdots,l$, $n=0,1,\cdots,\lceil (L-l)/2 \rceil$. The solution $\bs{u}$ is then a vector composed of $u^{lmn}$ with $l,m,n$ within the above range, and the two-dimensional moment equations can be formulated as
\begin{equation} \label{eq:2D_moment_eqn}
\bs{A}_1 \frac{\partial \bs{u}}{\partial x_1} + \bs{A}_2 \frac{\partial \bs{u}}{\partial x_2} = \frac{1}{\epsilon} \bs{L} \bs{u}.
\end{equation}
Such a choice of the truncation covers the classical Euler equations ($L = 1$), Grad's 13-moment equations ($L = 2$) and Grad's 26-moment equations ($L = 3$). In the following sections, two test cases will be studied under these settings.

\subsubsection{Heat transfer in a cavity} \label{sec:2D_heat}

The first test case is a heat transfer problem where the domain is a unit square $\Omega = (0,1) \times (0,1)$. The boundaries can have different temperatures, causing heat transfer inside the chamber.
All walls are fully diffusive, so that the boundary conditions of the moment equations \eqref{eq:2D_moment_eqn} can again be formulated (according to Section $3$ of Supplementary Material) .
In our test, we set the temperature of the top wall $(0,1)\times\{1\}$ to be $1$, and all other walls have temperature $0$. Numerically, the boundary conditions are processed using the ghost-cell method as in the one-dimensional case \eqref{eq:bc}. To uniquely determine the solution, we again set the total mass to be $1$:
\begin{displaymath}
\int_{\Omega} \rho(\bs{x}) \,\mathrm{d}\bs{x} =
\frac{1}{2\sqrt{\pi}}\int_{\Omega} u^{000}(\bs{x}) \,\mathrm{d}\bs{x} = 1.
\end{displaymath}

To solve the heat transfer problem, we  use a $50\times 50$ uniform grid to discretize the spatial domain, and choose $L = 10$ in the truncation of the series \eqref{eq:3d_expansion}. The total number of equations is $341$. The macroscopic part of the solution $\bs{u}^1$ in \eqref{eq:u_twopart} and \eqref{eq:u_decomp} consists of $13$ variables, including all $u^{lmn}$ with $l \leqslant 2$ and $n \leqslant \lceil (2-l)/2 \rceil$. Thus, the equations of $\bs{u}^1$ are essentially Grad's 13-moment equations, which are equivalent to the Navier-Stokes equations in the linear, steady-state case. In \eqref{eq:u_decomp}, by asymptotic analysis \cite{Struchtrup2004}, the vector $\bs{u}^k$ for $k = 2,\ldots,L-1$ is chosen as a vector consisting of all variables in $\mathcal{U}_{k+1} \setminus \mathcal{U}_k$  with
\begin{displaymath}
\mathcal{U}_k = \{u^{lmn} \mid l =0,1,\ldots,k, \quad m = -l,\ldots,l, \quad n = 0, \ldots, \lceil (k-l)/2 \rceil\}.
\end{displaymath}
The numerical solutions of the density and temperature computed using the second-order method are displayed in Figures \ref{fig:2D_rho} and \ref{fig:2D_T}. For larger $\epsilon$, both the density and temperature in the domain distribute more homogeneously, and the temperature jump on the top wall becomes more obvious, leading to a large discontinuity in the top-left and top-right corners.

\begin{figure}[!ht] \label{fig:2D_rho}
    \centering
    \subfigure[$\epsilon = 1$]{
    \includegraphics[width=0.31\textwidth, trim=15 3 45 20, clip]{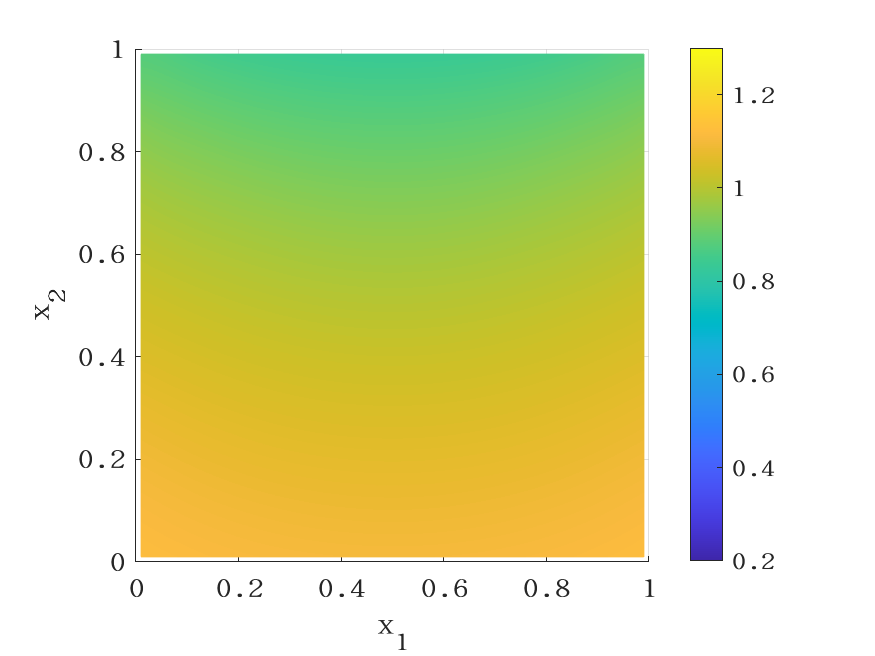}
    } 
    \subfigure[$\epsilon = 10^{-2}$]{
    \includegraphics[width=0.31\textwidth, trim=15 3 45 20, clip]{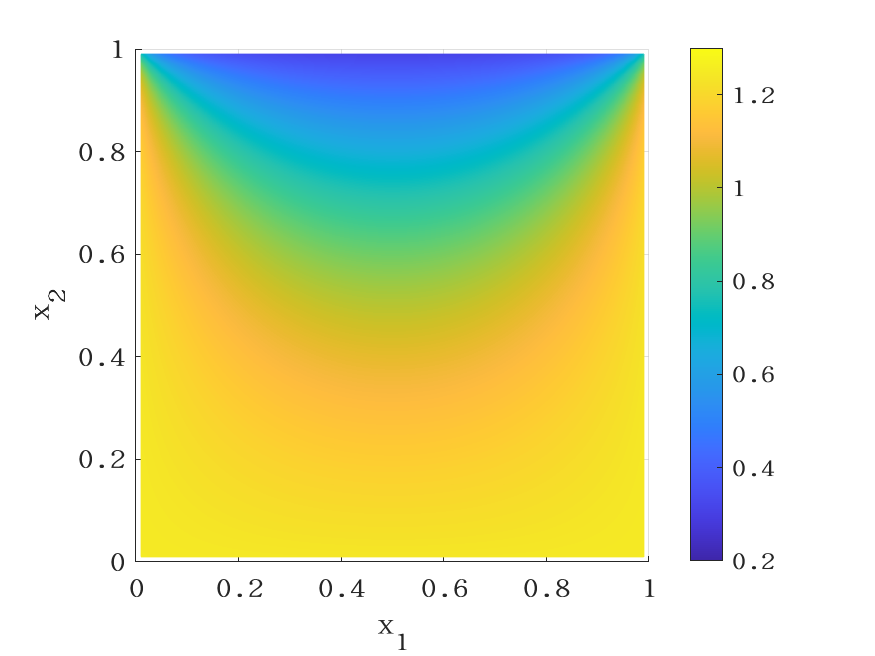}
    } 
    \subfigure[$\epsilon = 10^{-4}$]{
    \includegraphics[width=0.31\textwidth, trim=15 3 45 20, clip]{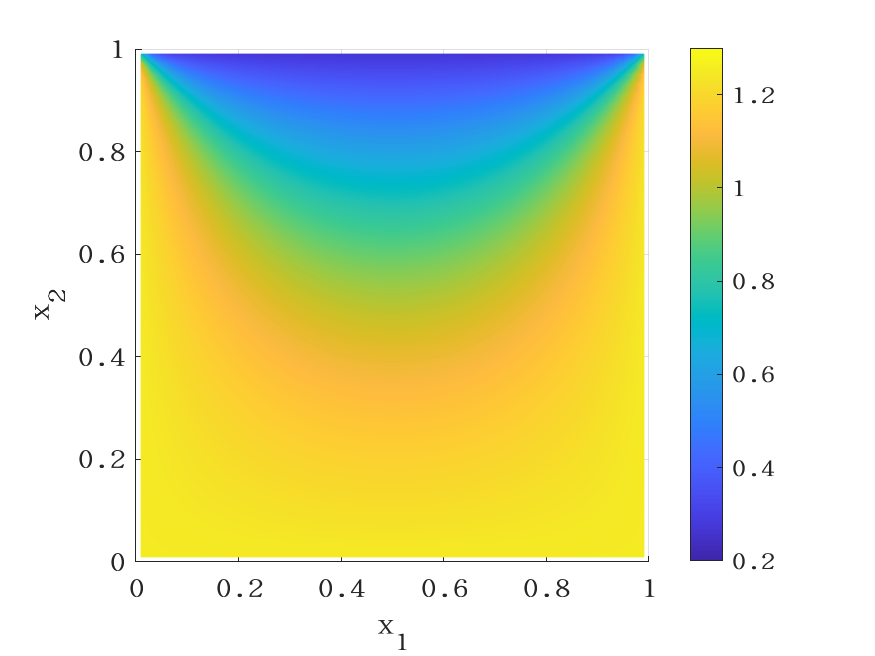}
    }
    \caption{The distributions of density for heat transfer in a cavity.}
\end{figure}
\begin{figure}[!ht] \label{fig:2D_T}
    \centering
    \subfigure[$\epsilon = 1$]{
    \includegraphics[width=0.31\textwidth, trim=15 3 45 20, clip]{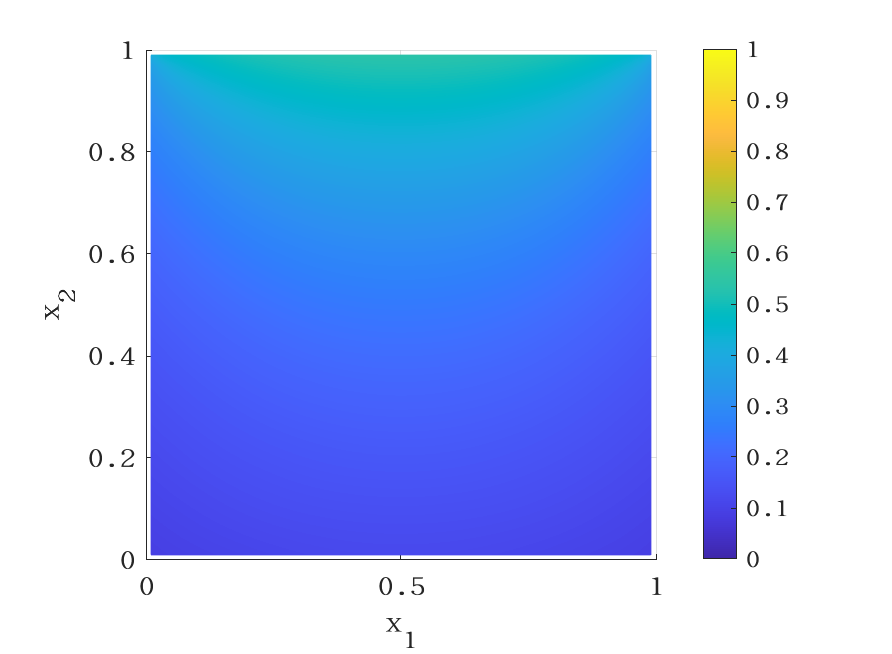}
    } 
    \subfigure[$\epsilon = 10^{-2}$]{
    \includegraphics[width=0.31\textwidth, trim=15 3 45 20, clip]{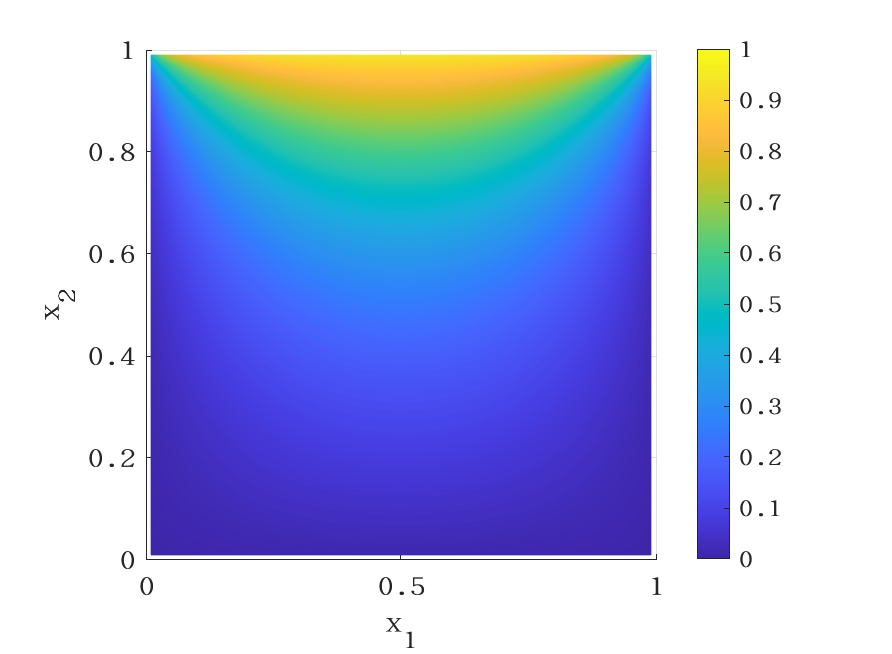}
    } 
    \subfigure[$\epsilon = 10^{-4}$]{
    \includegraphics[width=0.31\textwidth, trim=15 3 45 20, clip]{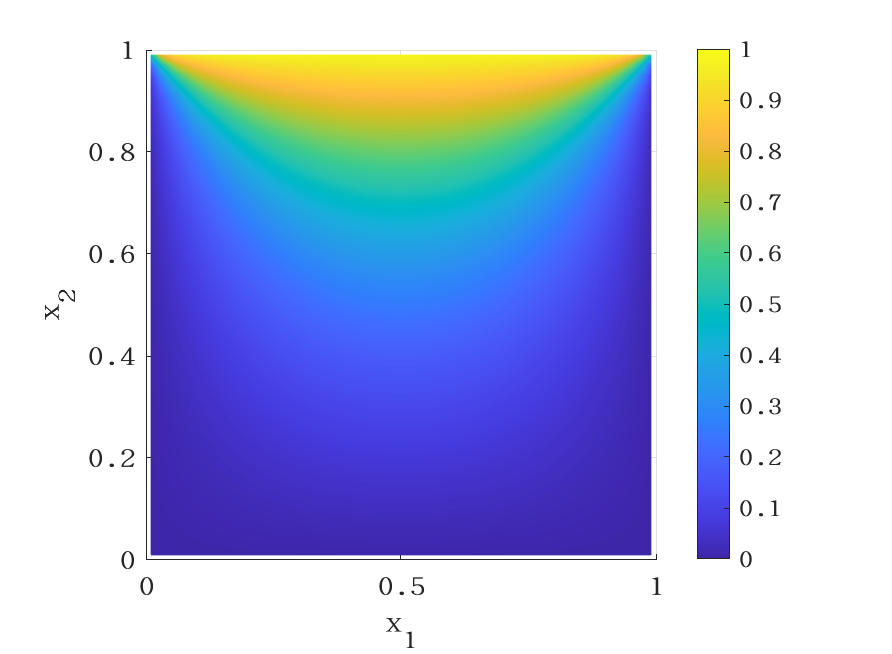}
    }
    \caption{The distributions of temperature for the heat transfer in a cavity.}
\end{figure}

We first study the performance of the iterative methods applied to the first-order scheme \eqref{eq:1D_first}. The norm of the residual versus the number of iterations is given in Figure \ref{fig:2D_BSGS_rate} for the BSGS and BSGS-MM methods. In this test case, even for $\epsilon = 0.1$, the BSGS-MM method reaches the threshold $10^{-6}$ within 20 iterations, which is faster than the BSGS method. However, it fails to converge for $\epsilon = 1$. Interestingly, the Hybrid BSGS-MM method does not converge for $\epsilon = 1$ as well, even if we increase $N_b$ to $20$. Further increasing $N_b$ is meaningless since the BSGS method requires only $47$ steps to reach the threshold of the residual. The Hybrid BSGS-MS method converges when $N_b$ is set to be $6$, and the number of inner iteration for solving the equations of $\bs{u}^1$ is given in Figure \ref{fig:2D_inner_eps1}. Figure \ref{fig:2D_BSGS_MM_inner} also gives results for other values of $\epsilon$. When $\epsilon = 10^{-1}$, the BSGS-MS method still fails to converge, but all other methods perform quite well. In general, when $\epsilon$ decreases, more inner iteration steps are required to solve the 13-moment equations, but this part takes a relatively small proportion of the total computational cost due to its small number of variables.

\begin{figure}[!ht] \label{fig:2D_BSGS_rate}
    \centering
    \subfigure[BSGS method]{
    \includegraphics[width=0.4\textwidth, trim=5 0 38 22, clip]{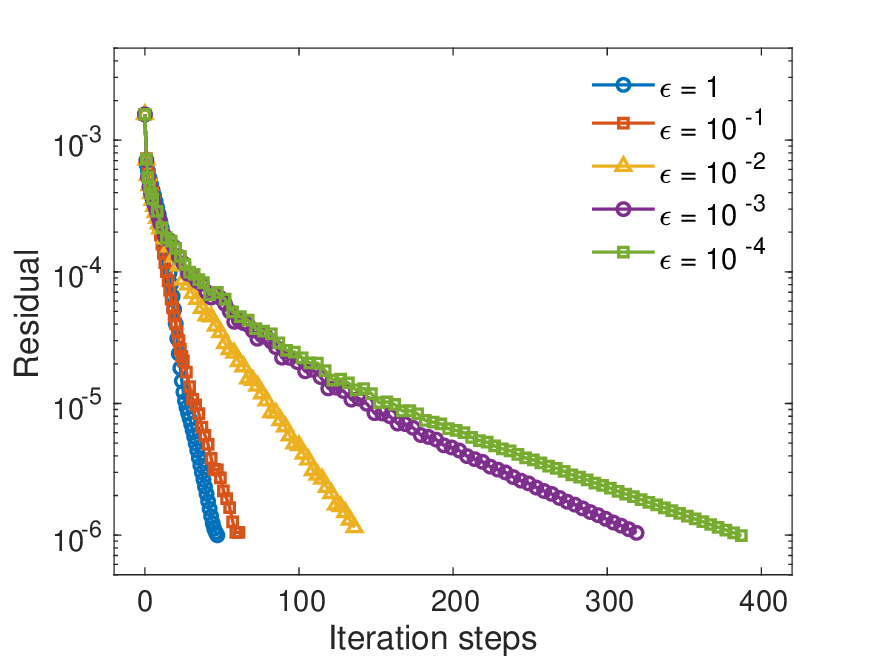}
    } \quad
    \subfigure[BSGS-MM method]{
    \includegraphics[width=0.405\textwidth, trim=5 0 25 22, clip]{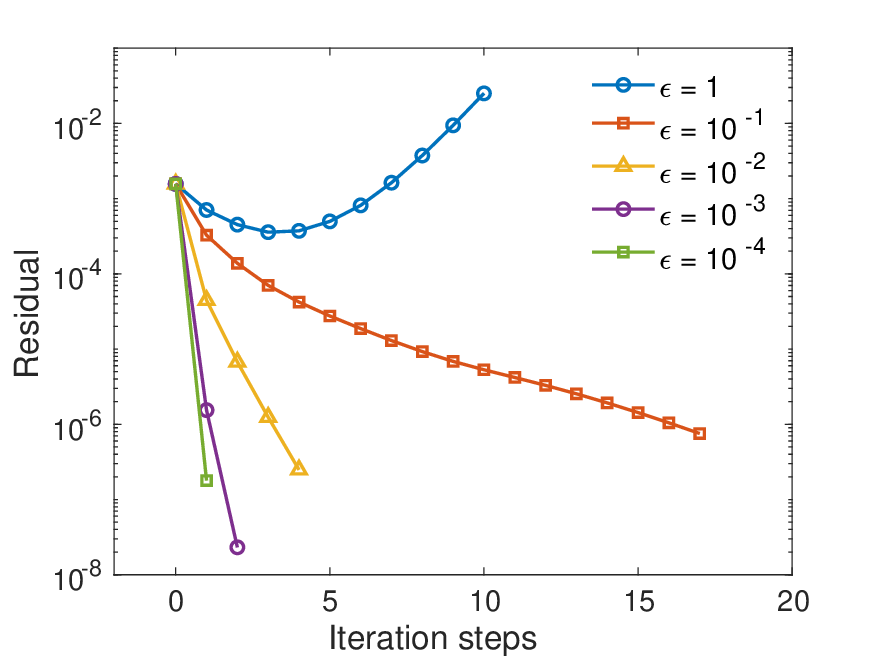}
    }
    \caption{The convergence rate of BSGS and BSGS-MM methods for heat transfer in a cavity.}
\end{figure}

\begin{figure}
\label{fig:2D_BSGS_MM_inner}
    \centering
    \subfigure[$\epsilon = 1$]{\label{fig:2D_inner_eps1}
    \includegraphics[width=0.4\textwidth, trim=10 0 28 18, clip]{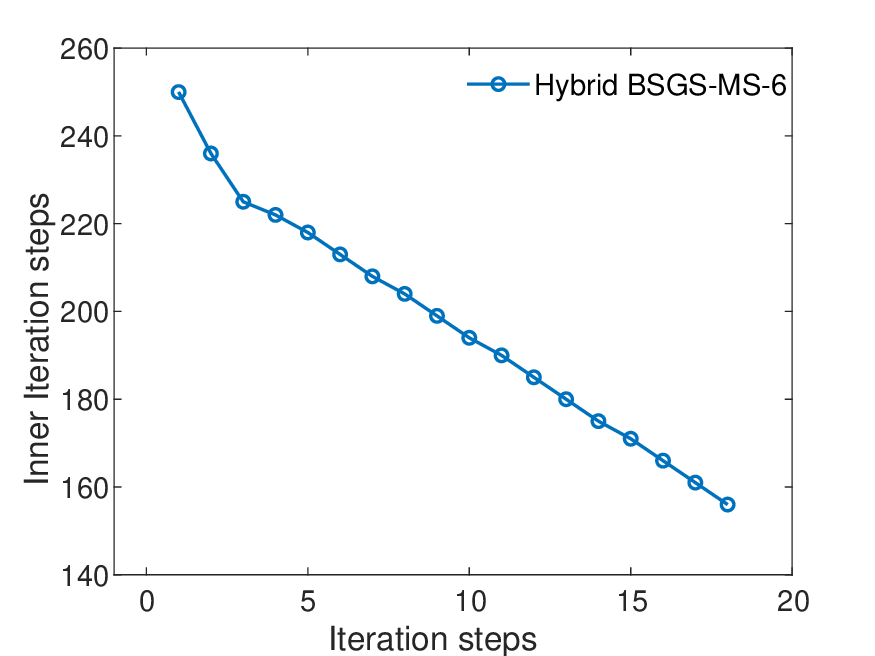}
    } \quad
    \subfigure[$\epsilon = 10^{-1}$]{
    \includegraphics[width=0.405\textwidth, trim=10 0 28 18, clip]{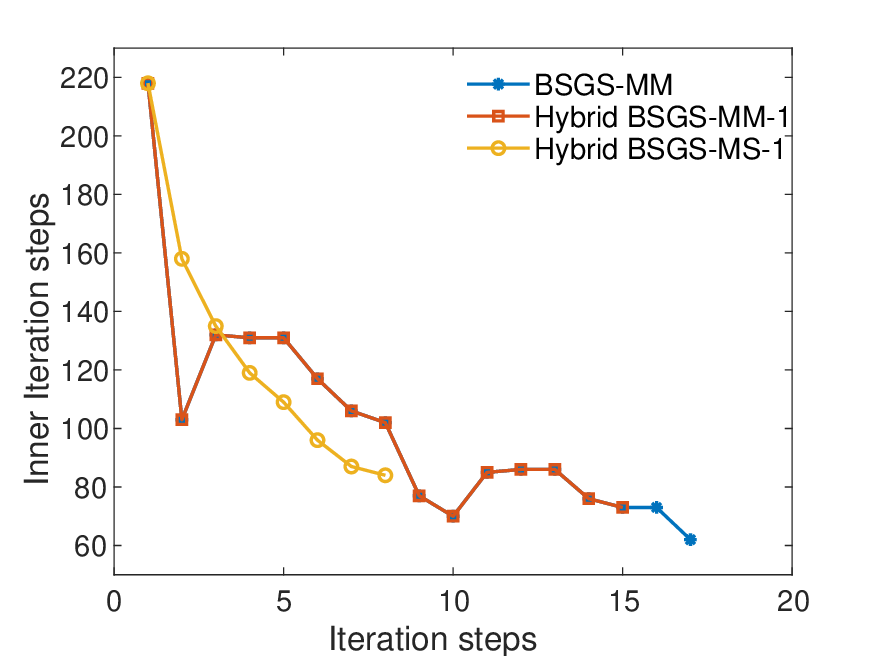}
    } \\
    \subfigure[$\epsilon = 10^{-2}$]{
    \includegraphics[width=0.4\textwidth, trim=10 0 28 18, clip]{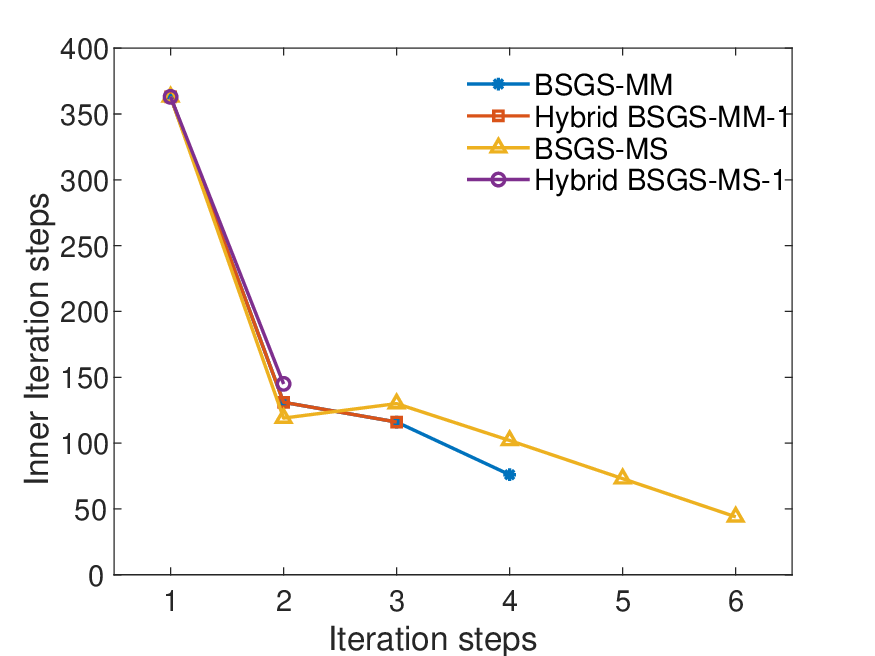}
    }  \quad
    \subfigure[$\epsilon = 10^{-3}$]{
    \includegraphics[width=0.41\textwidth, trim=5 0 28 18, clip]{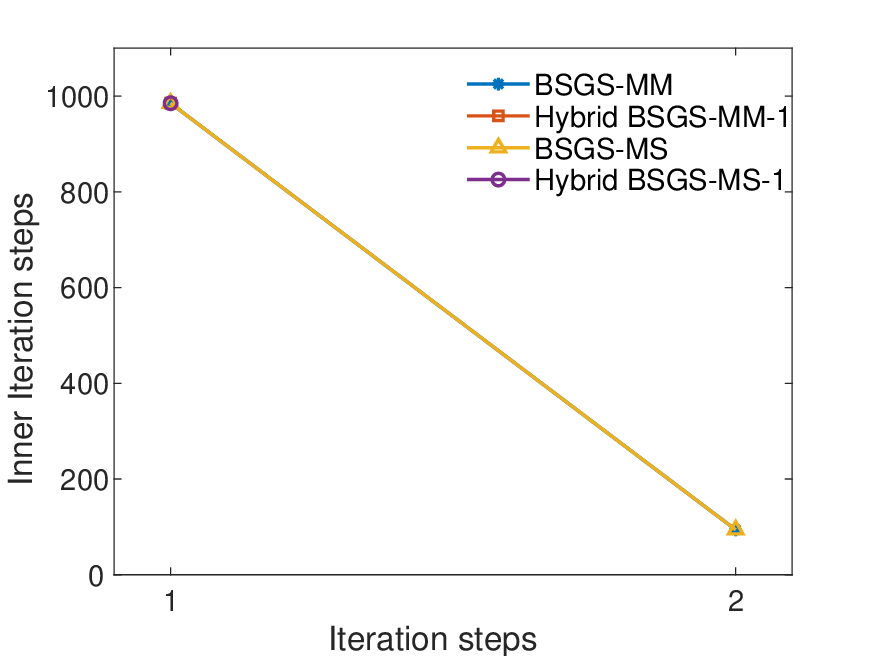} 
    } 
    \caption{The inner iterations of (Hybrid) BSGS-MM and BSGS-MS methods for heat transfer in a cavity.}
\end{figure}

The behavior of iterative methods for the second-order method \eqref{eq:1D_second} is generally the same, but Figure \ref{fig:2D_BSSR_rate} shows that the BSSR method requires more iterations to converge compared with the first-order method, and the BSSR-MM method diverges even for $\epsilon = 10^{-1}$. The numbers of inner iterations are given in Figure \ref{fig:2D_BSSR_MM_inner}, but we again emphasize that such data become less significant when $\bs{u}^1$ only takes up a small proportion of the entire solution. It is worth mentioning that when $\epsilon = 1$, the Hybrid BSSR-MM also fails to converge for a reasonable choice of $N_b$.
 
\begin{figure}[!ht] \label{fig:2D_BSSR_rate}
    \centering
    \subfigure[BSSR method]{
    \includegraphics[width=0.4\textwidth, trim=5 0 38 22, clip]{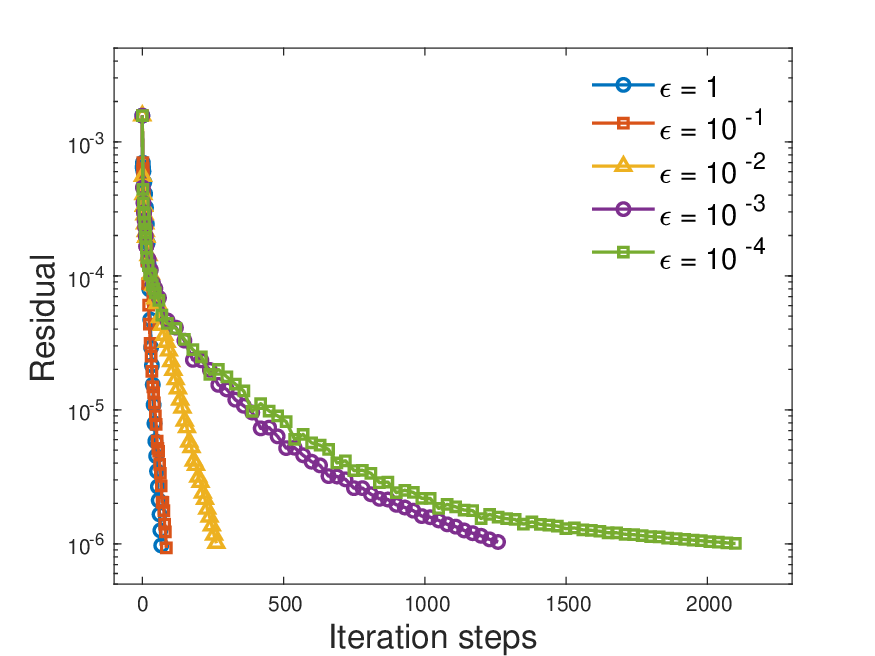}
    } \quad
    \subfigure[BSSR-MM method]{
    \includegraphics[width=0.4\textwidth, trim=5 0 38 22, clip]{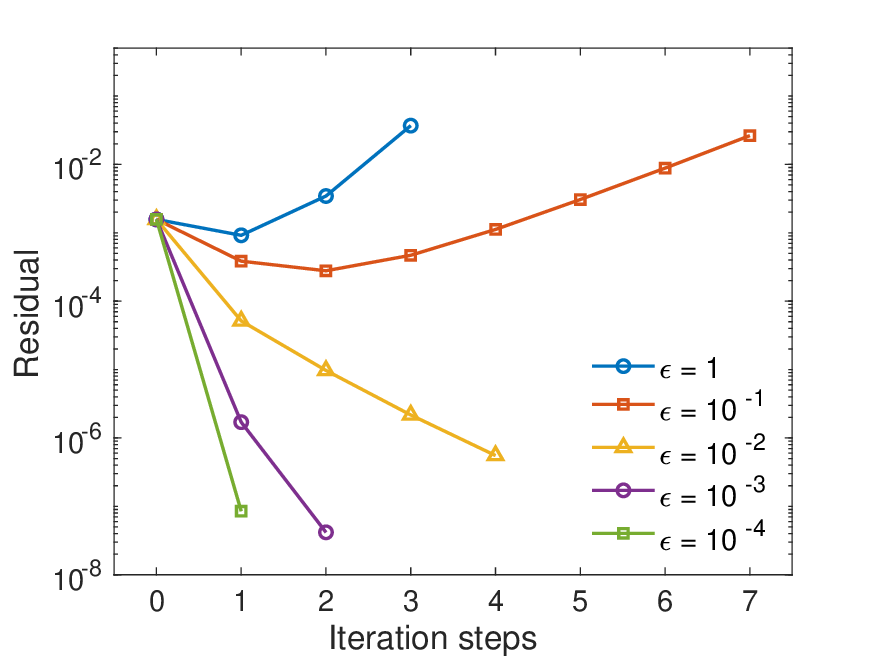}
    }
    \caption{The convergence rate of BSSR and BSSR-MM methods for heat transfer in a cavity.}
\end{figure}
\begin{figure}[!ht] \label{fig:2D_BSSR_MM_inner}
    \centering
    \subfigure[$\epsilon = 10^{-1}$]{
    \includegraphics[width=0.30\textwidth, trim=10 0 28 18, clip]{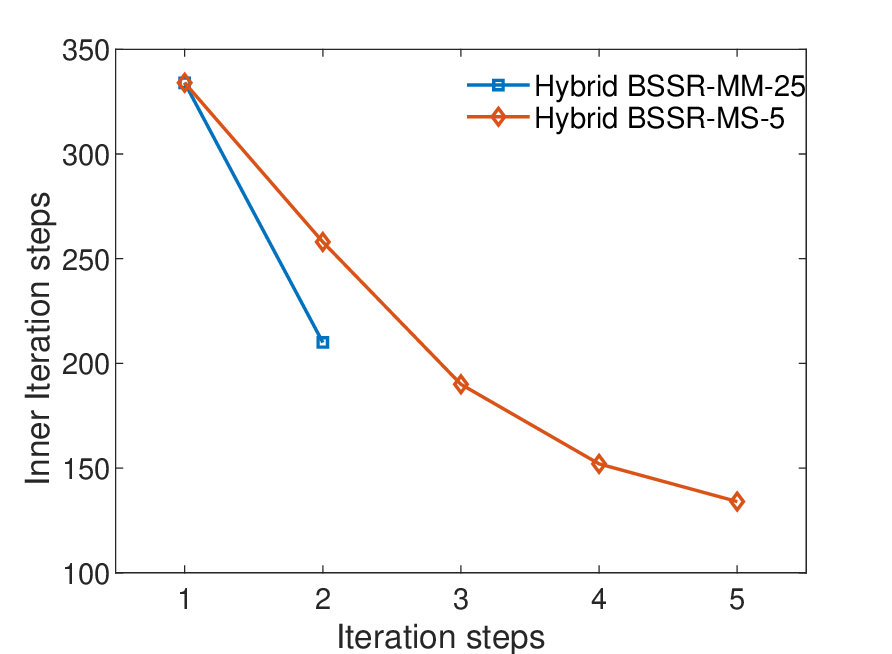}
    } 
    \subfigure[$\epsilon = 10^{-2}$]{
    \includegraphics[width=0.295\textwidth, trim=10 0 28 18, clip]{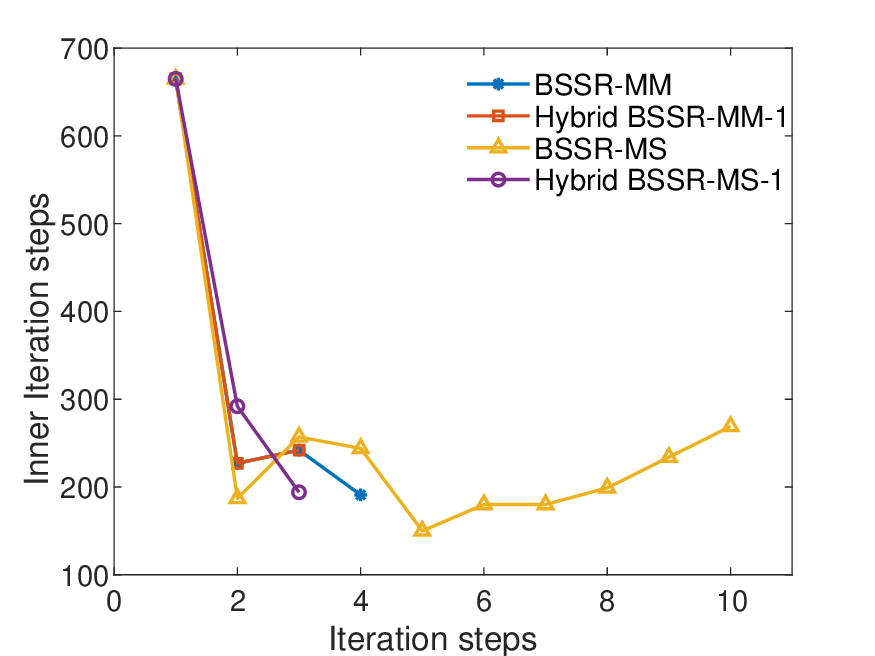}
    } 
    \subfigure[$\epsilon = 10^{-3}$]{
    \includegraphics[width=0.30\textwidth, trim=5 0 28 18, clip]{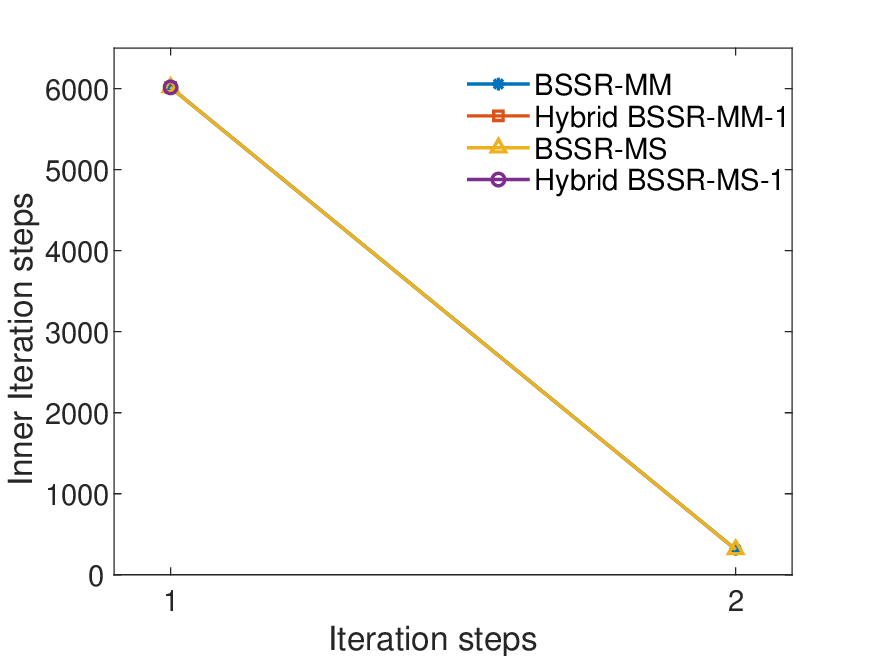} 
    }
    \caption{The inner iterations of (Hybrid) BSSR-MM and BSSR-MS methods for heat transfer in a cavity.}
\end{figure}

To summarize, we tabulate the computational time of all the methods above in Table \ref{tab:2D_time}. It is clear that for large $\epsilon$, the BSGS/BSSR method is still the optimal choice among this family of approaches. On the other end, when $\epsilon = 10^{-4}$, all methods with the micro-macro or multiscale decomposition have similar performances since they have only one or two outer iterations. Averagely, the Hybrid BSGS-MS/BSSR-MS method has the best performance since it only requires solving small matrices in each iteration. In these tests, when $\epsilon$ is close to zero, the computational cost is still high due to the inefficiency of solving the 13-moment equations (or Navier-Stokes-Fourier equations). This can be improved by adopting a better solver for the equations of $\bs{u}^1$.

\begin{table}[!ht] \label{tab:2D_time}
    \caption{Computing time for heat transfer in a cavity (measure time by the second).}
    \footnotesize
    \centering
    \begin{tabular}{cccccccc}
        \hline 
        \hline 
        & & & $\epsilon = 1$ & $\epsilon = 10^{-1}$ & $\epsilon = 10^{-2}$ & $\epsilon = 10^{-3}$ & $\epsilon = 10^{-4}$ \\
        \hline
        \multirow{6}{4em}{First order} & BSGS & & $140.17$ & $183.84$ & $410.29$ & $950.68$ & $1127.8$ \\
& BSGS-MM & & -- & $93.790$ & $40.277$ & $65.446$ & $71.606$ \\
& Hybrid BSGS-MM-$1$ && -- & $102.51$ & $43.035$ & $55.705$ & $70.045$ \\ 
& BSGS-MS & & -- & -- & $42.597$ & $59.723$ & $69.711$ \\
& Hybrid BSGS-MS-$N_b$ & & $321.56$ & $57.126$ & $22.851$ & $32.920$ & $70.575$ \\
& & & ($N_b=6$) & ($N_b=1$) & ($N_b=1$) & ($N_b=1$) & ($N_b=1$) \\
        \hline
        \multirow{7}{4em}{Second order} & BSSR & & $678.13$ & $881.67$ & $2714.7$ & $13139$ & $21693$ \\
& BSSR-MM & & -- & -- & $124.23$ & $348.22$ & $2530.4$ \\
& Hybrid BSSR-MM-$N_b$ & & -- &  $391.28$ & $124.66$ & $341.44$ & $2579.6$ \\
& & & & ($N_b=25$) & ($N_b=1$) & ($N_b=1$) & ($N_b=1$) \\
& BSSR-MS & & -- & -- & $125.03$ & $351.18$ & $2513.2$ \\
& Hybrid BSSR-MS-$N_b$ && -- & $258.07$ & $125.83$ & $342.41$ & $2558.9$ \\
& & & & ($N_b=5$) & ($N_b=1$) & ($N_b=1$) & ($N_b=1$) \\
        \hline
        \hline 
    \end{tabular}
\end{table}
 
\subsubsection{Cavity flow}

Lid-driven cavity flow is another common benchmark test for two-dimensional rarefied flows. We again assume that the gas is confined in a square cavity $\Omega = (0,1)\times (0,1)$. All the walls have the same temperature. The top lid at $y = 1$ has a horizontal velocity $\bs{U}_w = (1,0)^T$, and all other walls are stationary. The boundary conditions of moment methods can again be formulated (according to Section $3$ of Supplementary Material). In our tests, we again choose $L = 10$ and the grid size to be $50 \times 50$. Only the second-order scheme \eqref{eq:1D_second} is tested in our experiments. The numerical results for the temperature and the streamlines of the heat flux are plotted in Figure \ref{fig:temp_heatflux}. Note that the heat flux $\bs{q}$ can also be obtained directly from the variables $u^{lmn}$ by
\begin{displaymath}
    q_1(\bs{x}) = - \frac{5}{4} \sqrt{\frac{3}{\pi}} u_{111}(\bs{x}),\quad
    q_2(\bs{x}) = - \frac{5}{4} \sqrt{\frac{3}{\pi}} u_{1,-1,1}(\bs{x}),\quad
    q_3(\bs{x}) = - \frac{5}{4} \sqrt{\frac{3}{\pi}} u_{101}(\bs{x}).
\end{displaymath}
For large Knudsen numbers, it can be observed in the plots that the heat transfers from the cold area to the hot area, which is one of the typical rarefaction effects.

Table \ref{tab:cavityflow_time} lists the computational time of all the five iterative methods. The Hybrid BSSR-MS method again shows its competitiveness in most cases, despite its divergence for $\epsilon = 1$. For $\epsilon = 10^{-3}$ and $10^{-4}$, all methods except BSSR requires only one outer iteration, and again the performance is limited by the Navier-Stokes solver, which is to be improved in our future works.

\begin{figure}[!ht] \label{fig:temp_heatflux}
    \centering
    \subfigure[$\epsilon = 1$]{
    \includegraphics[width=0.31\textwidth, trim=15 3 35 5, clip]{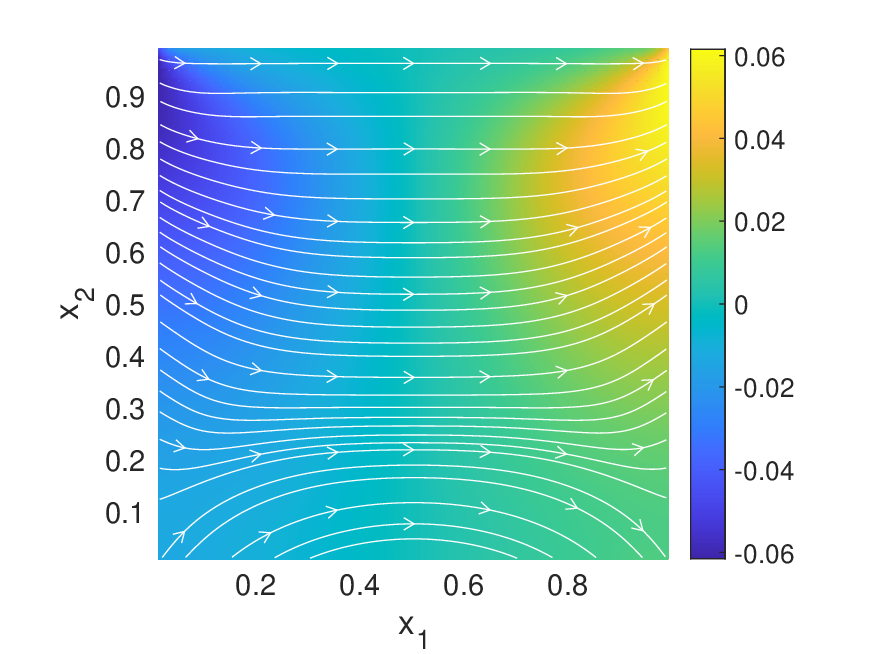}
    }
    \subfigure[$\epsilon = 10^{-2}$]{
    \includegraphics[width=0.31\textwidth, trim=15 3 35 5, clip]{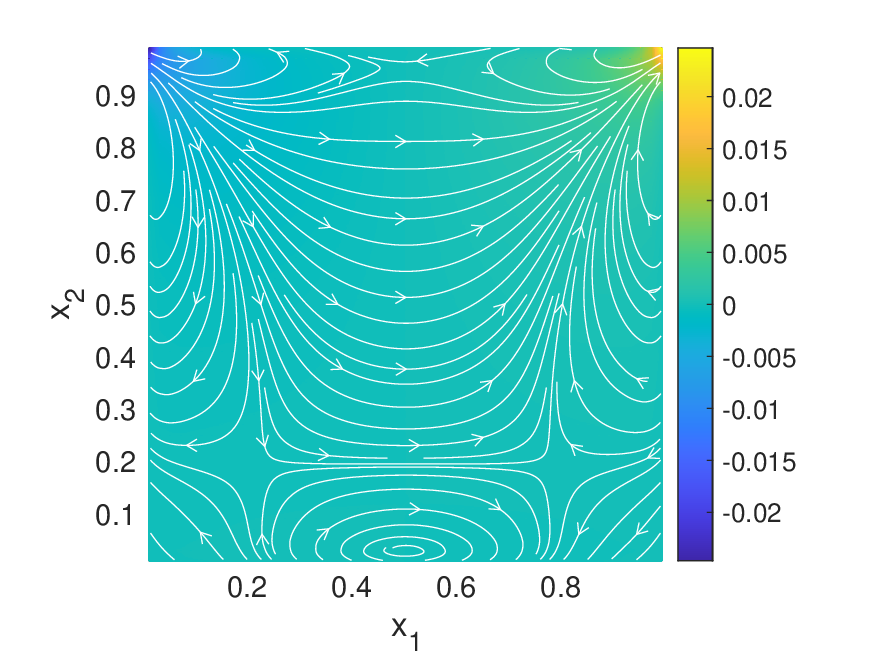} 
    }
    \subfigure[$\epsilon = 10^{-4}$]{
    \includegraphics[width=0.31\textwidth, trim=15 3 35 5, clip]{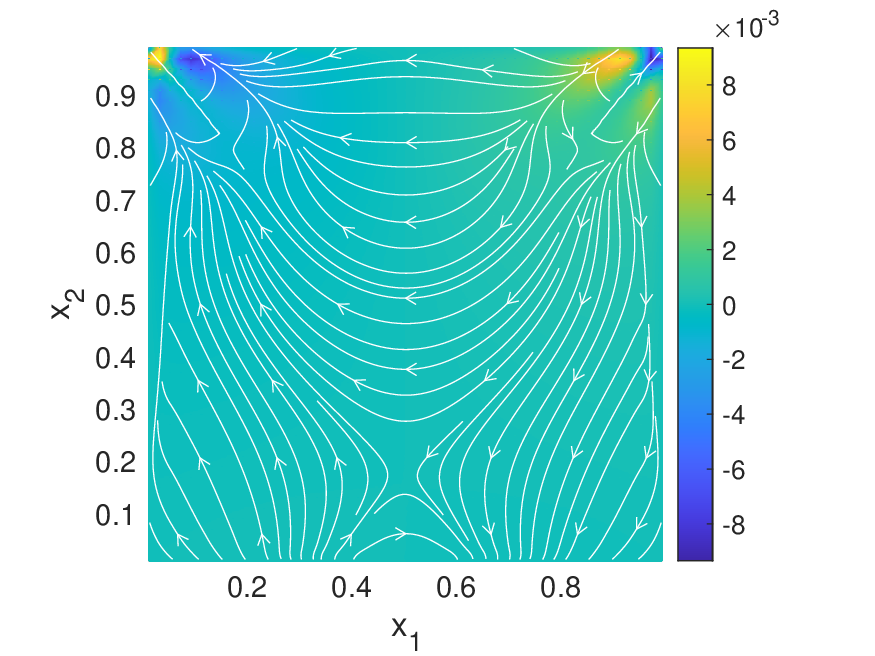}
    }    
    \caption{The distributions of temperature and the streamline of heat flux for lid-driven cavity flow.}
\end{figure}

\begin{table}[!ht] \label{tab:cavityflow_time}
    \caption{Computing time for lid-driven cavity flow (measure time by the second).}
    \footnotesize
    \centering
    \begin{tabular}{cccccccc}
        \hline 
        \hline 
        & & & $\epsilon = 1$ & $\epsilon = 10^{-1}$ & $\epsilon = 10^{-2}$ & $\epsilon = 10^{-3}$ & $\epsilon = 10^{-4}$ \\
        \hline
        \multirow{6}{4em}{Second order} & BSSR & & $836.52$ & $1031.3$ & $2908.0$ & $11322$ & $21591$ \\
& BSSR-MM & & -- & -- & $132.47$ & $297.18$ & $1354.6$ \\
& Hybrid BSSR-MM-$1$ & & -- &  -- & $141.85$ & $281.39$ & $1349.2$ \\
& BSSR-MS & & -- & -- & -- & $294.17$ & $1370.5$ \\
& Hybrid BSSR-MS-$N_b$ && -- & $306.53$ & $129.35$ & $284.82$ & $1356.3$ \\
& & & & ($N_b=5$) & ($N_b=1$) & ($N_b=1$) & ($N_b=1$) \\
        \hline
        \hline 
    \end{tabular}
\end{table}

%%================Conclusion========================

\section{Conclusion} \label{sec:conclu}
We have studied a family of iterative methods for solving steady-state linearized moment equations. Generally speaking, the BSGS method works well for highly rarefied gases, but the convergence slows down as the Knudsen number decreases. On the contrary, the BSGS-MM method has better efficiency in the dense regime. Better performance can be achieved by coupling these two methods. Our numerical tests confirm the theoretical analysis of our methods. However, the experiments also reveal some practical issues:
\begin{itemize}
\item Although the Fourier analysis confirms the convergence of both BSGS and BSGS-MM methods for a wide range of $\epsilon$, when the wall boundary conditions are imposed, the BSGS-MM method may still diverge for large $\epsilon$, especially when the method is generalized to second-order schemes. This requires more careful analysis of the iterative method on bounded domains, which is a part of our ongoing work.
\item In methods with micro-macro or multiscale decomposition, the inner iteration may take a significant amount of time despite a very small number of outer iterations. Usually the macroscopic part is essentially Euler or Navier-Stokes equations. One can replace the simple BSGS solver used in this work with more specific solvers to gain improved efficiency.
\end{itemize}
The general idea of this family of approaches is also applicable to nonlinear moment methods. The moment method has been applied to the time-dependent Boltzmann equation with the quadratic collision term in \cite{Hu2020burnett}, and the combination of the symmetric Gauss-Seidel method and the multigrid technique has been tested for steady-state moment equations of BGK-type collisions, where the ansatz for the distribution function is nonlinear \cite{Hu2019}. It is then relatively straightforward to add the micro-macro decomposition or the multiscale decomposition, which only requires to split the moment equations into several parts according to the magnitudes of moments. The micro-macro decomposition has been applied in many previous works such as \cite{Koellermeier2023}), and the magnitudes of moments have been studied in \cite{Cai2012}. These works have established solid foundation for us to develop efficient iterative algorithms for nonlinear moment equations in our future work.

\bibliographystyle{siamplain}
\bibliography{M163424.bib}

\begin{thebibliography}{10}

\bibitem{Abdelmalik2016}
{\sc M.~Abdelmalik and E.~van Brummelen}, {\em An entropy stable discontinuous
  {G}alerkin finite-element moment method for the {B}oltzmann equation},
  Comput. Math. Appl., 72 (2016), pp.~1988--1999.

\bibitem{Adams2002}
{\sc M.~Adams and E.~Larsen}, {\em Fast iterative methods for
  discrete-ordinates particle transport calculations}, Progress in Nuclear
  Energy, 40 (2002), pp.~3--159.

\bibitem{Alterman}
{\sc Z.~Alterman, K.~Frankowski, and C.~L. Pekeris}, {\em Eigenvalues and
  eigenfunctions of the linearized {B}oltzmann collision operator for a
  {M}axwell gas and for a gas of rigid spheres}, The Astrophysical Journal
  Supplement series, 7 (1962), pp.~291--331.

\bibitem{Baliti2020}
{\sc J.~Baliti, M.~Hssikou, and M.~Alaoui}, {\em The 13-moments method for heat
  transfer in gas microflows}, Australian J. Mech. Eng., 18 (2020), pp.~80--93.

\bibitem{Bennoune2008}
{\sc M.~Bennoune, M.~Lemou, and L.~Mieussens}, {\em Uniformly stable numerical
  schemes for the {B}oltzmann equation preserving the compressible
  {N}avier-{S}tokes asymptotics}, J. Comput. Phys., 227 (2008), pp.~3781--3803.

\bibitem{Blanco1997}
{\sc M.~A. Blanco, M.~Fl{\'o}rez, and M.~Bermejo}, {\em Evaluation of the
  rotation matrices in the basis of real spherical harmonics}, Journal of
  Molecular structure: THEOCHEM, 419 (1997), pp.~19--27.

\bibitem{Boccelli2024}
{\sc S.~Boccelli, W.~Kaufmann, T.~E. Magin, and J.~G. McDonald}, {\em Numerical
  simulation of rareﬁed supersonic ﬂows using a fourth-order
  maximum-entropy moment method with interpolative closure}, J. Comput. Phys.,
  497 (2024), p.~112631.

\bibitem{Cai2010}
{\sc Z.~Cai and R.~Li}, {\em Numerical regularized moment method of arbitrary
  order for {B}oltzmann-{BGK} equation}, SIAM J. Sci. Comput., 32 (2010),
  pp.~2875--2907.

\bibitem{Cai2012}
{\sc Z.~Cai, R.~Li, and Y.~Wang}, {\em Numerical regularized moment method for
  high mach number flow}, Comm. Comput. Phys., 11 (2012), pp.~1415--1438.

\bibitem{Cai2020}
{\sc Z.~Cai and M.~Torrilhon}, {\em On the {H}olway-{W}eiss debate: Convergence
  of the {G}rad-moment-expansion in kinetic gas theory}, Phys. Fluids, 31
  (2020), p.~126105.

\bibitem{Cai2023}
{\sc Z.~Cai, M.~Torrilhon, and S.~Yang}, {\em Linear regularized 13-moment
  equations with {O}nsager boundary conditions for general gas molecules},
  2023.
\newblock To appear in SIAM J. Appl. Math.

\bibitem{Cercignani1988}
{\sc C.~Cercignani and C.~Cercignani}, {\em The boltzmann equation}, Springer,
  1988.

\bibitem{Chan1989}
{\sc T.~F. Chan and H.~C. Elman}, {\em Fourier analysis of iterative methods
  for elliptic pr}, SIAM review, 31 (1989), pp.~20--49.

\bibitem{Chapman1990}
{\sc S.~Chapman and T.~G. Cowling}, {\em The mathematical theory of non-uniform
  gases: an account of the kinetic theory of viscosity, thermal conduction and
  diffusion in gases}, Cambridge university press, 1990.

\bibitem{Christhuraj2024}
{\sc E.~Christhuraja and M.~Torrilhon}, {\em Generic moment systems and
  {FEM}-based numerical solutions for the {B}oltzmann equation}, AIP Conf.
  Proc., 2996 (2024), p.~070001.

\bibitem{Claydon2017}
{\sc R.~Claydon, A.~Shrestha, A.~Rana, J.~Sprittles, and D.~Lockerby}, {\em
  Fundamental solutions to the regularised 13-moment equations: efficient
  computation of three-dimensional kinetic effects}, J. Fluid Mech., 833
  (2017), p.~R4, \url{https://doi.org/10.1017/jfm.2017.763}.

\bibitem{Dimarco2018}
{\sc G.~Dimarco, R.~Loub\'{e}re, J.~Narski, and T.~Rey}, {\em An efficient
  numerical method for solving the {B}oltzmann equation in multidimensions}, J.
  Comput. Phys., 353 (2018), pp.~46--81.

\bibitem{Grad1949}
{\sc H.~Grad}, {\em On the kinetic theory of rarefied gases}, Communications on
  pure and applied mathematics, 2 (1949), pp.~331--407.

\bibitem{Gu2009}
{\sc X.~Gu and D.~R. Emerson}, {\em A high-order moment approach for capturing
  non-equilibrium phenomena in the transition regime}, J. Fluid Mech., 636
  (2009), pp.~177--216.

\bibitem{Himanshi2023}
{\sc Himanshi, A.~Rana, and V.~Gupta}, {\em Fundamental solutions of an
  extended hydrodynamic model in two dimensions: Derivation, theory, and
  applications}, Phys. Rev. E, 108 (2023), p.~015306.

\bibitem{Hu2020burnett}
{\sc Z.~Hu and Z.~Cai}, {\em {B}urnett spectral method for high-speed rarefied
  gas flows}, SIAM J. Sci. Comput., 42 (2020), pp.~B1193--B1226.

\bibitem{Hu2019}
{\sc Z.~Hu and G.~Hu}, {\em An efficient steady-state solver for microflows
  with high-order moment model}, J. Comput. Phys., 392 (2019), pp.~462--482.

\bibitem{Jaiswal2019}
{\sc S.~Jaiswal, A.~A. Alexeenko, and J.~Hu}, {\em A discontinuous {G}alerkin
  fast spectral method for the full {B}oltzmann equation with general collision
  kernels}, J. Comput. Phys., 378 (2019), pp.~178--208.

\bibitem{Koellermeier2023}
{\sc J.~Koellermeier and H.~Vandecasteele}, {\em Hierarchical micro-macro
  acceleration for moment models of kinetic equations}, J. Comput. Phys., 488
  (2023), p.~112194.

\bibitem{Kumar1966}
{\sc K.~Kumar}, {\em Polynomial expansions in kinetic theory of gases}, Ann.
  Phys., 37 (1966), pp.~113--141.

\bibitem{LeVeque1988}
{\sc R.~J. LeVeque and L.~N. Trefethen}, {\em Fourier analysis of the sor
  iteration}, IMA journal of numerical analysis, 8 (1988), pp.~273--279.

\bibitem{Levermore1996}
{\sc C.~Levermore}, {\em Moment closure hierarchies for kinetic theories}, J.
  Stat. Phys., 83 (1996), pp.~1021--1065.

\bibitem{Li2023}
{\sc R.~Li and Y.~Yang}, {\em On well-posed boundary conditions for the linear
  non-homogeneous moment equations in half-space}, J. Stat. Phys., 190 (2023),
  p.~185.

\bibitem{Liu2023}
{\sc W.~Liu, Z.~Liu, Z.~Zhang, C.~Teo, and C.~Shu}, {\em {G}rad's distribution
  function for 13 moments-based moment gas kinetic solver for steady and
  unsteady rarefied flows: Discrete and explicit forms}, Comput. Math. Appl.,
  137 (2023), pp.~112--125.

\bibitem{Lockerby2016}
{\sc D.~Lockerby and B.~Collyer}, {\em Fundamental solutions to moment
  equations for the simulation of microscale gas flows}, J. Fluid Mech., 806
  (2016), pp.~413--–436.

\bibitem{Muller1993}
{\sc I.~M{\"u}ller and T.~Ruggeri}, {\em Extended Thermodynamics}, vol.~37 of
  Springer tracts in natural philosophy, Springer-Verlag, New York, 1993.

\bibitem{Rana2013}
{\sc A.~Rana, M.~Torrilhon, and H.~Struchtrup}, {\em A robust numerical method
  for the {R}13 equations of rarefied gas dynamics: Application to lid driven
  cavity}, J. Comput. Phys., 236 (2013), pp.~169--186.

\bibitem{Sarna2020}
{\sc N.~Sarna, J.~Giesselmann, and M.~Torrilhon}, {\em Convergence analysis of
  {G}rad's {H}ermite expansion for linear kinetic equations}, SIAM J. Numer.
  Anal., 58 (2020), pp.~1164--1194.

\bibitem{Sarna2018}
{\sc N.~Sarna and M.~Torrilhon}, {\em Entropy stable {H}ermite approximation of
  the linearised {B}oltzmann equation for inflow and outflow boundaries}, J.
  Comput. Phys., 369 (2018), pp.~16--44.

\bibitem{Sarna2017}
{\sc N.~Sarna and M.~Torrilhon}, {\em On stable wall boundary conditions for
  the {H}ermite discretization of the linearised {B}oltzmann equation}, J.
  Stat. Phys., 170 (2018), pp.~101--126.

\bibitem{Schneider2022}
{\sc F.~Schneider and T.~Leibner}, {\em First-order continuous- and
  discontinuous-{G}alerkin moment models for a linear kinetic equation:
  {R}ealizability-preserving splitting scheme and numerical analysis}, J.
  Comput. Phys., 456 (2022), p.~111040.

\bibitem{Singh2024}
{\sc S.~Singha, H.~Songa, and M.~Torrilhon}, {\em Modal discontinuous
  {G}alerkin simulations for {G}rad's 13 moment equations: Application to
  {R}iemann problem in continuum-rarefied flow regime}, J. Comput. Theor.
  Transport,  (2024).

\bibitem{Struchtrup2002}
{\sc H.~Struchtrup}, {\em Heat transfer in the transition regime: Solution of
  boundary value problems for {G}rad's moment equations via kinetic schemes},
  Phys. Rev. E, 65 (2002), p.~041204.

\bibitem{Struchtrup2004}
{\sc H.~Struchtrup}, {\em Stable transport equations for rarefied gases at high
  orders in the {K}nudsen number}, Phys. Fluids, 16 (2004), pp.~3921--3934.

\bibitem{Struchtrup2003}
{\sc H.~Struchtrup and M.~Torrilhon}, {\em Regularization of {G}rad's 13 moment
  equations: {D}erivation and linear analysis}, Physics of Fluids, 15 (2003),
  pp.~2668--2680.

\bibitem{Struchtrup2013}
{\sc H.~Struchtrup and M.~Torrilhon}, {\em Regularized 13 moment equations for
  hard sphere molecules: Linear bulk equations}, Phys. Fluids, 25 (2013),
  p.~052001.

\bibitem{Su2020Can}
{\sc W.~Su, L.~Zhu, P.~Wang, Y.~Zhang, and L.~Wu}, {\em Can we find
  steady-state solutions to multiscale rarefied gas flows within dozens of
  iterations?}, Journal of Computational Physics, 407 (2020), p.~109245.

\bibitem{Su2020Fast}
{\sc W.~Su, L.~Zhu, and L.~Wu}, {\em Fast convergence and asymptotic preserving
  of the general synthetic iterative scheme}, SIAM Journal on Scientific
  Computing, 42 (2020), pp.~B1517--B1540.

\bibitem{Theisen2021}
{\sc L.~Theisen and M.~Torrilhon}, {\em {FenicsR13}: A tensorial mixed finite
  element solver for the linear {R13} equations using the {FEniCS} computing
  platform}, ACM Trans. Math. Softw., 47 (2021).

\bibitem{Torrilhon2006}
{\sc M.~Torrilhon}, {\em Two dimensional bulk microflow simulations based on
  regularized {G}rad's 13-moment equations}, SIAM Multiscale. Model. Simul., 5
  (2006), pp.~695--728.

\bibitem{Torrilhon2017}
{\sc M.~Torrilhon and N.~Sarna}, {\em Hierarchical {B}oltzmann simulations and
  model error estimation}, J. Comput. Phys., 342 (2017), pp.~66--84.

\bibitem{Woude2024}
{\sc D.~van~der Woude, E.~van Brummelen, E.~Arlemark, and M.~Abdelmalik}, {\em
  A {$\varphi$}-divergence based finite element moment method for the
  polyatomic {ES}-{BGK} {B}oltzmann equation}, J. Comput. Phys., 503 (2024),
  p.~112813.

\bibitem{Zeng2023}
{\sc J.~Zeng, R.~Yuan, Y.~Zhang, Q.~Li, and L.~Wu}, {\em General synthetic
  iterative scheme for polyatomic rarefied gas flows}, Computers \& Fluids, 265
  (2023), p.~105998.

\bibitem{Zhu2021}
{\sc L.~Zhu, X.~Pi, W.~Su, Z.-H. Li, Y.~Zhang, and L.~Wu}, {\em General
  synthetic iterative scheme for nonlinear gas kinetic simulation of
  multi-scale rarefied gas flows}, J. Comput. Phys., 430 (2021), p.~110091.

\end{thebibliography}

\newpage

%%================appendix========================

\appendix

\section{Convergence analysis for the first-order BSGS method} \label{appen:BSGS}

For simplicity, we study the convergence of the first-order BSGS method by assuming that the velocity variable $\bs{v}$ is also one-dimensional (denoted by $v$ below), so that the steady-state Boltzmann equation reads
\begin{equation*}
v \frac{\partial f}{\partial x}(x,v)  = \frac{1}{\epsilon} \mathcal{L}[f](x,v),
\end{equation*}
where the kernel of the linearized collision operator is
\begin{displaymath}
  \operatorname{ker} \mathcal{L} = \operatorname{span} \{ \bs{\Phi}(v) \omega(v) \},
\end{displaymath}
where
\begin{equation} \label{eq:base_phi}
\bs{\Phi}(v) = (1,v,v^2-1), \quad \omega(v) = \frac{1}{\sqrt{2\pi}} \exp \left( -\frac{v^2}{2}\right).
\end{equation}
Furthermore, we will focus on the problem without velocity discretization, so that the upwind scheme becomes
\begin{equation} \label{eq:conti_first}
- v^+ \bar{f} _{j-1} (v)
+ \left( |v| - \frac{1}{\epsilon} \Delta x \mathcal{L} \right) \bar{f}_j (v) 
+ v^- \bar{f} _{j+1} (v)
=0,
\end{equation}
where $v^+ = \max\{v,0\}$, $v^- = \min\{v,0\}$, and $\bar{f}_j(v)$ approximates the average of $f(x,v)$ in the $j$th cell. Such a simplification allows us to compute some integrals exactly and get analytical results for the convergence rate. We expect that the conclusion will also be applicable when $N$ in the moment equations \eqref{eq:moment_eqn} is sufficiently large. Note that by comparing \eqref{eq:conti_first} and \eqref{eq:1D_first}, one can find that $\bs{A}$, $\bs{L}$ and $\bar{\bs{u}}_j$ are the discrete versions of $v$, $\mathcal{L}$ and $\bar{f}_j$, respectively.

The symmetric Gauss-Seidel method can be applied to \eqref{eq:conti_first} in a similar manner to the BSGS method for the moment equations. We use $\bar{f}_j^{(n+1/2)}(v)$ to denote the solution after the left-to-right sweeping and use $\bar{f}_j^{(n+1)}(v)$ to denote the solution after the right-to-left sweeping. Assume the exact solution of \eqref{eq:conti_first} is $f_j^*(v)$. Then the error functions
\begin{displaymath}
h_j^{(n)}(v) := \bar{f}_j^{(n-1/2)}(v) - f_j^*(v), \quad e_j^{(n)}(v) := \bar{f}_j^{(n)}(v) - f_j^*(v)
\end{displaymath}
satisfy the following equations:
\begin{equation} \label{eq:iteration}
\left \{
\begin{aligned}
& {-v}^+ h_{j-1}^{(n+1)}(v)
+ \left( |v| - \frac{1}{\epsilon} \Delta x \mathcal{L} \right) h_j^{(n+1)}(v) 
+ v^- e_{j+1}^{(n)}(v)
=0, \\
& {-v}^+ h_{j-1}^{(n+1)}(v)
+ \left( |v| - \frac{1}{\epsilon} \Delta x \mathcal{L} \right) e_j^{(n+1)}(v)
+ v^- e_{j+1}^{(n+1)}(v)
=0. 
\end{aligned}
\right.
\end{equation}
A rigorous analysis of \eqref{eq:iteration} requires a close look into the boundary conditions, which becomes rather involved especially when we have wall boundary conditions. Below we will take a simpler strategy and apply the Fourier analysis by assuming periodic boundary conditions. Under this assumption, the lower/upper triangular matrices used in the symmetric Gauss-Seidel method will become circulant matrices, so that the corresponding numerical method is not exactly the one used in practice. Reference \cite{Chan1989} carried out comparison between the Fourier analysis and the classical analysis for a class of methods including the Gauss-Seidel method and the SSOR method. It is found that the two analyses provide quite similar results, although the Fourier analysis sometimes predicts a slightly slower convergence rate. Such a phenomenon is further explained in \cite{LeVeque1988}. For our purposes, the Fourier analysis can already reveal sufficient insights of the problem.

Following the Fourier stability analysis, we assume
\begin{equation} \label{eq:fourier_form}
\begin{pmatrix}
h_j^{(n)} \\
e_j^{(n)}
\end{pmatrix}
= \lambda^n \bs{w}(v) \exp (\mathrm{i} \xi j \Delta x).
\end{equation}
Substituting \eqref{eq:fourier_form} into the error equations \eqref{eq:iteration}, one obtains
\begin{equation} \label{eq:eigen_eqn}
\begin{pmatrix}
\lambda \left[ |v| - \frac{1}{\epsilon} \Delta x \mathcal{L} - v^+ e^{-\mathrm{i} \eta} \right] & v^- e^{\mathrm{i} \eta} \\
- \lambda v^+ e^{-\mathrm{i} \eta} & \lambda \left[ |v| - \frac{1}{\epsilon} \Delta x \mathcal{L} + v^- e^{\mathrm{i} \eta} \right]
\end{pmatrix}
\bs{w}(v) = \bs{0},
\end{equation}
where $\eta = \xi \Delta x$.

On account of the obstacles of solving the Boltzmann equation for low Knudsen numbers, we are more concerned about the situations where $\epsilon$ is small.
We therefore carry out asymptotic expansions of $\lambda$ and $\bs{w}(v)$ with respect to $\epsilon$:
\begin{small}
\begin{equation} \label{eq:pertur_expan}
\lambda = \lambda_0 + \epsilon \lambda_1 + \epsilon^2 \lambda_2 + \cdots, \qquad
\bs{w}(v) = \bs{w}_0 (v) + \epsilon \bs{w}_1 (v) + \epsilon^2 \bs{w}_2 (v) + \cdots.
\end{equation}
Replacing $\lambda$ and $\bs{w}(v)$ in the equations \eqref{eq:eigen_eqn} by the expansions \eqref{eq:pertur_expan}, we find that the equations for the $O(\epsilon^{-1})$, $O(1)$ and $O(\epsilon)$ terms are:
\begin{subequations} \label{eq:eps_order}
\begin{align}
\label{eq:O0}
\lambda_0
\begin{pmatrix}
\mathcal{L} & 0 \\
0 & \mathcal{L}
\end{pmatrix}
\bs{w}_0 (v) &= \bs{0}, \\[5pt]
\label{eq:O1}
\begin{pmatrix}
\lambda_0 \left( |v| - e^{-\mathrm{i} \eta} v^+ \right) & e^{\mathrm{i} \eta} v^- \\
-\lambda_0 e^{-\mathrm{i} \eta} v^+ & \lambda_0 \left( |v| + e^{\mathrm{i} \eta} v^- \right)
\end{pmatrix}
\bs{w}_0 (v) &= \Delta x 
\begin{pmatrix}
\mathcal{L} & 0 \\
0 & \mathcal{L}
\end{pmatrix}
\big[ \lambda_0 \bs{w}_1 (v) + \lambda_1 \bs{w}_0 (v) \big], \\[10pt]
\label{eq:O2}
\begin{split}
\begin{pmatrix}
\lambda_0 \left( |v| - e^{-\mathrm{i} \eta} v^+ \right) & e^{\mathrm{i} \eta} v^- \\
- \lambda_0 e^{-\mathrm{i} \eta} v^+ & \lambda_0 \left( |v| + e^{\mathrm{i} \eta} v^- \right)
\end{pmatrix}
\bs{w}_1 (v) \qquad \\
+ \begin{pmatrix}
\lambda_1 \left( |v| - e^{-\mathrm{i} \eta} v^+ \right) & 0 \\
-\lambda_1 e^{-\mathrm{i} \eta} v^+ & \lambda_1 \left( |v| + e^{\mathrm{i} \eta} v^- \right)
\end{pmatrix}
\bs{w}_0 (v) 
&= \Delta x 
\begin{pmatrix}
\mathcal{L} & 0 \\
0 & \mathcal{L}
\end{pmatrix}
\big[ \lambda_0 \bs{w}_2 (v) + \lambda_1 \bs{w}_1 (v) + \lambda_2 \bs{w}_0 (v) \big].
\end{split}
\end{align}
\end{subequations}
\end{small}
The theorem below states properties of $\lambda_0$:

\begin{theorem} \label{thm:lambda}
Let $\eta \in (0,2\pi)$ and $\Delta x > 0$. Assume that $\lambda_0$, $\lambda_1$ and $\bs{w}_0(v) $, $\bs{w}_1(v)$, $\bs{w}_2(v)$ satisfy \eqref{eq:eps_order} for any $v \in \mathbb{R}$, and $\bs{w}_0(v)$ is a nonzero function. 
Then it holds that $|\lambda_0| \leq 1/(5-4 \cos(\eta))$. The maximum value is attained when $\lambda_0 = 1/(5-4\cos(\eta))$.
\end{theorem}
\begin{proof}
If $\lambda_0 = 0$, it clearly satisfies $|\lambda_0| \leq 1/(5-4 \cos(\eta))$. We will thus focus on the case where $\lambda_0$ is nonzero. According to \eqref{eq:O0}, both components of $\bs{w}_0$ must be in the kernel of the linearized collision operator $\mathcal{L}$, and therefore there exist $\bs{\alpha} = (\alpha_0, \alpha_1, \alpha_2)^T$ and $\bs{\beta} = (\beta_0, \beta_1, \beta_2)^T$ satisfying
\begin{equation} \label{eq:w_0}
\bs{w}_0 (v) =
\begin{pmatrix}
\bs{\Phi}(v) & \\
& \bs{\Phi}(v)
\end{pmatrix}
\begin{pmatrix}
\bs{\alpha} \\
\bs{\beta}
\end{pmatrix} \omega(v),
\end{equation}
where $\bs{\Phi}$ is defined in \eqref{eq:base_phi}. Left-multiplying both sides of  \eqref{eq:O1} by 
$
\begin{pmatrix}
\bs{\Phi}^*(v) & \\
& \bs{\Phi}^*(v)
\end{pmatrix}
$
and integrating the result with respect to $v$, we can obtain by the conservation properties of the collision operator that
\begin{equation} \label{eq:int_eps0}
\bs{Q} 
\begin{pmatrix}
\bs{\alpha} \\
\bs{\beta}
\end{pmatrix}
= \bs{0}
\end{equation}
where
\begin{small}
\begin{equation} \label{eq:Q}
\bs{Q} = \frac{1}{\sqrt{2\pi}}
\begin{pmatrix}
\lambda_0 (2 - e^{-\mathrm{i} \eta}) & -\lambda_0 \sqrt{\frac{\pi}{2}} e^{-\mathrm{i} \eta} & \lambda_0 (2 - e^{-\mathrm{i} \eta}) & -e^{\mathrm{i} \eta} &
\sqrt{\frac{\pi}{2}} e^{\mathrm{i} \eta} & -e^{\mathrm{i} \eta} \\
-\lambda_0 \sqrt{\frac{\pi}{2}} e^{-\mathrm{i} \eta} & 2 \lambda_0 (2-e^{-\mathrm{i} \eta}) & 
-\lambda_0 \sqrt{2 \pi} e^{-\mathrm{i} \eta} & \sqrt{\frac{\pi}{2}} e^{\mathrm{i} \eta} & 
-2 e^{\mathrm{i} \eta} & \sqrt{2 \pi} e^{\mathrm{i} \eta} \\
\lambda_0 (2 - e^{-\mathrm{i} \eta}) & -\lambda_0 \sqrt{2 \pi} e^{-\mathrm{i} \eta} & 
5 \lambda_0 (2 - e^{-\mathrm{i} \eta}) & -e^{\mathrm{i} \eta} & 
\sqrt{2 \pi} e^{\mathrm{i} \eta} & - 5 e^{\mathrm{i} \eta} \\
-\lambda_0 e^{-\mathrm{i} \eta} & -\lambda_0 \sqrt{\frac{\pi}{2}} e^{-\mathrm{i} \eta} & 
-\lambda_0 e^{-\mathrm{i} \eta} & \lambda_0 (2 - e^{\mathrm{i} \eta}) & 
\lambda_0 \sqrt{\frac{\pi}{2}} e^{\mathrm{i} \eta} & \lambda_0 (2 - e^{\mathrm{i} \eta}) \\
-\lambda_0 \sqrt{\frac{\pi}{2}} e^{-\mathrm{i} \eta} & -2 \lambda_0 e^{- \mathrm{i} \eta} & 
-\lambda_0 \sqrt{2 \pi} e^{-\mathrm{i} \eta} & \lambda_0 \sqrt{\frac{\pi}{2}} e^{\mathrm{i} \eta} & 
2 \lambda_0 (2 - e^{\mathrm{i} \eta}) & \lambda_0 \sqrt{2 \pi} e^{\mathrm{i} \eta} \\
-\lambda_0 e^{-\mathrm{i} \eta} & -\lambda_0 \sqrt{2 \pi} e^{-\mathrm{i} \eta} & 
-5 \lambda_0 e^{- \mathrm{i} \eta} & \lambda_0 (2 - e^{\mathrm{i} \eta}) &
\lambda_0 \sqrt{2 \pi} e^{\mathrm{i} \eta} & 5 \lambda_0 (2 - e^{\mathrm{i} \eta})
\end{pmatrix}.
\end{equation}
\end{small}
Since $\bs{w}_0$ is nonzero, the vector $\bs{\alpha}$ and $\bs{\beta}$ cannot be both zero vectors. Thus, in \eqref{eq:int_eps0}, the matrix $\bs{Q}$ must be singular, \textit{i.e.} $\det(\bs{Q})=0$. This gives a sextic equation of $\lambda_0$. The equation has four real roots:
\begin{equation*}
\lambda_0^{(1)} = \frac{1}{5-4\cos(\eta)}, \qquad
\lambda_0^{(2,3,4)} = 0.
\end{equation*}
The other two complex roots satisfy
\begin{equation*}
\begin{split}
|\lambda_0^{(5)}| = |\lambda_0^{(6)}| &=
\frac{16-5\pi}{\sqrt{(16-5\pi)^2 + 128(16+15\pi)(1-\cos(\eta)) + 256(16-5\pi)(1-\cos(\eta))^2}} \\
& < \frac{16-5\pi}{\sqrt{(16-5\pi)^2 + 8(16-5\pi)^2 (1 - \cos(\eta)) + 16(16-5\pi)^2(1 - \cos(\eta))^2}} = |\lambda_0^{(1)}|.
\end{split}
\end{equation*}
This confirms that $\lambda_0 \leq 1/(5-4 \cos(\eta))$, and the maximum value is attained if $\lambda_0 = \lambda_0^{(1)}$.
\end{proof}

Note that this estimation of $\lambda_0$ holds for any collision operators. If the collision operator $\mathcal{L}$ is the linearized BGK operator defined by
\begin{equation} \label{eq:opera_L}
\mathcal{L}[f](x,v) = \mathcal{M}[f](x,v) - f(x,v)
\end{equation}
with
\begin{equation*}
\mathcal{M}[f](v) = \left[ \int_{\mathbb{R}} \bs{\Phi}(v) f(v) \,\mathrm{d}v \right] \begin{pmatrix} 1 \\ v \\ (v^2-1) / 2\end{pmatrix} \omega(v),
\end{equation*}
then $\lambda_1$ can also be computed:

\begin{theorem}
\label{thm:lambda1}
Let $\eta \in (0,2\pi)$ and $\Delta x > 0$. Assume that $\lambda_0$, $\lambda_1$ and $\bs{w}_0(v) $, $\bs{w}_1(v)$, $\bs{w}_2(v)$ satisfy \eqref{eq:eps_order} with $\mathcal{L}$ defined by \eqref{eq:opera_L}, and $\bs{w}_0(v)$ is a nonzero function. 
If $\lambda_0 = 1/(5-4 \cos(\eta))$, then
\begin{equation} \label{eq:lambda_1}
\lambda_1 = - \frac{36 \sqrt{2 \pi} [1 - \cos (\eta)]}{5 \Delta x [5 - 4 \cos (\eta)]^2} < 0.
\end{equation}
\end{theorem}
\begin{proof}
When $\lambda_0 = 1/(5-4\cos(\eta))$, the function $\bs{w}_0(v)$ given by \eqref{eq:w_0} can be obtained by solving \eqref{eq:int_eps0}. The result is
\begin{equation} \label{eq:w0}
\bs{w}_0(v) = \begin{pmatrix}
e^{\mathrm{i} \eta} \left( 2-e^{\mathrm{i} \eta} \right) \\
1
\end{pmatrix} \left( v^2-3 \right) \omega(v).
\end{equation}
To obtain $\lambda_1$, we take moments of \eqref{eq:O2} to get
\begin{equation} \label{eq:int_eps10}
\begin{split}
&
\int_{-\infty}^{\infty}
\begin{pmatrix}
\bs{\Phi}^*(v) & \\
& \bs{\Phi}^*(v)
\end{pmatrix}
\left[
\begin{pmatrix}
\lambda_0 \left( |v| - e^{-\mathrm{i} \eta} v^+ \right) & e^{\mathrm{i} \eta} v^- \\
-\lambda_0 e^{-\mathrm{i} \eta} v^+ & \lambda_0 \left( |v| + e^{\mathrm{i} \eta} v^- \right)
\end{pmatrix}
\bs{w}_1 (v) \right. \\
& \hspace{12em} \left. +
\begin{pmatrix}
\lambda_1 \left( |v| - e^{-\mathrm{i} \eta} v^+ \right) & 0 \\
-\lambda_1 e^{-\mathrm{i} \eta} v^+ & \lambda_1 \left( |v| + e^{\mathrm{i} \eta} v^- \right)
\end{pmatrix}
\bs{w}_0 (v) 
\right] \omega(v) \mathrm{d}v
= \bs{0}.
\end{split}
\end{equation}
Since $\mathcal{L}$ is the BGK collision operator defined by \eqref{eq:opera_L}, the first-order term $\bs{w}_1(v)$ can be decomposed into $\bs{w}_1(v) = \bs{w}_{10}(v) + \bs{w}_{11}(v)$ where
\begin{equation} \label{eq:w11}
\bs{w}_{10} =
\begin{pmatrix}
\bs{\Phi}(v) & \\
& \bs{\Phi}(v)
\end{pmatrix}
\begin{pmatrix}
\bs{\gamma} \\
\bs{\tau}
\end{pmatrix} \omega(v), \quad
\bs{w}_{11}(v) = -\begin{pmatrix}
\mathcal{L} & 0 \\
0 & \mathcal{L}
\end{pmatrix}
\bs{w}_1 (v),
\end{equation}
where $\bs{\gamma} = (\gamma_0, \gamma_1, \gamma_2)^T$, $\bs{\tau} = (\tau_0, \tau_1, \tau_2)^T$ are undetermined. Plugging \eqref{eq:w11} into \eqref{eq:O1} yields
\begin{equation*}
\begin{pmatrix}
\lambda_0 \left( |v| - e^{-\mathrm{i} \eta} v^+ \right) & e^{\mathrm{i} \eta} v^- \\
-\lambda_0 e^{-\mathrm{i} \eta} v^+ & \lambda_0 \left( |v| + e^{\mathrm{i} \eta} v^- \right)
\end{pmatrix}
\bs{w}_0 (v) = - \Delta x \lambda_0 \bs{w}_{11} (v),
\end{equation*}
which allows us to express $\bs{w}_{11}(v)$ using $\bs{w}_0(v)$, so that $\bs{w}_1(v)$ can be expressed using $\bs{w}_{10}(v)$ and $\bs{w}_0(v)$:
\begin{displaymath}
\bs{w}_1(v) = \begin{pmatrix}
\bs{\Phi(v)} & \\
& \bs{\Phi(v)}
\end{pmatrix}
\begin{pmatrix}
\bs{\gamma} \\
\bs{\tau}
\end{pmatrix} \omega(v)
- \frac{1}{\lambda_0 \Delta x}
\begin{pmatrix}
\lambda_0 \left( |v| - e^{-\mathrm{i} \eta} v^+ \right) & e^{\mathrm{i} \eta} v^- \\
-\lambda_0 e^{-\mathrm{i} \eta} v^+ & \lambda_0 \left( |v| + e^{\mathrm{i} \eta} v^- \right)
\end{pmatrix} \bs{w}_0(v),
\end{displaymath}
so that \eqref{eq:int_eps10} can be rewritten as 
\begin{small}
\begin{equation} \label{eq:int_eps1}
\begin{split}
\bs{Q} \begin{pmatrix} \bs{\gamma} \\ \bs{\tau} \end{pmatrix} =
\int_{-\infty}^{\infty}
\begin{pmatrix}
\bs{\Phi^*(v)} & \\
& \bs{\Phi^*(v)}
\end{pmatrix}
\left[ \frac{1}{\Delta x}
\begin{pmatrix}
\lambda_0 \big[ |v|^2 - (2 e^{-\mathrm{i} \eta} - e^{-2 \mathrm{i} \eta}) (v^+)^2 \big] & -(2 e^{\mathrm{i} \eta} - e^{2 \mathrm{i} \eta}) (v^-)^2 \\
- \lambda_0 (2 e^{- \mathrm{i} \eta} - e^{- 2 \mathrm{i} \eta}) (v^+)^2 & \lambda_0 \big[ |v|^2 - (2 e^{\mathrm{i} \eta} - e^{2 \mathrm{i} \eta}) (v^-)^2 \big]
\end{pmatrix}
\right. \\
\left. -
\begin{pmatrix}
\lambda_1 \left( |v| - e^{-\mathrm{i} \eta} v^+ \right) & 0 \\
-\lambda_1 e^{-\mathrm{i} \eta} v^+ & \lambda_1 \left( |v| + e^{\mathrm{i} \eta} v^- \right)
\end{pmatrix}
\right] \bs{w}_0 (v) \,\mathrm{d}v,
\end{split}
\end{equation}
\end{small}
where the matrix $\bs{Q}$ has been defined in \eqref{eq:Q}. Since $\bs{Q}$ is singular, we can find its left eigenvector $\bs{\ell} = \big(2,0,-1,-2(1-2e^{\mathrm{i}\eta}), 0, 1-2e^{\mathrm{i}\eta}\big)$ satisfying $\bs{\ell} \bs{Q} = 0$. Left-multiplying both sides of \eqref{eq:int_eps1} by $\bs{\ell}$ turns the left-hand side to zero, and the right-hand side can be directly integrated by \eqref{eq:w0}. The result is
\begin{equation*} 
 \frac{5 \lambda_1 \Delta x (e^{-\mathrm{i} \eta}-2)^2 (e^{\mathrm{i} \eta}-2)^2 + 18 \sqrt{2 \pi} (e^{-\mathrm{i} \eta}-1) (e^{\mathrm{i} \eta}-1)}{2 \pi e^{-\mathrm{i} \eta} (e^{-\mathrm{i} \eta}-2) (e^{\mathrm{i} \eta}-2)} = 0,
\end{equation*}
from which $\lambda_1$ can be solved and the result is \eqref{eq:lambda_1}.
\end{proof}

These two theorems indicate that the amplification factor of the BSGS iteration is $1/[5-4 \cos (\Delta x)] + O(\epsilon)$, attained when $\xi = 1$ or $N-1$. Compared with the conventional iterative scheme, this factor does not tend to $1$ as $\epsilon$ approaches zero. Note that this factor does tend to $1$ as $\Delta x$ approaches zero. This is the typical behavior of the classical iterative methods and can usually be fixed by applying the multigrid method. This has not implemented and will be explored in our future work.

In the case of BGK collisions, we do see a slowdown of the convergence when $\epsilon$ gets smaller, since $\lambda_1$ can be bounded by
\begin{displaymath}
\lambda_1 \leqslant -\frac{36\sqrt{2\pi}(1-\cos(\Delta x)}{5\Delta x[5-4\cos(\Delta x)]^2} \approx -\frac{18}{5}\sqrt{2\pi}\Delta x.
\end{displaymath}

\section{Convergence analysis for the BSSR method} \label{appen:BSSR}

If we carry out the same convergence analysis as Appendix \ref{appen:BSGS} to BSSR method, based on the Fourier stability analysis, we can obtain the equations
\begin{equation*} 
\begin{pmatrix}
\lambda \left[ \frac{1}{4} v^+ e^{-2 \mathrm{i} \eta} - \left( \frac{1}{4} v + v^+ \right) e^{- \mathrm{i} \eta} + \left(\frac{3}{4} + \alpha \right) |v| - \frac{1}{\epsilon} \Delta x \mathcal{L} \right] 
& -\alpha |v| + \left( \frac{1}{4} v + v^- \right) e^{\mathrm{i} \eta} - \frac{1}{4} v^- e^{2 \mathrm{i} \eta} \\
\lambda \left[ \frac{1}{4} v^+ e^{-2 \mathrm{i} \eta} - \left( \frac{1}{4} v + v^+ \right) e^{- \mathrm{i} \eta} - \alpha |v| \right] 
& \lambda \left[ \left(\frac{3}{4} + \alpha \right) |v| +  \left( \frac{1}{4} v + v^+ \right) e^{- \mathrm{i} \eta} - \frac{1}{\epsilon} \Delta x \mathcal{L} \right]
\end{pmatrix}
\bs{w}(v) = \bs{0},
\end{equation*}
like \eqref{eq:eigen_eqn}. By assuming the asymptotic expansion \eqref{eq:pertur_expan}, one can find the equations of order $O(\epsilon^{-1})$ to be the same as \eqref{eq:O0}, so that \eqref{eq:w_0} still holds, while the $O(1)$ equations are
\begin{align*}
&
\begin{pmatrix}
\lambda_0 \left[ \frac{1}{4} v^+ e^{-2 \mathrm{i} \eta} - \left( \frac{1}{4} v + v^+ \right) e^{- \mathrm{i} \eta} + \left(\frac{3}{4} + \alpha \right) |v| \right] 
& -\alpha |v| + \left( \frac{1}{4} v + v^- \right) e^{\mathrm{i} \eta} - \frac{1}{4} v^- e^{2 \mathrm{i} \eta} \\
\lambda \left[ \frac{1}{4} v^+ e^{-2 \mathrm{i} \eta} - \left( \frac{1}{4} v + v^+ \right) e^{- \mathrm{i} \eta} - \alpha |v| \right] 
& \lambda \left[ \left(\frac{3}{4} + \alpha \right) |v| +  \left( \frac{1}{4} v + v^+ \right) e^{- \mathrm{i} \eta} \right]
\end{pmatrix} 
\bs{w}_0(v) \\
& \hspace{25em} = \Delta x 
\begin{pmatrix}
\mathcal{L} & 0 \\
0 & \mathcal{L}
\end{pmatrix}
\big[ \lambda_0 \bs{w}_1 (v) + \lambda_1 \bs{w}_0 (v) \big].
\end{align*}
By taking the moments of these equations, one can still get a homogeneous linear system of $\boldsymbol{\alpha}$ and $\boldsymbol{\beta}$, and the values of $\lambda_0$ can be obtained by setting the determinant of the coefficient matrix to be zero. This again results in a sextic equation of $\lambda_0$. The modulus of the root with largest magnitude, which depends on both $\alpha$ and $\eta$, is plotted in Figure \ref{fig:BSSR_lambda}, from which one can observe that  $|\lambda_0| \leqslant 1$ in all cases, showing that the convergence rate has a lower bound independent of when $\epsilon$ approaches zero.
\begin{figure}[!ht] \label{fig:BSSR_lambda}
    \centering
      \includegraphics[width=0.45\textwidth, trim=40 15 55 35, clip]{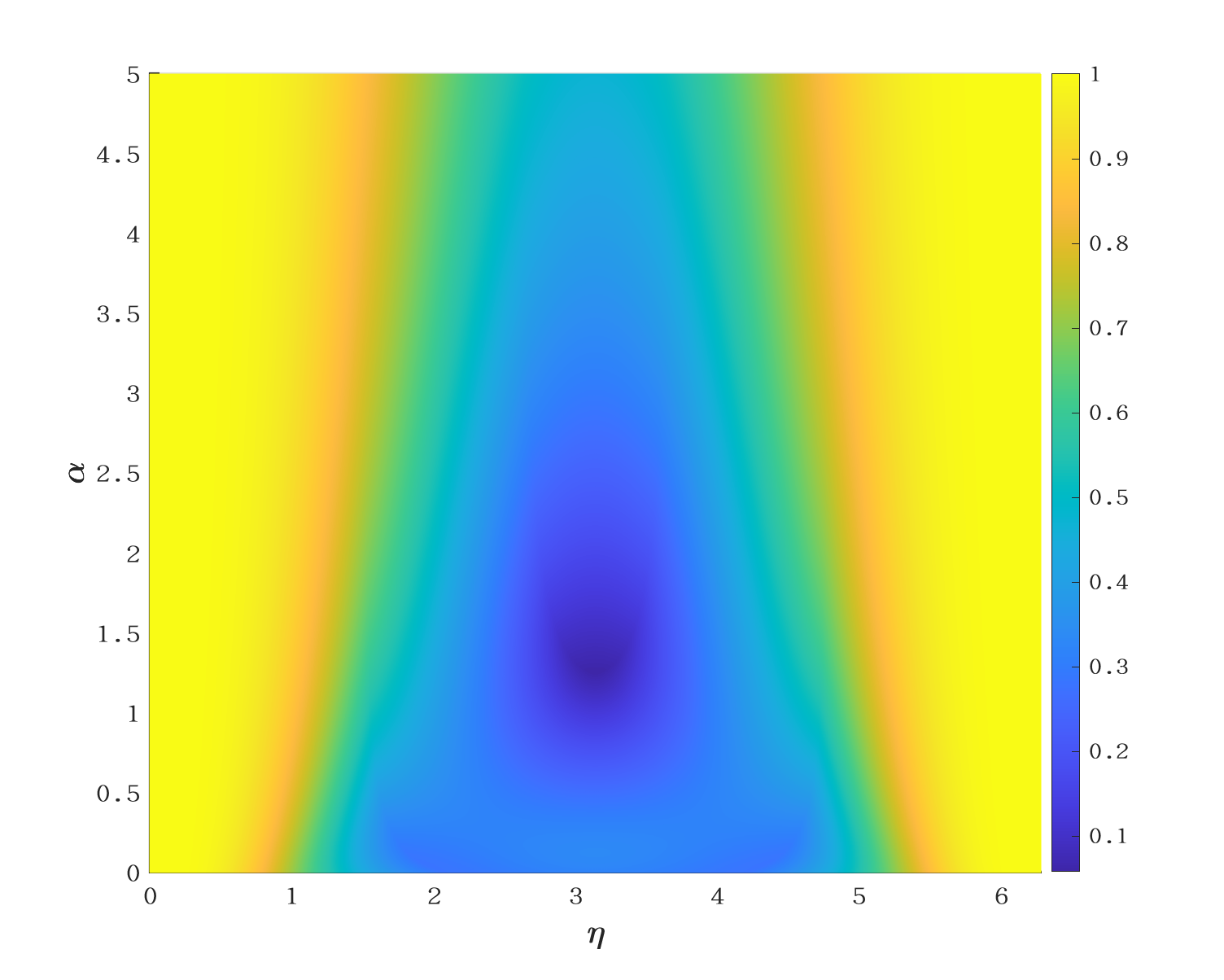}
    \caption{The distribution of the maximum value of $|\lambda_0|$.}
\end{figure}

\section{Convergence analysis for the BSGS-MM method} \label{appen:BSGS_MM}

The convergence of the BSGS-MM method is analyzed in a similar manner to that of the BSGS method in \ref{appen:BSGS}. Here we need to split the distribution function into two parts representing $\bs{u}^1$ and $\bs{u}^2$:
\begin{displaymath}
f(x,v) = f^1(x,v) + f^2(x,v),
\end{displaymath}
with $f^1 = \mathcal{P}f$ and $f^2 = (\mathcal{I} - \mathcal{P})f$, where $\mathcal{I}$ is the identity operator, and the projection operator $\mathcal{P}$ is defined by
\begin{displaymath}
\mathcal{P}f(x,v) = \sum_{n=0}^{N_0} \left( \int_{\mathbb{R}} |\varphi_n(v)|^2 \omega(v) \,\mathrm{d}v \right)^{-1} \left( \int_{\mathbb{R}} \varphi_n(v) f(x,v) \,\mathrm{d}v \right) \varphi_n(v).
\end{displaymath}
The collision operator $\mathcal{L}$ can be split similarly into the following two operators:
\begin{displaymath}
\mathcal{L}_1 = \mathcal{P} \mathcal{L}, \qquad \mathcal{L}_2 = (\mathcal{I} - \mathcal{P}) \mathcal{L}.
\end{displaymath}
Additionally, we define the operators $\mathcal{V}_{1,2}^\pm$ and $|\mathcal{V}_{1,2}|$ by
\begin{displaymath}
\mathcal{V}_1^{\pm} f = \mathcal{P} (v^{\pm} f), \qquad \mathcal{V}_2^{\pm} f = (\mathcal{I} - \mathcal{P}) (v^{\pm} f), \qquad |\mathcal{V}_1| = \mathcal{V}_1^+ - \mathcal{V}_1^-, \qquad |\mathcal{V}_2| = \mathcal{V}_2^+ - \mathcal{V}_2^-.
\end{displaymath}
Thus, the BSGS-MM method without discretization of $v$ can be written as
\begin{equation} \label{eq:BSGS_MM_iter} 
\left \{
\begin{aligned}
& - \mathcal{V}_1^+ \bar{f}^{1,(n+1)}_{j-1} (v) + \left( |\mathcal{V}_1| - \frac{1}{\epsilon} \Delta x \mathcal{L}_1 \right) \bar{f}^{1,(n+1)}_j(v) + \mathcal{V}_1^- \bar{f}^{1,(n+1)}_{j+1}(v) \\
& \hspace{16em} = \mathcal{V}_1^+ \bar{f}^{2,(n)}_{j-1}(v) - |\mathcal{V}_1| \bar{f}^{2,(n)}_j(v) - \mathcal{V}_1^- \bar{f}^{2,(n)}_{j+1}(v), \\
& - \mathcal{V}_2^+ \bar{f}^{1,(n+1)}_{j-1} (v) + |\mathcal{V}_2| \bar{f}^{1,(n+1)}_j(v) + \mathcal{V}_2^- \bar{f}^{1,(n+1)}_{j+1}(v) \\
& \hspace{10em} = \mathcal{V}_2^+ \bar{f}^{2,(n+1/2)}_{j-1}(v) - \left( |\mathcal{V}_2| - \frac{1}{\epsilon} \Delta x \mathcal{L}_2 \right) \bar{f}^{2,(n+1/2)}_j(v) - \mathcal{V}_2^- \bar{f}^{2,(n)}_{j+1}(v), \\
& - \mathcal{V}_2^+ \bar{f}^{1,(n+1)}_{j-1} (v) + |\mathcal{V}_2| \bar{f}^{1,(n+1)}_j(v) + \mathcal{V}_2^- \bar{f}^{1,(n+1)}_{j+1}(v) \\
& \hspace{10em} = \mathcal{V}_2^+ \bar{f}^{2,(n+1/2)}_{j-1}(v) - \left( |\mathcal{V}_2| - \frac{1}{\epsilon} \Delta x \mathcal{L}_2 \right) \bar{f}^{2,(n+1)}_j(v) - \mathcal{V}_2^- \bar{f}^{2,(n+1)}_{j+1}(v).
\end{aligned}
\right.
\end{equation}

Mimicking the analysis of the BSGS method, we let $f_j^*(v)$ be the exact solution of \eqref{eq:conti_first} and define the error functions
\begin{displaymath}
e_j^{1,(n)}(v) = \bar{f}_j^{1,(n)}(v) - \mathcal{P} f_j^*(v), \quad h_j^{2,(n)} = \bar{f}_j^{2,(n-1/2)} - (\mathcal{I} - \mathcal{P}) f_j^*, \quad e_j^{2,(n)}(v) = \bar{f}_j^{2,(n)}(v) - (\mathcal{I} - \mathcal{P}) f_j^*(v).
\end{displaymath}
Then the evolution of the error functions can be immediately obtained by replacing the solutions in \eqref{eq:BSGS_MM_iter} with the corresponding error functions. Note that the solution of $\bar{f}_j^1(v)$ at the current time step (the $n$th step) never appears in \eqref{eq:BSGS_MM_iter}. We can regard $\bar{f}_j^{1,(n)}$ as an auxiliary variable and consider \eqref{eq:BSGS_MM_iter} as an iterative scheme for $\bar{f}_j^2(v)$.
We now apply the Fourier analysis by assuming that
\begin{equation*}
\begin{pmatrix}
h_j^{2,(n)}(v) \\
e_j^{2,(n)}(v)
\end{pmatrix}
= \lambda^n 
\begin{pmatrix}
\eta^2(v) \\ \varepsilon^2(v)
\end{pmatrix}
\exp (\mathrm{i} \xi j \Delta x),
\end{equation*}
and $\eta^2(v)$ and $\varepsilon^2(v)$ cannot be both zero.
It is not difficult to see that $e_j^{1,(n+1)}$ also has the form $e_j^{1,(n+1)}(v) = \lambda^n \varepsilon^1(v) \exp(\mathrm{i} \xi j \Delta x)$, and the functions $\varepsilon^1(v)$, $\eta^2(v)$ and $\varepsilon^2(v)$ satisfy
\begin{equation} \label{eq:BSGS_MM_lambda}
\begin{pmatrix}
- \mathcal{V}_1^+ e^{- \mathrm{i} \eta} + |\mathcal{V}_1| + \mathcal{V}_1^- e^{\mathrm{i} \eta} - \frac{\Delta x}{\epsilon} \mathcal{L}_1 & 0 & - \mathcal{V}_1^+ e^{- \mathrm{i} \eta} + |\mathcal{V}_1| + \mathcal{V}_1^- e^{\mathrm{i} \eta} \\
- \mathcal{V}_2^+ e^{- \mathrm{i} \eta} + |\mathcal{V}_2| + \mathcal{V}_2^- e^{\mathrm{i} \eta} & - \lambda \left( \mathcal{V}_2^+ e^{- \mathrm{i} \eta} - |\mathcal{V}_2| + \frac{\Delta x}{\epsilon} \mathcal{L}_2 \right) & \mathcal{V}_2^- e^{\mathrm{i} \eta} \\
- \mathcal{V}_2^+ e^{- \mathrm{i} \eta} + |\mathcal{V}_2| + \mathcal{V}_2^- e^{\mathrm{i} \eta} & - \lambda \mathcal{V}_2^+ e^{- \mathrm{i} \eta} & \lambda \left( |\mathcal{V}_2| + \mathcal{V}_2^- e^{\mathrm{i} \eta} - \frac{\Delta x}{\epsilon} \mathcal{L}_2 \right)
\end{pmatrix}
\begin{pmatrix}
\varepsilon^1(v) \\ \eta^2(v) \\ \varepsilon^2(v)
\end{pmatrix}
 = \bs{0},
\end{equation}
where we have again used $\eta = \xi \Delta x$ for conciseness. When $\epsilon$ is small, the leading-order term is
\begin{displaymath}
    \mathcal{L}_1 \varepsilon_0^1(v) = 0, \quad
    \lambda_0 \mathcal{L}_2 \eta_0^2(v) = 0, \quad
    \lambda_0 \mathcal{L}_2 \varepsilon_0^2(v) = 0.
\end{displaymath}
Since the linear collision operator is negative semidefinite with the null space $\operatorname{ker} \mathcal{L} = \operatorname{span}\{\bs{\Phi}\omega\} \subset \operatorname{Ran} \mathcal{P}$, the operator $\mathcal{L}_2$ is negative definite on $\operatorname{Ran}(\mathcal{I}-\mathcal{P})$, so that $\lambda_0=0$, indicating fast convergence for small values of $\epsilon$.

However, when $\epsilon$ is large, the convergence of the BSGS-MM method may be slow. To demonstrate this, we solve \eqref{eq:BSGS_MM_lambda} numerically for difference values of $\epsilon$ and $\eta$. The results for the linearized BGK collision operator are plotted in Figure \ref{fig:max_lambda}. It is clear that when $\epsilon$ tends to zero, the spectral radius also tends to zero since $\bs{u}^2$ is nearly zero, and $\bs{u}^1$ is solved exactly during the iteration. 

\begin{figure}[!ht] \label{fig:max_lambda}
    \centering
    \subfigure[The maximum value of $|\lambda|$]{
    \includegraphics[width=0.4\textwidth, trim=50 10 65 45, clip]{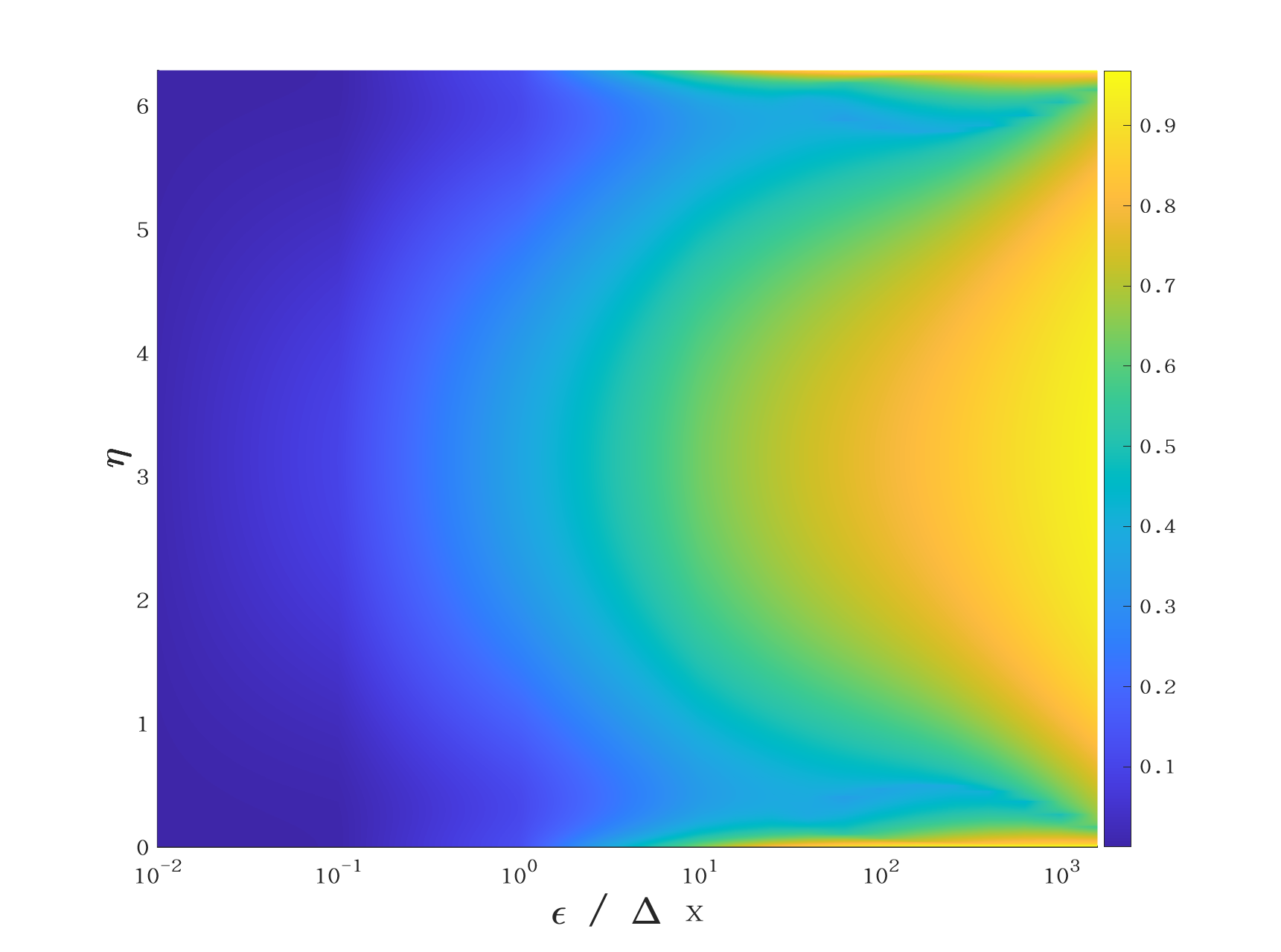}
    } \quad
    \subfigure[$\max_{\eta}|\lambda|$]{
    \includegraphics[width=0.39\textwidth, trim=5 2 35 20, clip]{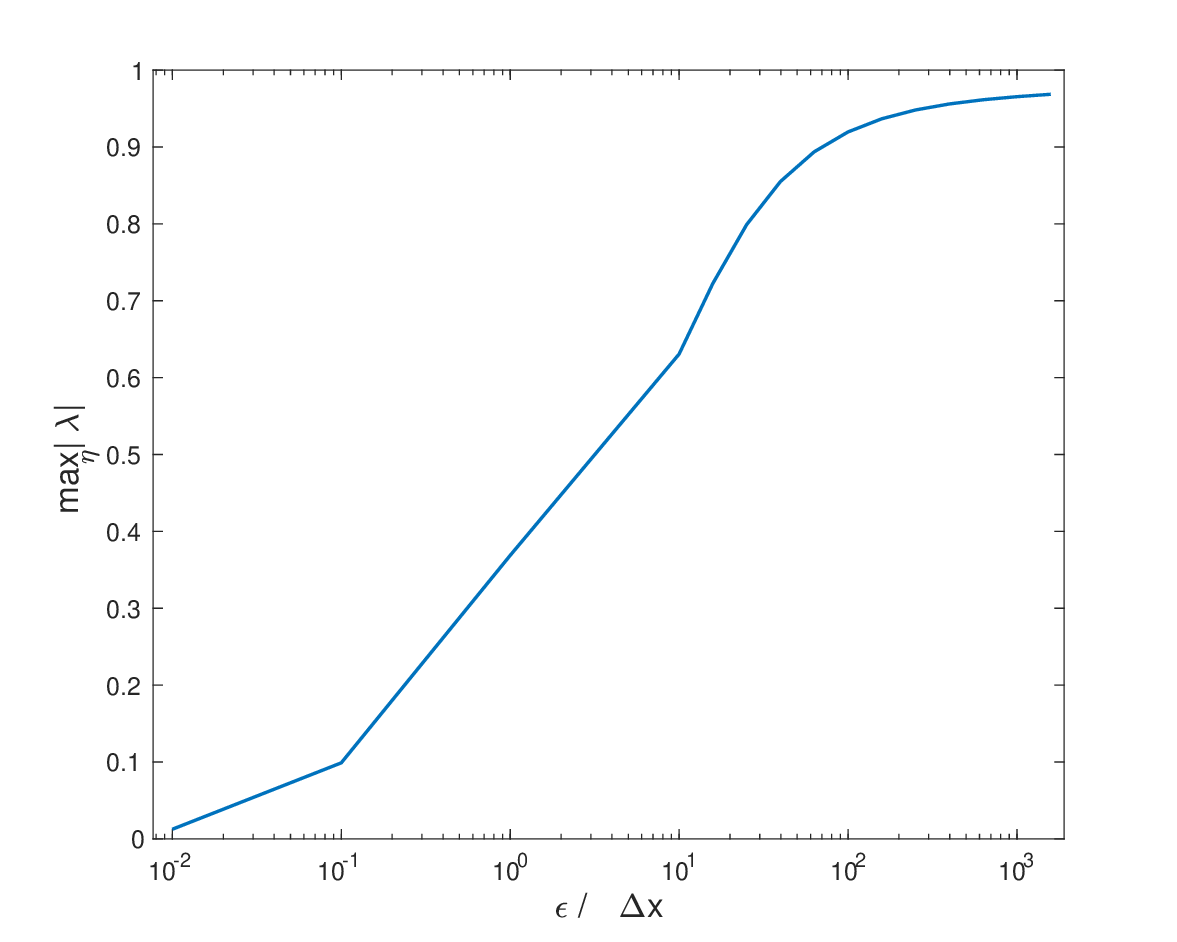}
    }
    \caption{The distribution of the maximum value of $|\lambda|$.}
\end{figure}

\end{document}